\documentclass{article}

\usepackage{textcomp}
\usepackage[utf8]{inputenc}
\usepackage[T1]{fontenc}

\usepackage[numbers,sort&compress]{natbib}

\usepackage{hyperref}
\usepackage[dvipsnames]{xcolor}
\hypersetup{
    colorlinks,
    linkcolor={red!60!black},
    citecolor={green!40!black}
}

\usepackage{amssymb}
\usepackage{amsmath}
\usepackage{amsthm}
\usepackage{dsfont}

\newtheorem{lemma}{Lemma}
\newtheorem{corollary}{Corollary}

\usepackage{booktabs}
\usepackage{siunitx}
\usepackage{multirow}

\usepackage{caption}
\usepackage{subcaption}
\usepackage[capitalize,noabbrev,nameinlink]{cleveref}
\usepackage{orcidlink}

\usepackage{chngcntr}
\usepackage{apptools}
\usepackage{mathtools}
\usepackage{algorithm,algorithmic}
\usepackage{stmaryrd}
\usepackage{booktabs}
\usepackage{color}
\usepackage{nicefrac}
\usepackage{multirow}
\usepackage{lipsum}

\usepackage{thmtools}
\usepackage{thm-restate}

\DeclareMathOperator*{\argmax}{\mathrm{arg\,max}}

\DeclareMathOperator{\Tr}{Tr}
\DeclareMathOperator{\rank}{rank}
\DeclareMathOperator{\diag}{diag}
\DeclareMathOperator{\Diag}{Diag}

\usepackage{bm}

\newcommand{\innp}[2]{\left\langle #1, #2 \right\rangle}
\newcommand{\norm}[1]{\left\| #1 \right\|}
\newcommand{\vx}{\mathbf{x}}
\newcommand{\vvv}{\mathbf{v}}

\newcommand{\vy}{\mathbf{y}}

\newcommand{\vl}{\mathbf{l}}
\newcommand{\vu}{\mathbf{u}}
\newcommand{\vz}{\mathbf{z}}
\newcommand{\vw}{\mathbf{w}}
\newcommand{\vb}{\mathbf{b}}
\newcommand{\va}{\mathbf{a}}
\newcommand{\vm}{\mathbf{m}}
\newcommand{\valpha}{\bm{\alpha}}
\newcommand{\vbeta}{\bm{\beta}}
\newcommand{\R}{ {\mathbb R} } % defines a short cut for the symbol of real numbers
\newcommand{\Z}{ {\mathbb Z} } % define symbol integers
\newcommand{\N}{ {\mathbb N} } % define symbol naturals

\newcommand{\SymMat}{\mathcal{S}}

\usepackage{xspace}
\newcommand{\Xset}{{\ensuremath{\mathcal{X}}}\xspace}

\begin{document}

\begin{center}
 {\Large Discrete eigenvalue optimization from entropic smoothing\\ and first-order methods}
\end{center}

\vspace{7mm}

\noindent\textbf{Deborah Hendrych}\orcidlink{0000-0003-0705-1356}\hfill\href{mailto:hendrych@zib.de}{\ttfamily hendrych@zib.de}\\
\emph{\small Interactive Optimization \& Learning\\
Zuse Institute Berlin \& Technische Universit\"at Berlin \\
Berlin, Germany}\\
\\
\textbf{Mathieu Besançon}\orcidlink{0000-0002-6284-3033}\hfill\href{mailto:mathieu.besancon@inria.fr}{\ttfamily mathieu.besancon@inria.fr}\\
\emph{\small Univ.~Grenoble Alpes, Inria, Laboratoire d'Informatique de Grenoble \& Zuse Institute Berlin \\
Grenoble, France}\\
\\
\textbf{Sebastian Pokutta}\orcidlink{0000-0001-7365-3000} \hfill\href{mailto:pokutta@zib.de}{\ttfamily pokutta@zib.de}\\
\emph{\small Interactive Optimization \& Learning \\
Zuse Institute Berlin \& Technische Universit\"at Berlin \\
Berlin, Germany}\\
\\

\vspace{5mm}

\begin{center}
\begin{minipage}{0.85\textwidth}
\begin{center}
 \textbf{Abstract}
\end{center}
We study the maximization of the minimum eigenvalue under combinatorial and integrality constraints. We propose a new approach based on branch-and-bound combines entropic smoothing of the minimum eigenvalue function and
Frank-Wolfe methods over concave relaxations of the constraints, thereby exploiting combinatorial structure through linear optimization oracles. 
We establish approximation and convergence guarantees, including for truncated gradients computed from partial eigendecompositions, and introduce rank- and eigenvalue-based pruning and duality-based variable fixing. 
We evaluate the method on E-optimal experimental design and maximum algebraic connectivity problems and compare it with SCIP-SDP. 
The results show that our approach is particularly effective for large-dimensional instances and problems with additional combinatorial structure, 
whereas SCIP-SDP performs better on moderately sized instances with simpler constraints.

\end{minipage}
\end{center}

\vspace{0mm}

\section{Introduction} \label{sec:introduction}

We are interested in in solving problems of the following structure:
\begin{align}
    \max \; &\lambda_{\min}(M(\vx)) \label{eq:MaxMinEigenvalueProblem} \\
    \text{s.t. } & \vx \in \mathcal{X} \nonumber \\
     \;& \vx \in \{0,1\}^m. \nonumber 
\end{align}
where $M:\R^m \to \SymMat_n$ is a linear map of the form
\begin{align}
    M(\vx) = C + \sum_{i=1}^m x_i \vm_i \vm_i^\intercal
\end{align}
with $C \in \SymMat_n^+$ and $\vm_i \in \R^n$.
The constraint set $\mathcal{X}$ is assumed to be polyhedral and to encode a combinatorial structure.
We will only consider binary problems in this paper but many of the following discussions can be extended to the general integer case.

A wide range of problems can be cast in this form, such as the E-Optimal Experiment Design problem or maximizing the algebraic connectivity of the network.
The minimum eigenvalue function is concave but not smooth. 
It can be naturally formulated as an \emph{mixed-integer semi-definite programming (MISDP)} problem:
\begin{align}
    \max \; &\zeta \tag{S} \label{eq:E-SDPFormulation} \\
    \text{s.t. } &  M(\vx) -\zeta I \succeq 0 \nonumber \\
     \; & \vx \in \mathcal{X}\nonumber \\
     \; &\vx \in \{0,1\}^m .\nonumber 
\end{align}
As the problem dimension increases, solving the MISDP formulation becomes computationally expensive. 
On the other hand, combinatorial structures in $\mathcal{X}$ may not be exploited as efficiently by MISDP solvers.
\emph{Frank-Wolfe (FW)} type algorithms are a natural fit for concave optimization over sets which allow for efficient linear maximization.
This motivates us to tackle Problem \ref{eq:MaxMinEigenvalueProblem} using FW type algorithms, in particular, a branch-and-bound framework utilizing
FW methods to solve concave relaxations of the problem.

The Frank-Wolfe algorithm is a projection-free first-order method for smooth constrained optimization,
introduced independently in \citet{frank1956algorithm} and \citet{levitin1966constrained}. 
It has gained traction again in the last decade due to its simplicity and scalability stemming from inexpensive iterations,
and the fact that it induces sparse iterates, i.e., iterates constructed from the convex combination of a low number of extreme points \citep{bomze2020active,wirth2024pivotingframeworkfrankwolfealgorithms},
properties that have driven its recent success in machine learning applications.
Instead of projections, FW solves linear maximization problems over the constraint set, obtaining extreme points of the feasible region and forming solutions as convex combinations of these points.
As a by-product of iterations, the algorithm provides a \emph{Frank-Wolfe gap} which upper-bounds the optimality gap at any iteration.
These properties make FW also interesting for a branch-and-bound (B\&B) framework for mixed-integer optimization, as presented in \citet{hendrych2023convex}, where a generic mixed-integer convex solving framework is proposed. 

Common Frank-Wolfe methods typically require an objective that is differentiable and smooth, either in the Lipschitz sense or with other Lipschitz-like smoothness conditions such as relative smoothness \citep{vyguzov2025adaptive}
or self-concordance \citep{dvurechensky2023generalized,carderera2021simple,carderera2024scalable} .
Most variants may fail to converge when the function is only subdifferentiable, see the counterexample in \citet{nesterov2018complexity}.
When the nonsmooth function is piecewise linear, the fully-corrective Frank-Wolfe can be shown to terminate in finite time \citep{bettiol2023oracle} but requires the exact optimization over subproblems growing in size with iterations.
A more versatile alternative is smoothing the objective function using the Nesterov smoothing technique \citep{nesterov2005smooth}, which has already been leveraged in FW variants when the objective is piecewise linear \citep{yurtsever2018conditional,besancon2024frankwolfealgorithmoraclebasedrobust}.
We use the later-proposed log-sum-exp smoothing of the eigenvalues for some form of spectral functions \citep{nesterov2007smoothing}, which is more suitable than the quadratic regularization smoothing for spectral objectives.
The original paper proposed this log-sum-exp smoothing of the maximum eigenvalue to take into account all eigenvalues with an adaptive filtering mechanism based on exponential decay,
in contrast to the spectral bundle method, which uses a fixed number of eigenvectors to construct a lower-bounding model.
Log-sum-exp smoothing for semi-definite optimization was shown in \citet{d2008smooth} to be robust to dropping terms corresponding to the largest eigenvalues in the computation of the gradient, offering the possibility for efficient computations of the few smallest eigenpairs to estimate the gradient without compromising the rate of convergence.
Recently, the log-sum-exp smoothing was shown in \citet{samakhoana2025elementary} to be close to optimal for the coordinate-wise maximum of a vector, a special case of our setting.

%In addition to the combinatorial difficulty, one aspect of the problem is the non-differentiability and non-smoothness of the objective function,
%which is a property most standard FW variants rely upon.
%Indeed, one can construct instances on which a broad class of algorithms relying on calls to linear minimization oracles may fail to find the optimum \cite%{nesterov2018complexity}.
%When the objective function is piecewise linear and convex, a dedicated fully-corrective Frank-Wolfe method can exploit
%linear optimization subproblems \citet{bettiol2023oracle}.
%We exploit an approach based on a smoothing scheme introduced in \citet{nesterov2007smoothing} for spectral functions specifically,
%and which relies on a log-sum-exp regularization.
%This technique is typically more suitable to spectral functions than the classical Nesterov smoothing based on a quadratic regularization of the dual \citet{nesterov2005smooth},
%already used in FW schemes, e.g., in \citet{yurtsever2018conditional,besancon2024frankwolfealgorithmoraclebasedrobust}.

There are two instances of Problem~\ref{eq:MaxMinEigenvalueProblem} we will highlight.
The Optimal Design of Experiments Problem (OEDP) consist in gaining the most information on a system from a limited number of experiments that need selecting.
Depending on the measure used for information, different formulations of the problem arise.
The \emph{E-Optimal Experiment Design (EOD)} fits into the general framework of Problem~\ref{eq:MaxMinEigenvalueProblem} with $\vm_i=\va_i$ encoding the experiment data,
$C=0$ and $\mathcal{X} =\{\vx \in \R^m_{\geq 0} \mid \sum_{i=1}^m x_i = N\}$ where $N$ is the budget on the number of experiments to be selected. 
So $\mathcal{X}$ is a scaled probability simplex for which linear maximization can be efficiently performed.
If the constant matrix $C$ is not the zero matrix, but itself a positive semi-definite matrix, the problem is known as the \emph{Bayesian E-Optimal Design (BEOD)} problem, or alternatively the \emph{E-Fusion Design} problem.
The aim is to maximize information gain in the presence of prior information, for example, in sensing applications,
with sensors already deployed and new sensors to be added.

In graph theory, the \emph{algebraic connectivity} of a graph $G=(V, E)$ is defined as $\lambda_2(L)$,
i.e., the second smallest eigenvalue of the Laplacian matrix $L$ of the graph and measures how well-connected the graph is.
The graph is disconnected if and only if $\lambda_2(L) = 0$.
The corresponding optimization problem is to maximize the algebraic connectivity of a given graph under some constraints.
For example, given a base network, a set of potential new edges and a budget, one would like to add a new edges to the network to maximize the algebraic connectivity.
A higher algebraic connectivity is desirable for robustness of the network.
By \citet{chan2018complexity,lamperski2024simple}, the algebraic connectivity can also be computed from the smallest eigenvalue of the shifted Laplacian matrix, so
\begin{align*}
\lambda_2(L) = \lambda_{\min}(L + \mathbf{1} \mathbf{1}^\intercal).
\end{align*}
The vector $\vm_i$ now represents a potential edge $jk$, with $m_i^j = 1$ and $m_i^k = -1$.
The constant matrix $C$ is the Laplacian matrix of the base graph.
Observe that for a connected graph, maximizing the algebraic connectivity under budget constraints is a special case of the Bayesian E-Optimal Design problem.
In \citet{mosk2008maximum}, it shown that the maximum algebraic connectivity augmentation problem is $NP$-hard,
also implying that the Bayesian E-Optimal Design problem is $NP$-hard.

The paper is structured as follows.
In \cref{sec:smoothing_minimum_eigenvalue}, we introduce the smoothing technique for the minimum eigenvalue function and establish its convergence guarantees in a Frank-Wolfe algorithm.
\cref{sec:algorithmic_techniques} introduces two pruning techniques for the branch-and-bound tree based on 
the maximal achievable rank and eigenvalue of the current feasible solution set.
In the discrete setting, the smoothed problem yields the same optimal solution as the original problem, 
given a small enough smoothing parameter.
The discussion for this can be found in \cref{sec:analysis_integral_instances}.
Aside from pruning, we also investigate possible fixing strategies from the dual problem, see \cref{sec:fixings_from_dual_problem}.
Lastly, we test our approach on instances of the E-Optimal Experiment Design problem and two variants of  Algebraic Connectivity problems, see \cref{sec:computational_experiments}.
Most proofs and supporting lemmas are deferred to \cref{app:proofs} and additional computational results are provided in \cref{app:comp_appendix}.

\paragraph{Contributions}

Our contributions are the following. We propose a mixed-integer concave optimization approach for the constrained eigenvalue maximization problem,
combining entropic smoothing with a Frank-Wolfe approach to solve continuous concave relaxations that are then combined in a branch-and-bound algorithm.
We study the convergence of our algorithm on these continuous relaxations and propose several techniques to accelerate the search,
based on the rank structure of the objective, on specific properties of integral instances, and on duality, indicating experiments which can be safely pruned. 
We provide extensive computational experiments on E-optimal design and algebraic connectivity instances.

\paragraph{Related work}
The E-Optimal Design problem has not been studied so extensively as its D-Optimal and A-Optimal counterparts, and most of the literature is considering a continuous counterpart of the problem.
Related spectral subset-selection problems include constrained generalized maximum-entropy sampling, which maximizes the product of the largest eigenvalues of a selected principal covariance submatrix and generalizes binary D-optimality \citep{ponte2026extended}.
Although this is not an E-optimal formulation, its extended-variable relaxations, scaling techniques, and branch-and-bound methodology are relevant to discrete spectral design.
A recent work by \citet{shi2026gradient} uses Wasserstein geometry to define optimal descent directions even in the case of higher multiplicity
of the minimum eigenvalue, the computation of which relies again on solving SDP problems.
Both \citet{sahu2021new} and \citet{rendi1995max} have developed primal-dual type algorithms to tackle maximization of the minimum eigenvalue.
In \citet{telen2015differentiable}, a formulation was developed for dynamic programming replacing the SDP constraint from \eqref{eq:E-SDPFormulation} 
with a set of nonlinear inequalities based on the Sylvester criterion for positive definiteness.
While we consider only a linear experiment model, in \citet{dette2015optimal}, (fractional) E-Optimal design problems are considered for second-order response surface models.
Approximation algorithms are designed for the A- and E-optimal design problems in \citet{chamon2017approximate,brown2024fast,lamperski2024simple}, based on a greedy algorithm or a simple rounding scheme.
The work of \citet{lamperski2024simple} also considers the Algebraic Connectivity problem in its minimum eigenvalue formulation.
The authors of \citet{farhadi2020lambda} show that the dual problem of maximizing the second smallest eigenvalue, the Maximum Variance Embedding problem, is $NP$-hard.
Both \citet{kolla2010subgraph} and \citet{wei2014algebraic} try to approximate the integer optimum, the former introducing a new
subgraph sparsification technique, the latter comparing a greedy algorithm, an algorithm specialized for adding/deleting edges, and 
a rounding heuristic based on the continuous SDP relaxation.
Finally, we will highlight a previous approach similar in spirit from the solver BiqCrunch \citep{krislock2017biqcrunch} which tackles pure-binary quadratic problems through branch-and-bound leveraging a smoothed dual SDP at every node.

\subsection*{Notation}
The set of integers is denoted by $\Z$ and the natural numbers without zero is denoted by $\N$.
We denote with $\SymMat_n \subset \R^{n \times n}$ the set of real symmetric matrices, $\SymMat_n^+$ denotes the set of positive 
semi-definite matrices, and $\SymMat_n^{++}$ denotes the set of positive definite matrices.
The diagonal vector of a matrix $A$ is denoted as $\diag(A)$ and the diagonal matrix of a vector $\vx$ is denoted as $\Diag(\vx)$.
Matrices and sets are denoted in capital letters, vectors in lowercase and boldface letters, and scalars in lowercase.
The rows of a matrix $A$ are denoted as $A_{i,\cdot}$ and the columns as $A_{\cdot,j}$.
For a natural number $k$, we use the shorthand $[k]$ for $\{1,\ldots k\}$.
We denote by $\norm{X}_{\mathrm{op}}$ the operator norm of a matrix, by $\norm{X}_*$ its trace norm and by $\norm{X}_2$ its spectral norm.
For the eigenvalues of a matrix $X \in \SymMat^n$, we use the notation $\lambda_1(X) \leq \lambda_2(X) \leq \cdots \leq \lambda_n(X)$ to 
denote the eigenvalues of $X$ in non-decreasing order with $\lambda_1(X) = \lambda_{\min}(X)$ and $\lambda_n(X) = \lambda_{\max}(X)$.

\section{Smoothing the minimum eigenvalue}\label{sec:smoothing_minimum_eigenvalue}

As stated in the introduction, Frank-Wolfe requires a concave and differentiable objective.
If the objective is $L$-smooth, the classical convergence guarantees hold \citep{jaggi2013revisiting,braun2022conditional}. 
While the minimum eigenvalue function is concave, if the eigenvalue has a multiplicity $> 1$ for a given point, it is not differentiable there.
Thus, the function has to be smoothed to make FW applicable.
We will utilize the technique from Nesterov in \citet{nesterov2007smoothing} and consider first the continuous version of problem \ref{eq:MaxMinEigenvalueProblem}.
\begin{align}
    \label{eq:EOptGeneral}
 \max_{\vx} & \; \lambda_{\min}\left(M(\vx)\right)\\
 \text{s.t. }& \; \vx \in \mathcal{X} \nonumber 
\end{align}

Let us define the function $f(M(\vx)) = \lambda_{\min}(M(\vx))$ and use the Nesterov smoothing on $f$:
\begin{equation}
    \label{eq:SmoothingMinEigMax}
    f_{\mu}(M(\vx)) = -\mu E\left(-1/\mu M(\vx)\right) + \mu \log n  = -\mu \log\left(\sum_{i=1}^n e^{-\lambda_i(M(\vx))/\mu}\right) + \mu \log n 
\end{equation}
with $\mu > 0$ being the smoothing parameter.
For convenience, let $F(-1/\mu M(\vx)) = \sum_{i=1}^n e^{-\lambda_i(M(\vx))/\mu}$. 
For all $\vx\in\mathcal{X}$ and $\mu>0$, the smoothed objective satisfies
\begin{align}
    f(M(\vx)) \leq f_\mu(M(\vx)) \leq f(M(\vx))+\mu\log n.
    \label{eq:smoothing_sandwich}
\end{align}

Now that we have established our auxiliary objective, we prove guarantees for the primal gap of
the original problem given our known guarantees on the smoothed objective.
\begin{restatable}{theorem}{thmFrankWolfeConvergence}
\label{th:FrankWolfeConvergence}
Let $\epsilon > 0$ and $\mu = \frac{\epsilon}{2\log n}$ and $g(\vx)=\lambda_{\min}(M(\vx))$.
We denote $\norm{M}_\mathrm{op} = \max\limits_{\norm{\vx}=1} \norm{M(\vx)}_2$.
The function $g_\mu(\vx) = f_{\mu}(M(\vx))$ is $L$-smooth with Lipschitz constant $L = \norm{M}_\mathrm{op}^2 /\mu$.
Denote $\mathcal{X} \subseteq \left[0,1\right]^m$ the compact feasible set of the relaxation of diameter $D$ and $\vx_t \in \mathcal{X}$ the $t$-th Frank-Wolfe iterate of the smoothed problem and by $g^*$ the optimal value of the original problem.
Then,
\begin{align}
    g^*  - g(\vx_t) \leq \norm{M}_\mathrm{op} D \sqrt{\frac{8 \log n}{t}}.
\end{align}
If furthermore, $\Xset$ is the hypersimplex $\{\vx \in \left[0,1\right]^m, \sum_i x_i = N\}$,
then the bound becomes
\begin{align*}
    g^*  - g(\vx_t)\leq 4 \norm{M}_\mathrm{op} \sqrt{\frac{N \log n }{t}}.
\end{align*}

\end{restatable}
Thus, Frank-Wolfe will yield a better solution with a smaller $\mu$ but potentially at the cost of numerical stability.
We note, however, that the $\mu$ defined by the above theorem is too small in practice, and will lead to numerical instability.
Ultimately, we are interested in an integral solution to our problem and we will be using Frank-Wolfe as the node solver
in a branch-and-bound scheme.
If $\mu$ is too small, the node evaluation becomes costly from the Lipschitz smoothness constant increase; if it is too large, the node-relaxation upper bound in the B\&B tree will be too loose.
Combining the best of both worlds, we will use a dynamic rule for $\mu$ by starting with a relatively large value and decreasing it
depending on the depth of the tree until a specified limit. 
At nodes with higher depths, many variables will already be fixed and the dimension of the subproblem is thereby reduced.

The gradient of the smoothed function $f_\mu$ is given in the following lemma.
\begin{restatable}{lemma}{lemmagradient}\label{lemma:gradient}
    Let $\{\lambda_i\}_{i\in[n]}$ be the eigenvalues of $M(\vx)$, with associated eigenvectors
    $\{\vu_i\}_{i\in[n]}$, then the entries of the gradient of $g_{\mu}$ are given by
    \begin{align*}
    \left[\nabla g_{\mu}(\vx)\right]_k = \sum_{i=1}^n \frac{e^{ -\lambda_i / \mu }}{\sum_{j=1}^n e^{ -\lambda_j / \mu }}
    \innp{\vu_i}{\vm_k}^2 .
    \end{align*}
\end{restatable}

Given that we are in the integer setting, a natural question is whether we can find a $\mu$ such that the corresponding optimization
problem has the same solution as the original problem.
First, we require a bound relating $f(M(\vx_\mu^*))$ and $f(M(\vx^*))$ where $\vx_\mu^*$ and $\vx^*$ are optimizer of the smoothed and original problem,
respectively.
By \cref{eq:smoothing_sandwich}, we have
\begin{equation}
    \label{eq:f_mu_bound}
    f_{\mu}(M(\vx)) \geq f(M(\vx)) \geq f_{\mu}(M(\vx))-\mu\log n.
\end{equation}
$\vx_\mu^*$ being an optimizer implies that 
\begin{align}
    f_{\mu}(M(\vx_{\mu}^*)) &\geq f_{\mu}(M(\vx)) \quad \forall \vx \in \mathcal{X} \label{eq:f_mu_opt}\\
\intertext{By optimality of $\vx^*$, we have }
f(M(\vx^*)) &\geq f(M(\vx_{\mu}^*)). \label{eq:f_opt_not} \\
\intertext{Combining the previous three inequalities, we get}
 f_{\mu}(M(\vx_{\mu}^*)) &\underset{\text{\ref{eq:f_mu_opt}}}{\geq} f_{\mu}(M(\vx^*))
    \underset{\text{\ref{eq:f_mu_bound}}}{\geq} f(M(\vx^*))
    \underset{\text{\ref{eq:f_opt_not}}}{\geq} f(M(\vx_{\mu}^*))
    \underset{\text{\ref{eq:f_mu_bound}}}{\geq} f_{\mu}(M(\vx_{\mu}^*))-\mu\log n.
    \label{eq:f_mu_bound_opt}
\end{align}
From (\ref{eq:f_mu_bound_opt}), we can conclude that
\[f(M(\vx^*)) - f(M(\vx_{\mu}^*)) \leq \mu \log n .\]
On the other hand, the set of feasible and non-optimizer points $\mathcal{Y}$ is finite by integrality constraints and the compact feasible region.
Therefore, we can find some $\delta > 0$ such that 
\[f(M(\vx^*)) - f(M(\vy)) \geq \delta \quad \forall \vy \in \mathcal{Y}.\]
So for \[\mu < \frac{\delta}{\log n}\] the integer optimizer of the smoothed problem has to be the same as the integer optimizer of the original problem.

Computing $\delta$ for a given feasible set and map $M$ is non-trivial and even then,
it will typically be too small, and so will the corresponding $\mu$, making it impractical for computations.

The tightness of the relaxation directly depends on the smoothing parameter $\mu$.
Thus, we investigated if a tighter relaxation can be obtained from the original function.
For D- and A-optimal design, several lines of work developed tighter relaxations transforming the objective function, see e.g.,
\citet{li2025augmented,ponte2024admm} for D-optimal and \citet{li2025strong} for A-optimal design.
In contrast, the E-optimal design of experiment captures most semi-definite optimization problems of interest as shown in the following \cref{prop:eoptgeneric}, leaving little hope for a tighter factorization.
This reformulation is standard and has motivated, e.g., the development of spectral bundle methods for SDP \citep{helmberg2000spectral}.

\begin{restatable}{proposition}{propeoptgeneric}
\label{prop:eoptgeneric}
Given symmetric matrices $C$, $\{A_i\}_{i\in[m]}$ and vector $\vb$,
consider the following semi-definite optimization template in the primal form:
\begin{align}\label{prob:primalSDP}
\min_{Y} \; & \innp{C}{Y} \\
\mathrm{s.t.} \; & \innp{A_i}{Y} = b_i \;\;\forall i \in [m] \nonumber \\
& Y \in \SymMat_+^n.\nonumber
\end{align}
Assume \eqref{prob:primalSDP} is strictly feasible, i.e., that there exists $Y \succ 0$ such that $\innp{A_i}{Y} = b_i \forall i \in [m]$,
and that there exists $\tau \in \R$ such that $\mathrm{Tr}(Y) \leq \tau$ is a valid inequality of the problem, i.e.~that it preserves at least one optimal solution.
Then solving \eqref{prob:primalSDP} amounts to solving an instance of \eqref{eq:EOptGeneral}.
\end{restatable}

\section{Branch-and-bound pruning}\label{sec:algorithmic_techniques}
In this section, we develop a custom rank-based pruning rule and an eigenvalue-based pruning rule for nodes in the B\&B tree.

\begin{restatable}{proposition}{proprankbasedpruning}[Rank-based pruning]
\label{prop:rankbasedpruning}
At a node $l$ of the B\&B tree, let $\mathcal{S}_l^0$, $\mathcal{S}_l^1$ be the set of variables fixed to zero, one, respectively,
and let $\mathcal{S}_l^f$ be the set of variables not fixed yet.
Define the remaining budget $N_l = N - \mathcal{S}_l^1$, the partial matrix, residual feasible set and residual matrix:
\begin{align*}
& C_l = C + \sum_{i\in \mathcal{S}_l^1}  \vm_i \vm_i^\intercal \\
& \Xset_l = \{ \vx \in \{0,1\}^{|\mathcal{S}_l^f|}, \sum_{i \in \mathcal{S}_l^f} x_i = N_l \} \\
& M^l = \begin{bmatrix}
    \vm_{i_1} & \vm_{i_2} \dots \vm_{i_{|\mathcal{S}_l^f|}}
\end{bmatrix}_{i_s \in \mathcal{S}_l^f}.
\end{align*}
If $\rank(C_l) + \min \{\rank(M^l), N_l\} < n$, then $\lambda_{\min}(C_l + \sum_{i \in \mathcal{S}_l^f} x_i \vm_i \vm_i^\intercal) = 0$ for any $\vx \in \Xset_l$
and the node can be pruned.
\end{restatable}

\cref{prop:rankbasedpruning} can be used to efficiently stop exploring nodes and only requires maintaining the rank of the partial solution constructed in the tree from the fixings.
If $C \succ 0$, the technique cannot prune any node since the problem is never rank-deficient.
We can however define an equivalent problem with $\hat{C} = C - \lambda_{\min}(C) I$, shifting the objective by $\lambda_{\min}(C)$.
We now state in \cref{prop:interlacing_bound} another upper bound that can be computed at any node prior to solving the continuous concave relaxation.

\begin{restatable}{proposition}{propinterlacingupperbound}[Eigenvalue-Based Upper Bound]\label{prop:interlacing_bound}
At a node $l$ of the B\&B tree, retain the notation of \cref{prop:rankbasedpruning}.
Let the eigenvalues of $C_l$ be ordered non-decreasingly as
$\lambda_1(C_l) \leq \cdots \leq \lambda_n(C_l)$.
Then for every $\vx \in \Xset_l$:
\begin{align*}
    \lambda_{\min}\!\left(C_l + \sum_{i \in \mathcal{S}_l^f} x_i \vm_i \vm_i^\intercal\right)
    \;\leq\;
    \lambda_{N_l+1}(C_l),
\end{align*}
with the convention $\lambda_{N_l+1}(C_l) = +\infty$ if $N_l \geq n$.
\end{restatable}
We will refer to the pruning rule based on \cref{prop:interlacing_bound} as eigenvalue-based pruning.

\section{Truncated gradient computations}

In this section, we leverage the approximate gradient technique from \citet{d2008smooth} -- presented for a projected gradient descent --
to our setup in which we only exploit linear optimization over the feasible set.

Define the partition function and soft-min density matrix as
\begin{align*}
\rho_\mu(\vx) := \sum_{i=1}^n e^{-\lambda_i(M(\vx))/\mu}, \qquad P_\mu(\vx) := \frac{1}{\rho_\mu(\vx)} \sum_{i=1}^n e^{-\lambda_i(M(\vx))/\mu}\, \vu_i(M(\vx))\vu_i(M(\vx))^\intercal,
\end{align*}
so that the gradient from \cref{lemma:gradient} reads
$[\nabla g_\mu(\vx)]_k = \langle P_\mu(\vx),\, \vm_k \vm_k^\intercal \rangle$ for $k \in [m]$.

For a truncation level $r \in [n]$, computing only the $r$ smallest eigenvalues of $M(\vx)$ yields the $r$-truncated density matrix and truncated gradient:
\begin{align*}
\widetilde{P}_\mu^{(r)}(\vx) = \frac{\sum_{i=1}^r e^{-\lambda_i(M(\vx))/\mu}\, \vu_i(M(\vx))\vu_i(M(\vx))^\intercal}{\sum_{i=1}^r e^{-\lambda_i(M(\vx))/\mu}}, \qquad \left[\widetilde{\nabla} g_\mu^{(r)}(\vx)\right]_k := \left\langle \widetilde{P}_\mu^{(r)}(\vx),\, \vm_k \vm_k^\intercal\right\rangle.
\end{align*}

We quantify the approximation quality through the tail weight:
\begin{align*}
\tau_r(\vx) := 1 - \frac{\sum_{i=1}^r e^{-\lambda_i(M(\vx))/\mu}}{\rho_\mu(\vx)}
= \frac{\sum_{i=r+1}^{n} e^{-\lambda_i(M(\vx))/\mu}}{\rho_\mu(\vx)},
\end{align*}
and the tail-control gap
\begin{align*}
\Delta_r(\vx) :=
\begin{cases}
\lambda_{r+1}(M(\vx))-\lambda_1(M(\vx)), & r<n,\\
+\infty, & r=n.
\end{cases}
\end{align*}
The corresponding linear-oracle error is denoted by
$\delta_r(\vx):=2\|M\|_{\mathrm{op}}^2\tau_r(\vx)$.
By bounding the spectral gap as an error measure, we can derive a Frank-Wolfe convergence proof under inexact gradients akin to the one originally from \citet{jaggi2013revisiting}.

\begin{restatable}[Frank-Wolfe convergence with truncated gradient]{theorem}{thmFWtruncated}
\label{thm:fwtruncated}
Let $\mu > 0$ and $L = \|M\|_{\mathrm{op}}^2/\mu$ be the smoothness constant of $g_\mu$ from \cref{th:FrankWolfeConvergence}.
Let $\bar{\delta} > 0$ and suppose that at each iteration $t$ the truncation level $r_t$ satisfies
$\delta_{r_t}(\vx_t) \leq \bar{\delta}$. Starting from $\vx_0 \in \mathcal{X}$, run the Frank-Wolfe iterates:
\begin{align*}
& \tilde{\vvv}_t = \underset{\vvv \in \mathcal{X}}{\argmax}\innp{\widetilde{\nabla} g_\mu^{(r_t)}(\vx_t)}{\vvv} \\
& \vx_{t+1} = (1-\gamma_t)\vx_t + \gamma_t \tilde{\vvv}_t \\
& \gamma_t = \frac{2}{t+2}.
\end{align*}
Then for all $T \geq 1$:
\begin{align*}
    g^* - g(\vx_T) \leq \mu\log n + \frac{2LD^2}{T+1} + 2\bar{\delta},
\end{align*}
where $g^* = \max_{\vx \in \mathcal{X}} g(\vx)$, and $D = \mathrm{diam}(\mathcal{X})$.
\end{restatable}

\begin{restatable}[Complexity and truncation condition]{corollary}{corcomplexitytruncation}
\label{cor:complexitytruncation}
Let $\varepsilon > 0$, $\mu = \varepsilon/(2\log n)$, $D^2 = 2N$ for the hypersimplex, and $\bar{\delta} = \varepsilon/8$.
Then $g^* - g(\vx_T) \leq \varepsilon$ is achieved at iteration $T$ such that
\begin{align*}
T \geq \frac{32N\|M\|_{\mathrm{op}}^2\log n}{\varepsilon^2}.
\end{align*}
For $r_t<n$, the truncation condition $\delta_{r_t}(\vx_t) \leq \varepsilon/8$ at iteration $t$ is satisfied whenever the spectral gap
\begin{align*}
\Delta_{r_t}(\vx_t) = \lambda_{r_t+1}(M(\vx_t)) - \lambda_1(M(\vx_t))
\geq \mu\log\!\left(\frac{16(n - r_t)\|M\|_{\mathrm{op}}^2}{\varepsilon}\right).
\end{align*}
For $r_t=n$, the gradient is exact and $\delta_{r_t}(\vx_t)=0$.
Similarly to the unconstrained case in \citet{d2008smooth}, the iteration count is independent of the truncation level $r_t$;
only the per-iteration cost varies, requiring $O(r_t n^2)$ operations for a partial eigendecomposition of $M(\vx_t)$ via matrix-vector products.
\end{restatable}

We highlight that the condition on $\Delta_r(\vx_t)$ is instance- and iterate-specific:
a large separation between the first eigenvalue and the discarded tail permits a small truncation level; in particular,
a large gap between $\lambda_2(M(\vx_t))$ and $\lambda_1(M(\vx_t))$ means that a single eigenvector pair can suffice.
Predicting this gap from the structure of $A$ a priori remains an open question and too cumbersome in general.
In practice, the condition $\delta_{r_t}(\vx_t) \leq \bar{\delta}$ can be checked adaptively while computing the smallest eigenpairs.

\section{Analysis for integral instances}\label{sec:analysis_integral_instances}

We consider in this section the specific case in which $M(\vx)$ is integral with for any binary assignment of $\vx$, e.g.,
when the entries of the vectors $\vm_k$ and matrix $C$ are integers.
This lets us derive formal guarantees that the smoothing constant never needs vanishing to zero to reach a lower precision.
Using the notations introduced in \cref{prop:rankbasedpruning}, we define two more quantities:
\begin{align*}
T^0_l &= \Tr(C_l) \\ 
T_l &= T^0_l + \sum_{i = 1}^{N_l} \norm{\vm_{(i)}}_2^2.
\end{align*}
where $\vm_{(i)}$ is the $i$-th vector of $\mathcal{S}_l^f$, sorted by non-increasing Euclidean norm.

\begin{restatable}[Minimum suboptimality gap]{theorem}{thmsafesmoothing}
\label{thm:safesmoothing}
Consider a node $l$ of the branch-and-bound tree on an instance for which $M_l(\vx)$ is integer for any $\vx \in \{0,1\}^m$,
and using the notations introduced in \cref{prop:rankbasedpruning}.
Define the submatrix $M_{\mathcal{S}_l^f}$ by selecting the $\vm_k$ corresponding to the set $\mathcal{S}_l^f$.
% reduced rank $r_l = \rank(A_{\mathcal{S}_l^f})$, with 
For $\vx \in \Xset_l$ such that $M_l(\vx) \succ 0$, define $q_{\vx} \in \Z\left[\lambda\right]$ the characteristic polynomial of $M_l(\vx)$:
\begin{align*}
q_{\vx}(\lambda) = \det(\lambda I - M_l(\vx)) = \prod_{i=1}^{n} (\lambda - \lambda_i(M_l(\vx))).
\end{align*}
Let $\vx_l^* \in \Xset_l$ be an integer optimal completion at node $l$ and $\vy \in \Xset_l$ be any non-optimal feasible completion such that $M_l(\vy) \succ 0$.
Denote $g^* = \lambda_{\min}(M_l(\vx_l^*))$, $g_y = \lambda_{\min}(M_l(\vy))$, and $\delta_l = g^* - g_y > 0$.
We have that $q_{\vx_l^*}$ and $q_{\vy}$ are integer polynomials of degree $n$,
with roots $g^*=\lambda_1(M_l(\vx_l^*)) \leq \dots \leq \lambda_n(M_l(\vx_l^*))$
and $g_y=\lambda_1(M_l(\vy)) \leq \dots \leq \lambda_n(M_l(\vy))$, respectively.
Define:
\begin{align*}
d = \deg(\gcd(q_{\vx_l^*}, q_{\vy})) \in \{0, \dots, n-1\},
\end{align*}
where the upper bound $n-1$ arises because $g_y$ is a root of $q_{\vy}$ but not $q_{\vx_l^*}$, since $g^* > g_y$.
Then, we have
\begin{align*}
\delta_l \geq \frac{1}{T_l^{n^2-nd-1}}
\end{align*}
\end{restatable}

We can use a bound with $d=0$ in the theorem to derive a corollary on a practical choice of $\mu_l$.
\begin{corollary}[Safe smoothing condition]\label{cor:smoothnessbound}
At node $l$, given a numerical tolerance $\varepsilon \in (0,1)$, a fixed smoothing parameter
\begin{align*}
    \mu_l = (1-\varepsilon) \frac{1}{T_l^{n^2-1} \log n}
\end{align*}
ensures that the set of integer optimizers of $g_{\mu}$ coincides with the set of optimizers of $g$.
\end{corollary}
\cref{cor:smoothnessbound} allows us -- together with a bound on $\norm{M}_\mathrm{op}$
to bound more finely the FW iteration count to reach a tolerance $\varepsilon$ at any node.
In addition, this bound improves further down the tree, since $T_l$ is non-increasing as more experiments are fixed.

\section{Duality-based variable fixing}\label{sec:fixings_from_dual_problem}

In practice, the budget is significantly smaller than the number of possible experiments/edges.
Thus, we are interested in obtaining an exclusion criterion on the experiment space,
allowing us to prune experiments that are certified to not be part of the optimal solution.
In \citet{ahipasaoglu2025column}, such an exclusion criterion was derived for the continuous experiment
design problem under the D-Optimal criterion to exclude experiments during the solution process. 
Taking a similar approach, we derive a criterion that allows the inclusion or exclusion of an experiment
based on the current node solution, dual problem at the node and the tree incumbent.
Note that in the following, we consider problem \eqref{eq:EOptGeneral} where $\mathcal{X}$ encodes a scaled 
probability simplex and integer constraints.

The primal problem of interest is 
\begin{align}
    \label{eq:PrimalProblem}
    \max_{\vx} & \; \lambda_{\min}\left(M(\vx)\right) \tag{P} \\
    \text{s.t. }& \; \sum\limits_{i=1}^m x_i = N \nonumber \\
    & \; \vl \leq \vx \leq \vu . \nonumber
\end{align}
where $\vl$ and $\vu$ are the specific lower and upper bounds coming from branching.
We will only consider the case where $\vl = 0$ and $\vu = 1$, the following derivation can 
be easily adapted to the general case to get tightenings on the variable $\vx$.
To derive the dual problem, we first reformulate problem as a semi-definite program (SDP).
\begin{align}
    \label{eq:SDPPrimal}
    \max_{\lambda, \vx} & \; \lambda \tag{SDP-P} \\
    \text{s.t. }& \; M(\vx) - \lambda I \succeq 0 \nonumber \\
    & \; \sum\limits_{i=1}^m x_i = N \nonumber \\
    & \; \vl \leq \vx \leq \vu \nonumber
\end{align}
Observe that $M(m/N\bm{1})$ is regular, then we can find
strictly feasible $(\lambda, \vx)$ for \eqref{eq:SDPPrimal}.
Hence, Slater's condition is satisfied and strong duality holds.
To start, the Lagrangian of \eqref{eq:SDPPrimal} is given by
\begin{align}
    \label{eq:Lagrangian}
    \mathcal{L}(\lambda, \vx, Z, \zeta, \valpha, \vbeta) = - \lambda - \innp{Z}{M(\vx) - \lambda I} - \innp{\zeta}{N - \sum\limits_{i=1}^m x_i} - \innp{\valpha}{\vx - \vl} - \innp{\vbeta}{\vu - \vx}.
\end{align}
Let $M_i$ be the derivative of $M(\vx)$ with respect to $x_i$.
The dual problem of \eqref{eq:SDPPrimal} is given by
\begin{align}
    \label{eq:SDPDual}
    \min_{Z, \zeta, \valpha, \vbeta} & \; N\zeta - \innp{\valpha}{\vl} + \innp{\vbeta}{\vu} + \innp{Z}{C} \tag{D} \\
    \text{s.t. }& \; \Tr(Z) = 1 \nonumber \\
    & \; \zeta = \innp{Z}{M_i} + \alpha_i - \beta_i \quad \forall i \in [m] \nonumber \\
    & \; Z \succeq 0 \nonumber \\ 
    & \; \valpha, \vbeta \geq 0 \nonumber.
\end{align}
Let us now suppose we have a set of dual variables $(Z, \zeta, \valpha, \vbeta)$. 
We can derive a tightening of bounds based
on these dual values and the tree incumbent. 
\begin{restatable}{proposition}{proptightening}\label{prop:tightening}
Let $S_D = (Z, \zeta, \valpha, \vbeta)$ be a dual solution of \eqref{eq:SDPDual} for given $\vl, \vu \in [0,1]^m$, $\vx$ a primal 
solution and $\hat{G}$ the tree incumbent.
Further, denote by $d_S = N\zeta - \innp{\valpha}{\vl} + \innp{\vbeta}{\vu} + \innp{Z}{C}$ the objective value of the dual problem.
For a variable $x_j$ that is not yet fixed, we have the following conditions which cannot remove any optimal solution of the problem:
\begin{align*}
    d_S - \alpha_j \leq \hat{G} & \Rightarrow \text{ fix $x_j = 0$}\\
    d_S - \beta_j\leq \hat{G} & \Rightarrow \text{ fix $x_j = 1$}.
\end{align*}
\end{restatable}

To utilize the conditions in \cref{prop:tightening} in practice, we need to compute a dual solution for the current node.
One option is to solve the dual problem \eqref{eq:SDPDual} for the current node.
This would yield the best dual values and therefore to the most fixed variables, but it is computationally too demanding.
Alternatively, we can derive a dual solution from a given primal solution $\vx$, which we formalize in the following lemma.

\begin{restatable}{lemma}{lemmadualsolutionfromprimal}\label{lem:dual_solution_from_primal}
Let $(\lambda, \vx)$ be a primal solution pair of \eqref{eq:SDPPrimal} and $\vl, \vu \in [0,1]^m$ be the lower and upper bounds at the current node.
Then, a dual solution $S_D$ can be computed from the primal solution $\vx$ as follows:
\begin{align}
    Z &= \frac{1}{p} \sum\limits_{i=1}^p \vw_i\vw_i^{\intercal} \label{eq:Z_from_primal} \\
    \zeta &= \vvv_{(N)} \label{eq:zeta_from_primal} \\
    \alpha_i &= \max(0, \zeta - v_i) \quad \forall i \in [m] \label{eq:alpha_from_primal} \\
    \beta_i &= \max(0, v_i - \zeta) \quad \forall i \in [m] \label{eq:beta_from_primal}.
\end{align}
where $\vw_i$ are the eigenvectors associated with the smallest eigenvalue of $M(\vx)$, $p$ is the multiplicity of
the minimum eigenvalue, $v_i=\innp{Z}{\vm_i\vm_i^\intercal}$ is the score of experiment $i$ and $\vvv_{(N)}$ is the $N$-th
largest component of $\vvv$ in descending order.
\end{restatable}

An important structure of the problem is simplicity of eigenvalues, which conditions whether the original function $\lambda_{\min}(M(\cdot))$ has a unique supergradient at a given point.
We derive in the following \cref{prop:simplicityeigengap} a condition for the eigenvalue of the matrix to remain simple at a given point.

\begin{restatable}[Simplicity via eigengap inheritance]{proposition}{propsimplicityeigengap}\label{prop:simplicityeigengap}
\label{prop:simplicity_eigengap}
Let $\gamma_C := \lambda_{2}(C) - \lambda_{1}(C) > 0$ be the eigengap of $C$ at its minimum eigenvalue.
Using the notation of \cref{prop:rankbasedpruning,sec:analysis_integral_instances}, if at node $l$:
\begin{align}
\gamma_C > T_l, \label{eq:simplicity_condition}
\end{align}
then $\lambda_{\min}(M_l(\vx) + C)$ is simple for every $\vx \in \Xset_l$.
\end{restatable}

\begin{restatable}[Refined Exclusion Criterion]{theorem}{thmexclusion}
\label{thm:exclusion}
Let $\vm_i \in \mathbb{Z}^{n}$ for all $ i \in [m]$. Retain the notation of
\cref{thm:safesmoothing},
at node $l$. Let $\vx_t \in \mathrm{conv}(\mathcal{X}_l)$ be the $t$-th
Frank-Wolfe iterate on the smoothed objective $g_{\mu_l}$ with FW gap
$\Gamma_t$ and certified upper bound
$\bar{g}_l = g_{\mu_l}(\vx_t) + \Gamma_t \geq g_{\mathrm{rel}}^*$.
Let $\bar{\zeta}_t = [\nabla g_{\mu_l}(\vx_t)]_{(N_l)}$, that is the $N_l$-th largest component of $\nabla g_{\mu_l}(\vx_t)$ over $i \in S_f^l$.
Assume the minimum eigenvalue of $M_l(\vx_t)$ is simple, with computable
eigengap:
\begin{align}
    \gamma_t = \lambda_{2}(M_l(\vx_t)) - \lambda_1(M_l(\vx_t)) > 0.\label{eq:eigengap}
\end{align}
Let $\vw_t$ be a computed approximation to the minimum eigenvector of
$M_l(\vx_t)$ satisfying
$\|\vw_t\vw_t^\intercal - \hat{\vw}_t\hat{\vw}_t^\intercal\|_F \leq \eta$,
where $\hat{\vw}_t$ is the exact minimum eigenvector of $M_l(\vx_t)$.
Define:
\begin{align}
    \varepsilon_l &= \|{M}_l\|_{\mathrm{op}}\sqrt{2N_l}, \label{eq:epsdef} \\
    \sigma_l &= \frac{\sqrt{2}\,\varepsilon_l}{\gamma_t} + \eta. \label{eq:sigma}
\end{align}
Then fixing $x_i = 0$ at node $l$ does not change the optimal value whenever:
\begin{align}
    & \varepsilon_l < \gamma_t / 2 \label{eq:safetyweyl} \\
    & (\vw_t^\intercal \vm_i)^2 
    \bar{\zeta}_t - \Gamma_t - \|\vm_i\|_2^2 \cdot \sigma_l, \label{eq:EC}
\end{align}
where $\bar{\zeta}_t = [\nabla g_{\mu_l}(\vx_t)]_{(N_l)}$ is the dual price of
the FW linear subproblem \eqref{eq:LPl} with respect to the budget
constraint.
\end{restatable}

\cref{thm:exclusion} can be leveraged together with the simplicity condition from \cref{prop:simplicityeigengap} to ensure the positive eigengap condition holds.

\section{Computational experiments}\label{sec:computational_experiments}

We assess the computational performance of the proposed approach within the \textsc{Boscia.jl}
framework \citep{hendrych2023convex,BosciaGitHub} which is based on a branch-and-bound algorithm utilizing Frank-Wolfe methods for concave relaxations.

\paragraph*{Adaptations of \textsc{Boscia}.}
To facilitate the use of our approach, a \emph{smoothing mode} was added to the framework. 
The user provides the original function and a function to compute subgradients at the a given point.
Additionally, the user provides a function which generates the smoothed objective and its gradient based on a given 
smoothing parameter $\mu$ and an optional node tolerance $\varepsilon$.
Further, an initial smoothing parameter $\mu_0$, a minimum smoothing parameter $\mu_{\min}$ and the decay factor can be provided.
The upper bound at the given node is computed as the minimum of the upper bound reported by FW from the relaxed smoothed problem and the upper bound of the original function computed from the smallest dual gap computed from the subgradients.
The implementation uses Boscia's minimization convention and therefore passes the objectives
$-g$ and $-g_\mu$, together with their negated (sub)gradients, to the solver.
All mathematical statements in this paper use the equivalent concave-maximization convention, and objective values shown in the paper are transformed back to that convention.
Lastly, there is a resolve feature which solves the current node problem with a tighter smoothing parameter starting from the current solution.
If the smoothing parameter is too large, the smoothed problem can converge to an integer solution which is suboptimal for the original problem. 
Thus, a small resolve is performed to verify the integer solution.
Note that the features like pruning and fixing can be done via the native callbacks in \textsc{Boscia.jl} and do not require
additional adaptation of solver source code itself. 

\paragraph*{Alternative methods for constrained eigenvalue optimization.}
We compare our approach to the open-source solver \textsc{SCIP-SDP}~\citep{gally2018framework} for mixed-integer SDP.
\textsc{SCIP-SDP} can execute either a nonlinear branch-and-bound solving an SDP at each node or an outer approximation algorithm in which eigenvector cuts are generated to separate non-PSD matrices.
We refer to the former as \textsc{SCIPSDP (B\&B)} and to the latter as \textsc{SCIPSDP (OA)} in the evaluations.

\subsection*{Considered problems}

\paragraph*{E-Optimal Experiment Design (EOD).} Given a set of $m$ experiments encoded in the matrix $A^{m \times n}$, where $n$ represents
the number of model parameters, and a budget $N \geq n$ which is assumed to be significantly smaller than $m$,
we want to maximize the minimum eigenvalue of the information matrix $X(\vx) = A^\intercal \Diag(\vx) A$ under the 
constraint $\sum_{i=1}^m x_i = N$ and where $\vx \in \{0,1\}^m$ is the selection vector of the experiments.
This is equivalent to minimize the worst variance in the model parameters and can be graphically interpreted as finding a minimum volume enclosing ellipsoid which is as close to a ball as possible.
We generate the two experiment data sets randomly, in one of the sets we enforce correlation between the experiments, in the other, the experiments are independent from each other. Considered dimensions are $50,80,100,120,150$ and $n$ is the square root of the dimension rounded down. For the budget, we consider two constructions, one is $N=1.5n$ and the other is $N=1.5n\log(n)$.
We will refer to this problem as EOD in the evaluations.

\paragraph*{Maximum Algebraic Connectivity (AGC).} Given a graph $G=(V,E)$ with $n$ vertices and $m$ edges and potential new
edges $R \subseteq V \times V\setminus E$, we want to maximize the algebraic connectivity of the graph under a budget $N$ on the number of additional edges.
As previously discussed, the algebraic connectivity is measured by the second smallest eigenvalue of the Laplacian matrix
of the graph. However, it can be transformed into a minimum eigenvalue problem by adding the all-ones matrix to the Laplacian matrix.
We also generate the graphs randomly; one set has connected base graphs and the other has disconnected base graphs.
The dimension denotes the number of potential new edges and ranges from $80$ to $200$. 
The number of nodes is a third of the dimension and the budget is half of the dimension.
This problem will be referred to as AGC in the evaluations.

\paragraph*{Algebraic Connectivity Spanning Tree (ACST).} Given a weighted graph $G=(V,E,W)$ with $n$ vertices and $m$ edges, we want to find a spanning tree of the graph which maximizes its algebraic connectivity:
\begin{align}
    \max_{\vx} \; &\lambda_2(L_{G(\vx)}) \tag{ACST} \label{eq:ACST} \\
    \text{s.t. } \; & G(\vx) \text{ is a spanning tree of } G, \nonumber 
\end{align}
where we note $G(\vx)$ the subgraph induced by set of edges whose incidence vector is $\vx$.
As stated in \citet{somisetty2024optimal}, the problem is $NP$-hard.
They used a Mixed Integer Linear Formulation with additional cutting planes to solve the problem.
We do not compare to their solution method since its primary motivation was to avoid building a MISDP solver from scratch but we use their formulation of the problem for \textsc{SCIP-SDP}.
\begin{align}
    \max_{\lambda,W,\vx} & \; \lambda \tag{SDP} \\
    \text{s.t. } & \; W \succeq 0 \nonumber \\
    & \; W_{ii} = \sum_{(i,j) \in E} w_{ij} x_{ij} - \frac{\lambda(n-1)}{n} \quad \forall i \in V \nonumber\\
    & \; W_{ij} = W_{ji} = -w_{ij}x_{ij} + \lambda/n \quad \forall (i,j) \in E \nonumber \\
    & \; \vx \in \text{spanning tree of } G \nonumber
  \end{align}
The authors \citet{somisetty2024optimal} point out that relaxing the spanning tree constraint to a simple budget constraint ($N =|V|- 1$)
still yields a spanning tree as the optimal solution by optimality, since any disconnected graph yields a zero objective.
Thus, we use this fact in the formulation that \textsc{SCIP-SDP} solves.
Further, we use the experiment set of \citet{somisetty2024optimal} which can be found on GitHub as part of the \textsc{LaplacianOpt.jl} package \citep{LaplacianOptGitHub}, though we only consider the number of nodes from 10 to 25.

\subsection*{Values of $\mu$ and the implementation of the smoothed gradient}

The actual performance of the proposed approach depends on the choice of the smoothing parameter $\mu$.
For the runs without instance-dependent scaling, we use the schedules summarized in
\cref{tab:mu_schedules}. At each decay step, the smoothing parameter is updated as
$\mu \leftarrow \max\{\rho^{d-1}\mu,\mu_{\min}\}$, where $\rho$ denotes the decay factor and $d$ the depth of the current node.
Here, $m$ is the number of candidate experiments or edges and $n$ is the number of
vertices in the spanning-tree instances.

\begin{table}[htbp]
    \centering
    \small
    \begin{tabular}{lccc}
        \toprule
        Criterion & Initial $\mu_0$ & Minimum $\mu_{\min}$ & Decay $\rho$ \\
        \midrule
        EOD & $m/10$ & $10^{-20/m}$ & $0.9$ \\
        AGC connected
            & $m/200$
            & $\begin{cases}10^{-300/m},&m\in\{80,100\},\\10^{-400/m},&m\in\{150,200\}\end{cases}$
            & $\begin{cases}0.9,&m\in\{80,100\},\\0.7,&m\in\{150,200\}\end{cases}$ \\
        AGC disconnected
            & $m/100$
            & $\begin{cases}10^{-300/m},&m\in\{80,100\},\\10^{-400/m},&m\in\{150,200\}\end{cases}$
            & $0.9$ \\
        ACST/ACSTS
            & $n/25$
            & $\max\{10^{-4},10^{-\min\{80,m\}/m}\}$
            & $0.8$ \\
        \bottomrule
    \end{tabular}
    \caption{Fixed smoothing schedules used for the different criteria.}
    \label{tab:mu_schedules}
\end{table}
As in seen in \cref{fig:Trajectory_smallest_eigenvalue}, the actual scale of the eigenvalues can differ quite
a lot between the different problem classes. Thus, we have another $\mu$ schedule which depends on the magnitude of the minimum eigenvalue
\begin{align}
    \label{eq:mu_schedule_adaptive}
    \mu_0 = 0.15 * \hat{\mu} \quad \mu_{\min} = 0.003 * \hat{\mu}
\end{align}
where $\hat{\mu}$ is the average minimum eigenvalue computed over a few sample designs.
The decay factor remains as specified in \cref{tab:mu_schedules}.

Regarding the stable implementation of the smoothed gradient, the exponent terms in the gradient can grow
large quickly even for moderate sizes of $\lambda$. For a numerical stable implementation, 
we define $K = \sum_{j=1}^n e^{ -\lambda_j / \mu }$ and $\omega_i = \frac{e^{ -\lambda_i / \mu }}{K}$. Then,
the gradient entries can be written as
\begin{align*}
    \left[\nabla g_{\mu}(\vx)\right]_k &= \sum_{i=1}^n \omega_i \innp{\vu_i}{\vm_k}^2. \\
    \intertext{Using the exponent rules, the $\omega_i$ can be computed as}
    \omega_i &= \exp(-\lambda_i / \mu - \log(K)).
\end{align*}
Julia's LogExpFunctions.jl package provides a logsumexp function that can be used to compute the gradient
entries in a numerically save manner.

\subsection*{Results}
The computations were performed on a machine with an Intel(R) Xeon(R) Gold 6342 CPU @ 2.80GHz and 512 GB of RAM.
The package versions on the Julia side are \textsc{Boscia.jl} v0.2.10 (branch \texttt{smoothing-mode}), \textsc{SCIP-SDP} version of the 
\texttt{main} branch on the 31st of March 2026, \textsc{SCIP} v10.0.1, \textsc{SCIP.jl} v0.12.8, \textsc{CombinatorialLinearOracles.jl} v0.1.5 and \textsc{FrankWolfe.jl} v0.6.2.
\textsc{SCIP-SDP} is configured with \textsc{SCIP} v10.0.1 and \textsc{Mosek} v10.2\footnote{Due to a recent and necessary bug fix in SCIP-SDP, we use commit \texttt{bf9caa0468d59b27c5f1b9becc8772e976747eb4} on the \texttt{main} branch of \texttt{SCIP-SDP}.}.
For EOD and AGC, we use the \emph{Decomposition Invariant Conditional Gradient (DICG)} \citep{garber2016linear} and
for ACST, the \emph{Blended Pairwise Conditional Gradient (BPCG)} \citep{tsuji2021sparser}.
The source code of the experiment is publicly available on Github, \url{https://github.com/ZIB-IOL/OptimalDesignWithBoscia/tree/e-optimal-arxiv}.

A summary of the results over all problems can be seen in \cref{tab:all_criteria_summary}.
More detailed results can be found in \cref{tab:by_dimension_(correlated),tab:by_dimension_(independent),tab:avg_by_dimension_connected,tab:avg_by_dimension_disconnected,tab:spanning_tree} in the \cref{app:comp_appendix}.
The results indicate that the proposed approach is particularly effective when additional structure constraints are present and for large dimensional problems.
On the instances with simpler structure and moderate dimensionality, \textsc{SCIP-SDP} outperforms our approach as it has tighter relaxations.

\begin{table}[htbp]
    \centering
    \caption{Summary over all problems. The geometric mean of the time is computed with a 1s shift. The geometric mean of the relative gap is only computed for the unsolved instances of that solver. \\
    For EOD with both data types, we use the adaptive $\mu$ scaling in \cref{eq:mu_schedule_adaptive}. For AGC with a connected base graph, we use the standard $\mu$
    schedule in \cref{tab:mu_schedules} and eigenvalue-based pruning. For AGC with a disconnected base graph, we use the adaptive $\mu$ scaling, eigenvalue-based pruning
    and truncate the gradient to use half of the eigenvalue spectrum. Lastly, for ACST, we utilize the rank-based pruning and also truncate the gradient to use half of the eigenvalue spectrum.}
    \label{tab:all_criteria_summary}
    \begin{tabular}{llrrrrr}
\toprule
 Solver & Metric & EOD corr. & EOD ind. & AGC corr. & AGC ind. & ACST \\
\midrule
\multirow{3}{*}{Boscia} & \% sol. & 50.0 & 24.0 & \textbf{100.0} & 20.0 & \textbf{14.3} \\
  & time (s) & 845.9 & 1906.6 & 9.3 & 1352.5 & \textbf{2227.5} \\
  & rel. gap & 1.83e-01 & 1.80e-01 & --- & 1.21e-01 & \textbf{8.46e+00} \\
\midrule
\multirow{3}{*}{SCIPSDP (OA)} & \% sol. & \textbf{58.0} & 30.0 & \textbf{100.0} & 10.0 & 0.0 \\
  & time (s) & \textbf{217.6} & 1080.1 & \textbf{0.7} & 2370.2 & 3600.2 \\
  & rel. gap & 1.93e-01 & 2.63e-01 & --- & 1.28e-01 & 7.47e+15 \\
\midrule
\multirow{3}{*}{SCIPSDP (B\&B)} & \% sol. & 50.0 & \textbf{54.0} & \textbf{100.0} & \textbf{45.0} & 0.0 \\
  & time (s) & 736.5 & \textbf{658.8} & 50.3 & \textbf{1152.4} & 3600.0 \\
  & rel. gap & \textbf{6.64e-02} & \textbf{1.27e-01} & --- & \textbf{6.89e-02} & 2.23e+04 \\
\bottomrule
\end{tabular}

\end{table}

\cref{fig:E_correlated_scipsdp,fig:AGC_correlated_scipsdp,fig:E_independent_scipsdp,fig:AGC_independent_scipsdp} display the percentage of solved instances over time for the different solvers.
Interestingly, \textsc{SCIPSDP (B\&B)} is the slowest in the smallest instances but it catches up later.
Note that overall, \textsc{SCIPSDP (B\&B)} performs better than \textsc{SCIPSDP (OA)} which is expected as the SDP relaxation yields the tightest bounds.

\begin{figure}
    \centering
    \begin{subfigure}{0.49\textwidth}
        \centering
        \includegraphics[width=\textwidth]{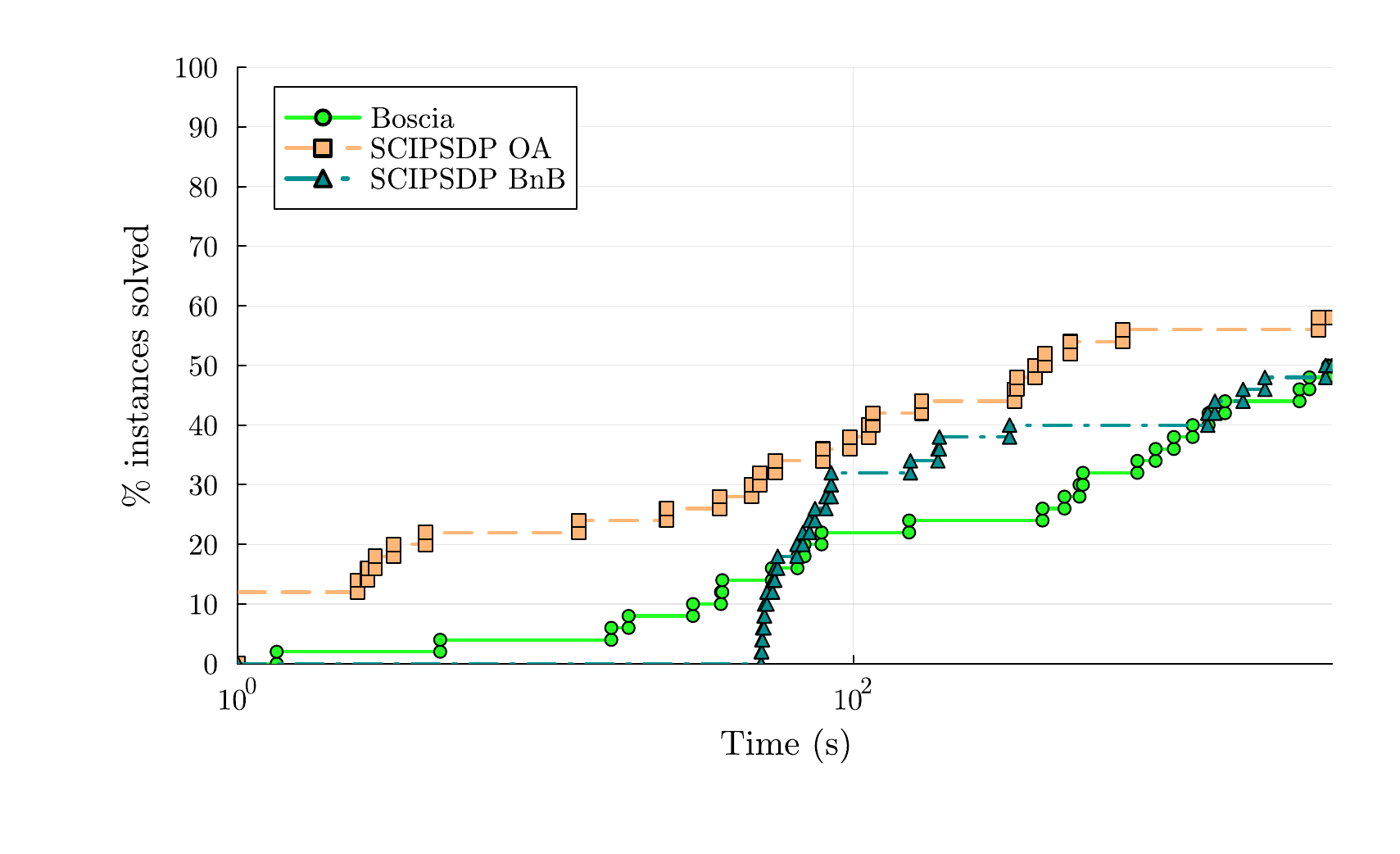}
        \caption{Termination over time for the E-Optimal Design Problem with correlated instances}
        \label{fig:E_correlated_scipsdp}
    \end{subfigure}
    \hfill
    \begin{subfigure}{0.49\textwidth}
        \centering
        \includegraphics[width=\textwidth]{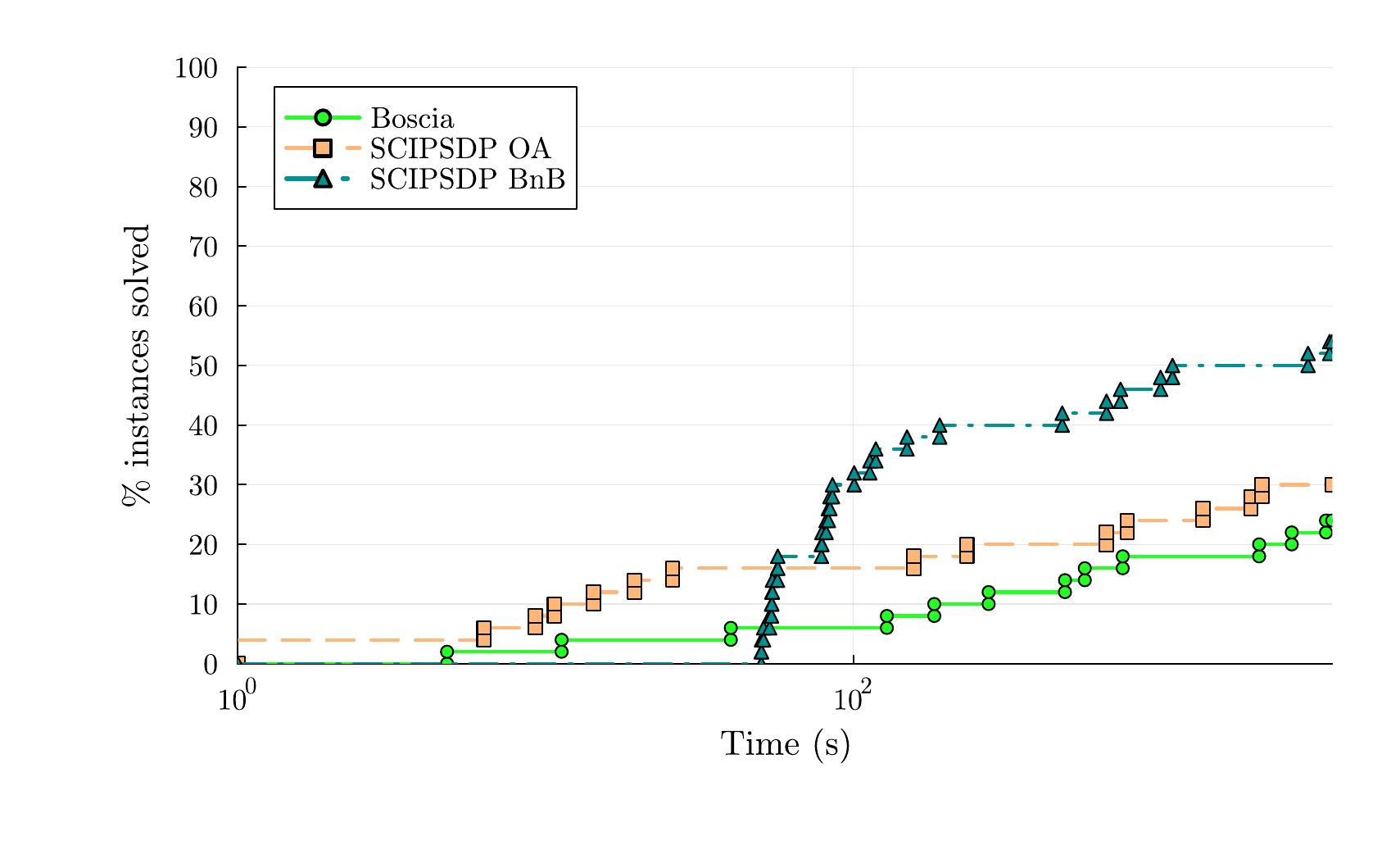}
        \caption{Termination over time for the E-Optimal Design Problem with independent instances}
        \label{fig:E_independent_scipsdp}
    \end{subfigure}
    \caption{Termination over time for the E-Optimal Design Problem under both data types.}
\end{figure}

\begin{figure}
    \centering
    \begin{subfigure}{0.49\textwidth}
        \centering
        \includegraphics[width=\textwidth]{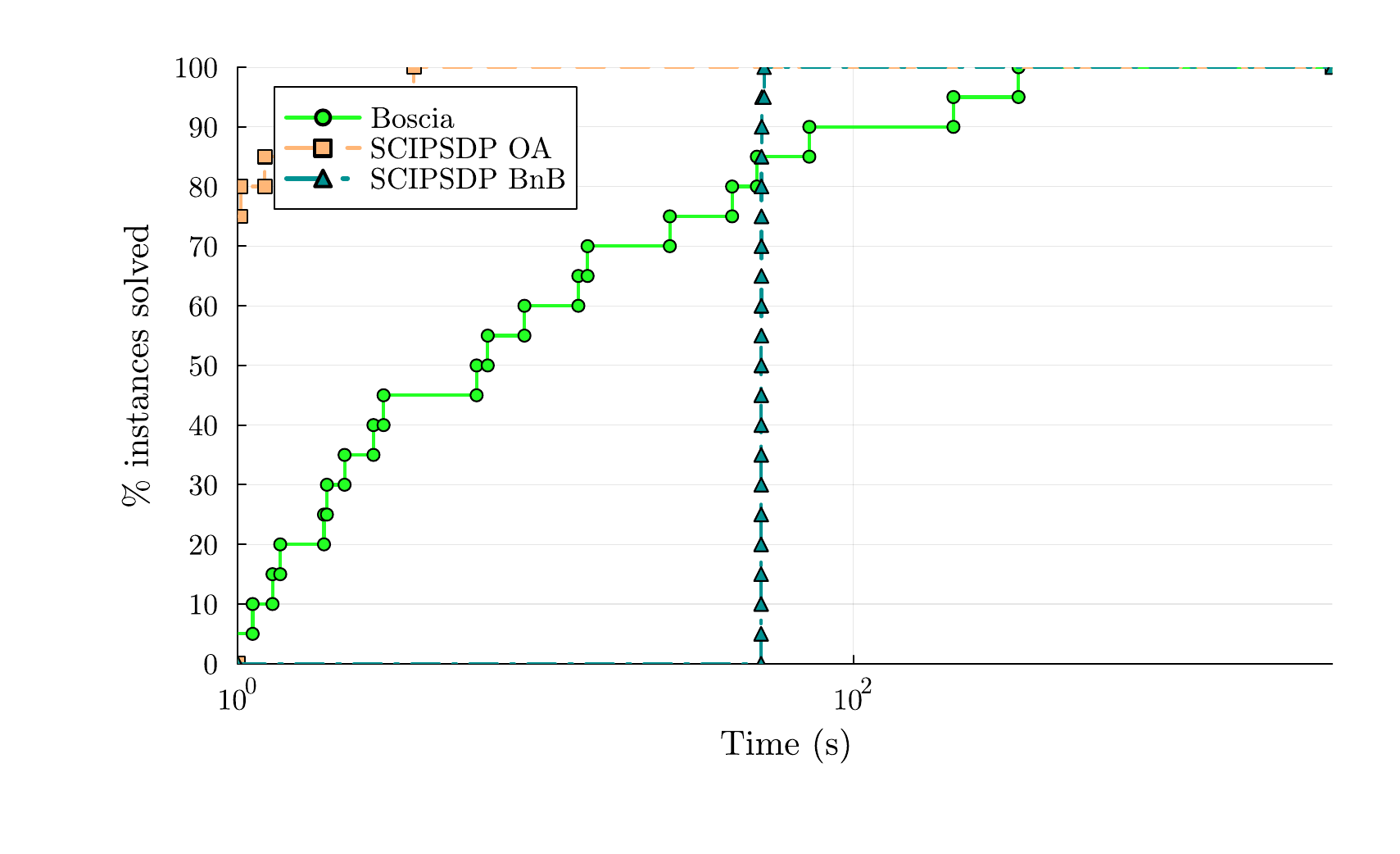}
        \caption{Termination over time for the Maximum Algebraic Connectivity Problem with connected base graphs.}
        \label{fig:AGC_correlated_scipsdp}
    \end{subfigure}
    \hfill
    \begin{subfigure}{0.49\textwidth}
        \centering
        \includegraphics[width=\textwidth]{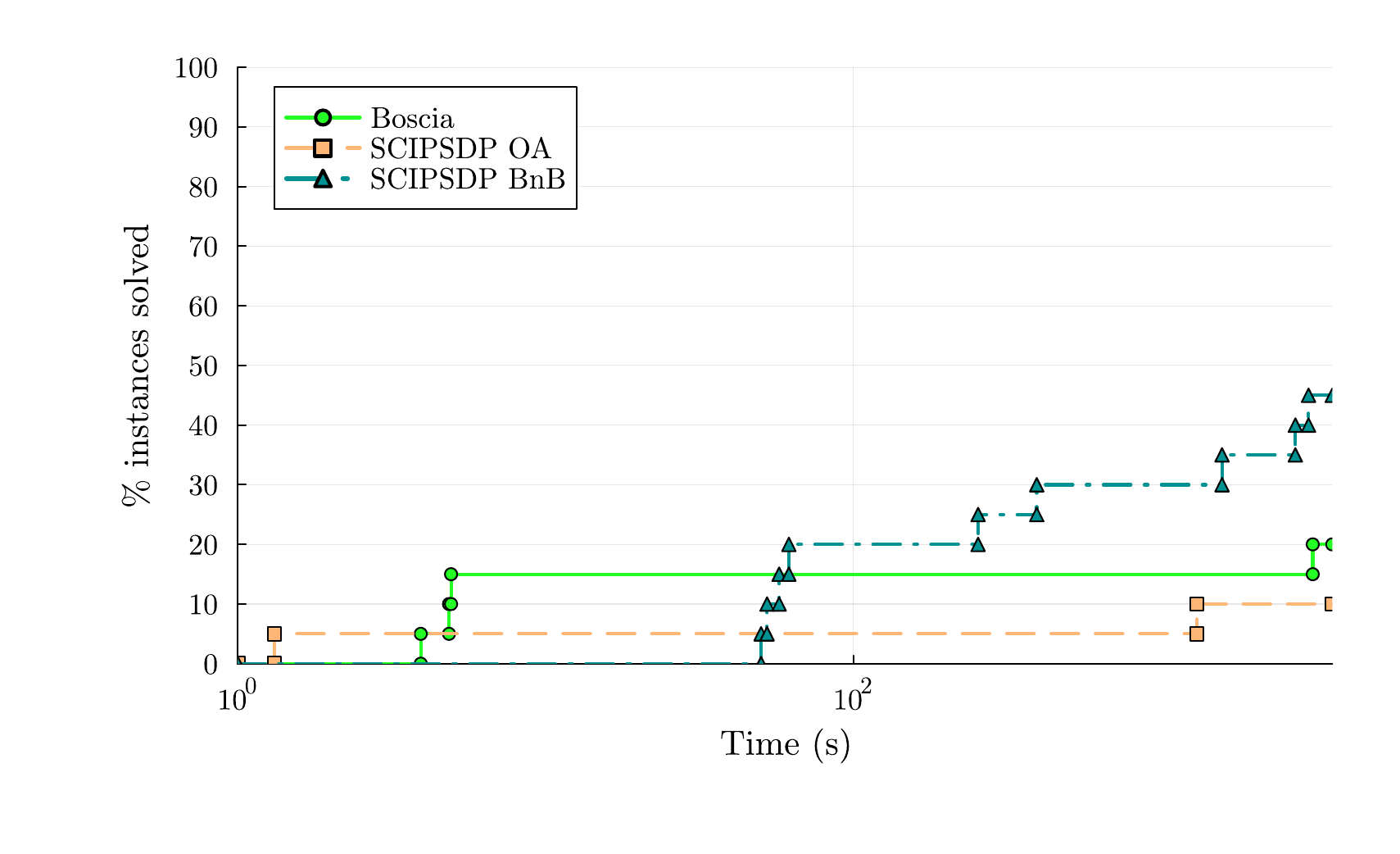}
        \caption{Termination over time for the Maximum Algebraic Connectivity Problem with disconnected base graphs.}
        \label{fig:AGC_independent_scipsdp}
    \end{subfigure}
    \caption{Termination over time for the Maximum Algebraic Connectivity Problem with connected and disconnected base graphs.}
\end{figure}

A further curiosity is that \textsc{SCIPSDP (OA)} performs better on the correlated instances of EOD and connected instances of AGC, similarly to our approach.
To understand the difference in performance for our approach, we display the trajectory of the smallest eigenvalue during the B\&B tree for a problem with 30 variables, see \cref{fig:Trajectory_smallest_eigenvalue}.
On the left side, the experiment set is correlated, on the right side, it is independent.
The black line denotes the optimal solution which is found at the root node in the correlated case and after 150 nodes in the independent case.
We can observe that more nodes can be potentially pruned in the correlated case, obviously speeding up the search.
The range of the eigenvalue is also larger for the independent data.
If all experiments are independent, then there is a larger set of candidate subsets which would span the experiment space,
potentially all providing similar information gains.
This impedes the proving of the optimal solution.
In fact, our approach finds the optimal solution early on and takes most of its time to prove its optimality.

\begin{figure}
    \centering
    \includegraphics[width=0.9\textwidth]{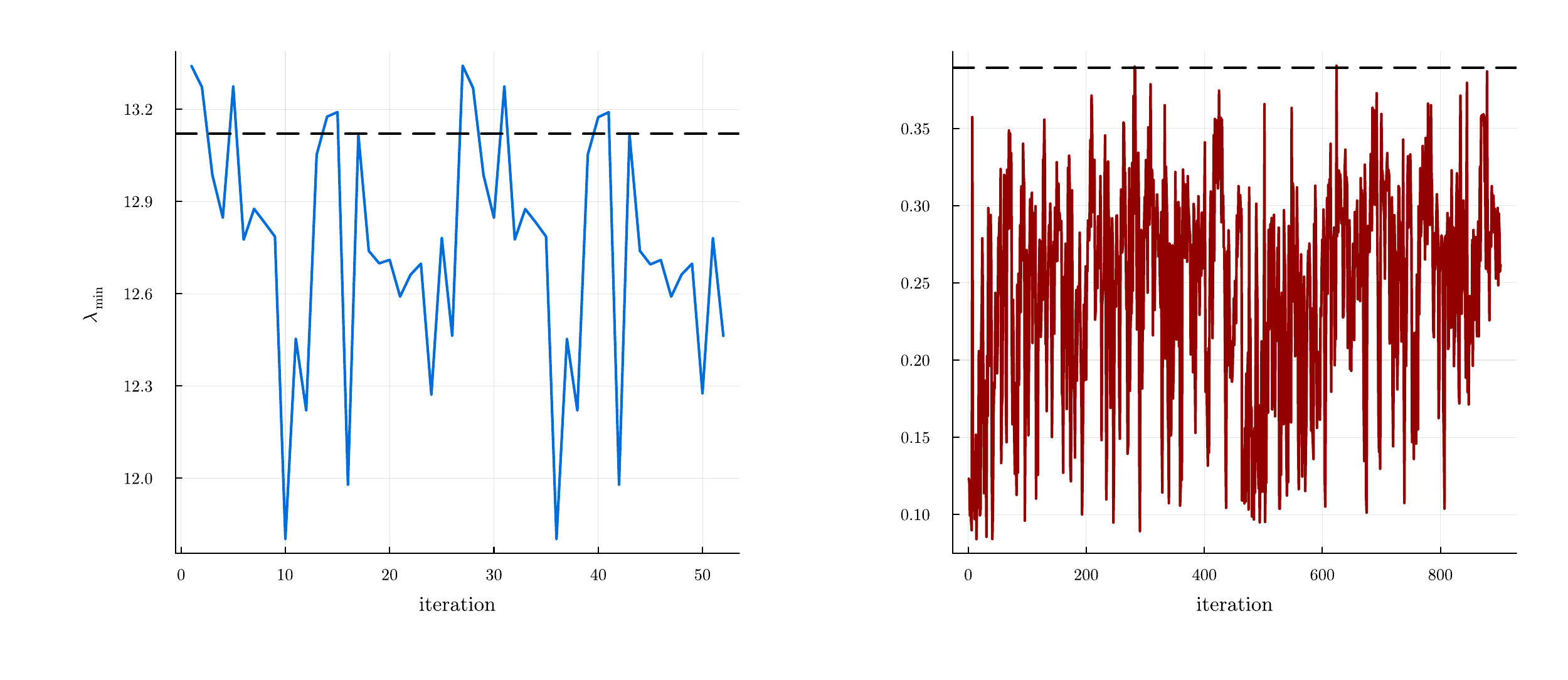}
    \caption{Sample trajectory of the smallest eigenvalue during the B\&B tree for a problem
    with 30 variables. On the left side, the experiment set is correlated, on the right side, 
    it is independent. The black line denotes the optimal solution.}
    \label{fig:Trajectory_smallest_eigenvalue}
\end{figure}
This points to the upper bound computation being the main bottleneck. 
As previously discussed, the SDP formulation is the tightest relaxation for the problem, as illustrated by the
incumbent and dual-bound trajectories in \cref{fig:progress_bounds}.

\begin{figure}
    \centering
    \begin{subfigure}{0.49\textwidth}
        \centering
        \includegraphics[width=\textwidth]{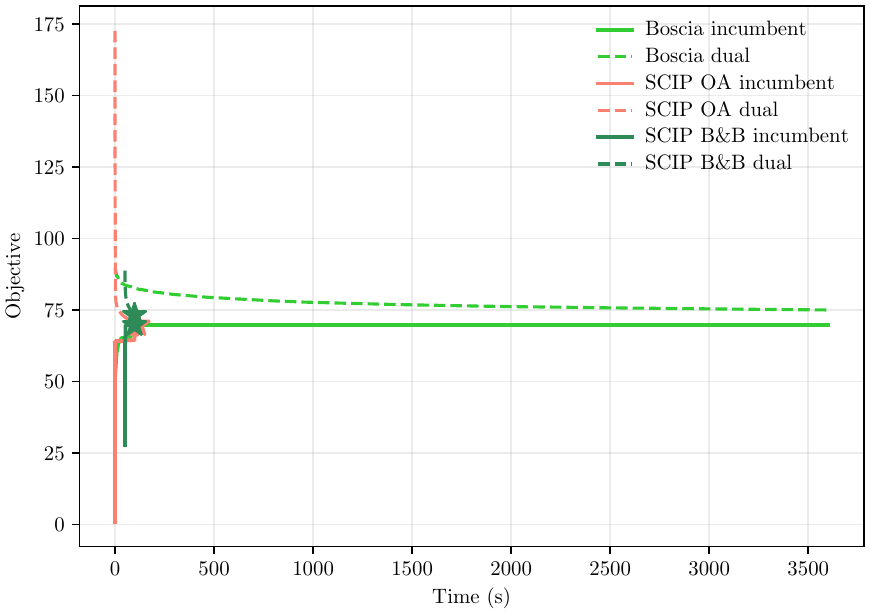}
        \caption{EOD with correlated data}
    \end{subfigure}
    \hfill
    \begin{subfigure}{0.49\textwidth}
        \centering
        \includegraphics[width=\textwidth]{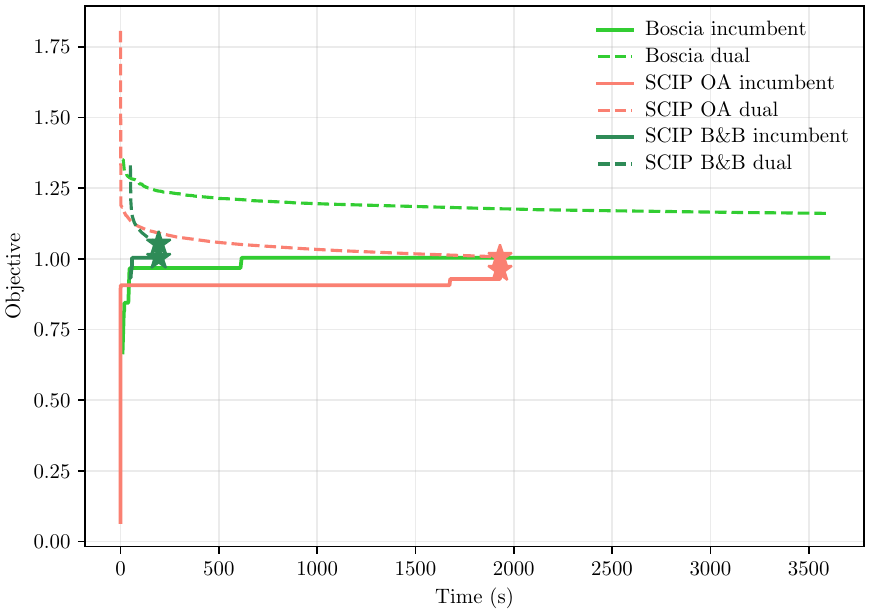}
        \caption{EOD with independent data}
    \end{subfigure}

    \begin{subfigure}{0.49\textwidth}
        \centering
        \includegraphics[width=\textwidth]{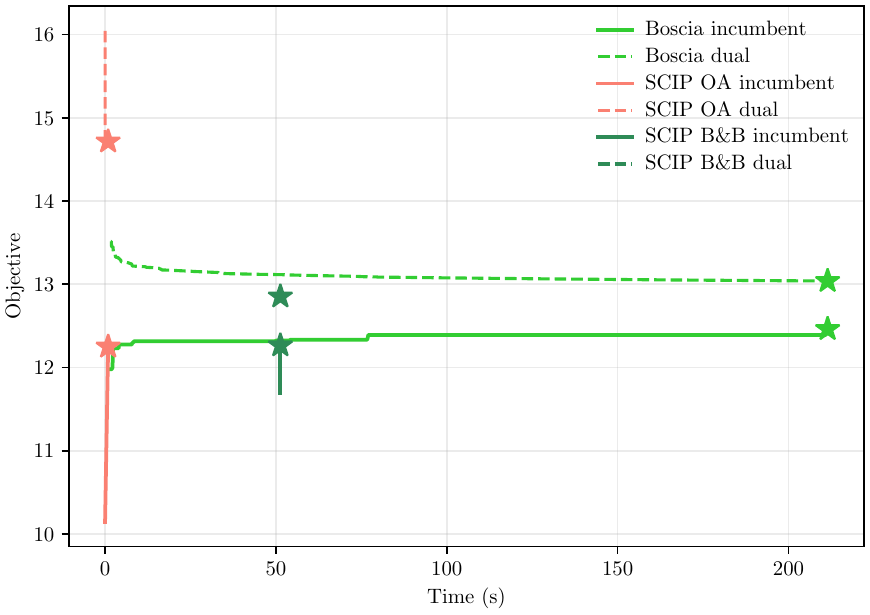}
        \caption{AGC with connected base graph}
    \end{subfigure}
    \hfill
    \begin{subfigure}{0.49\textwidth}
        \centering
        \includegraphics[width=\textwidth]{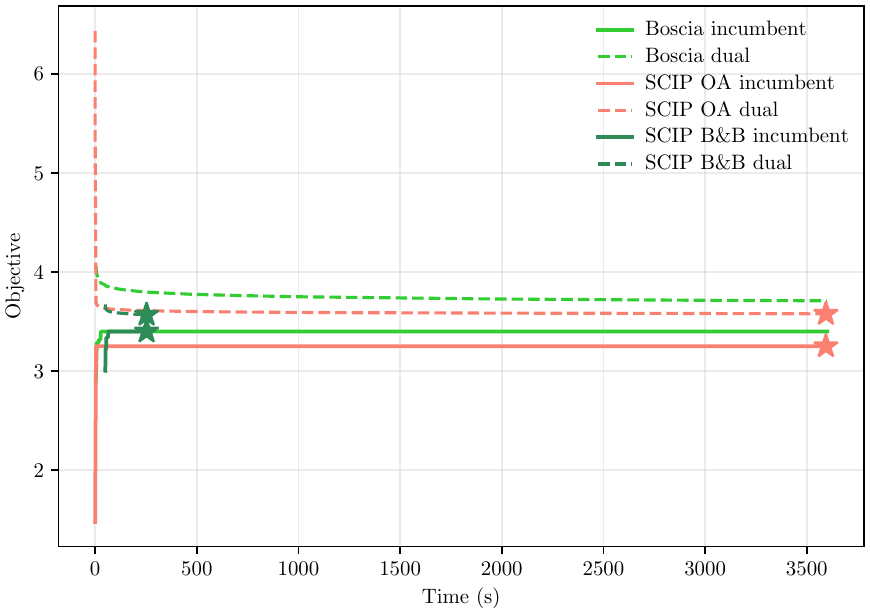}
        \caption{AGC with disconnected base graph}
    \end{subfigure}
    \caption{Progress of incumbent values and dual bounds for representative EOD and AGC instances on 80 variables and a budget of 12 experiments/edges.
    The star denotes the end of the solve for the corresponding solver.}
    \label{fig:progress_bounds}
\end{figure}
Note that Boscia finds the optimal solution early on and takes most of its time to prove its optimality.
However, the upper bound progress is much slower compared to SCIP-SDP (B\&B) and SCIP-SDP (OA).
The choice of the smoothing parameter has a huge effect on the performance speed of the approach. 
Our relaxation quality depends heavily on the magnitude of the smoothing parameter as showcased in \cref{fig:mucomparison}.
The figure displays the progress of the relative gap over the solve for different smoothing parameters for a chosen instance of EOD and AGC.
While in general a smaller initial smoothing parameter and faster decay rate is beneficial, this does not seem to be a universal rule.
\begin{figure}
    \centering
    \begin{subfigure}{0.49\textwidth}
        \centering
        \includegraphics[width=\textwidth]{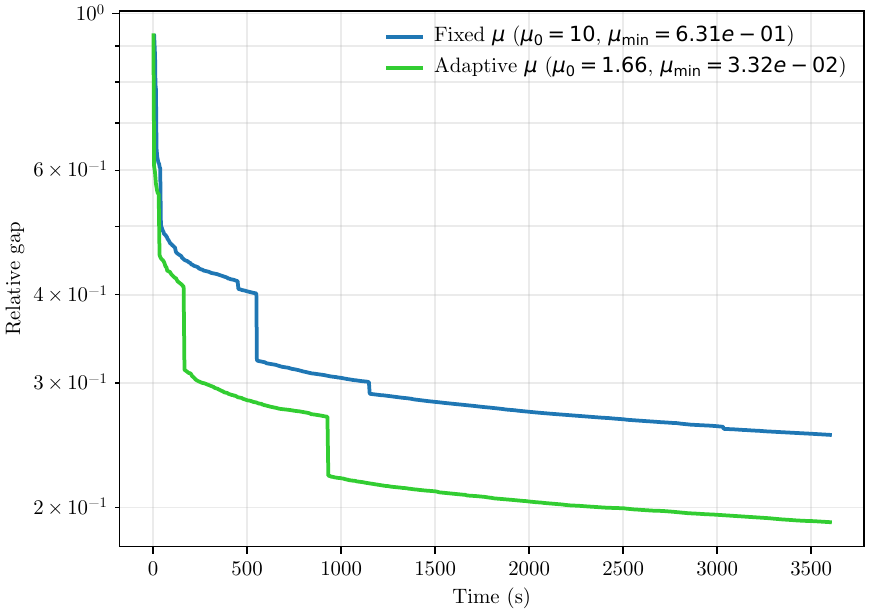}
        \caption{EOD with correlated data,}
    \end{subfigure}
    \hfill
    \begin{subfigure}{0.49\textwidth}
        \centering
        \includegraphics[width=\textwidth]{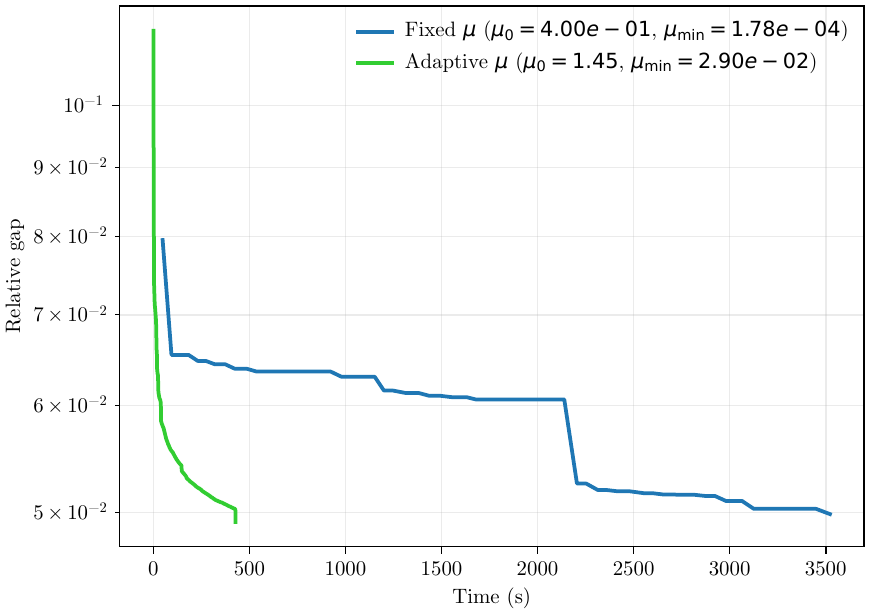}
        \caption{AGC with connected base graph.}
    \end{subfigure}
    \caption{Trajectory of the relative gap under different smoothing parameters for EOD and AGC.}
    \label{fig:mucomparison}
\end{figure}

The image shifts in larger dimensions as seen in \cref{fig:rel_gap_large_dim}. 
While SCIP-SDP (B\&B) still achieves the smallest relative gap, it fails to solve the root node and provide a dual and primal bound for instances with dimensions larger than or equal to 5000.
Note that Boscia yields a relatively similar gap compared to SCIP-SDP (B\%B) but the relative gaps reported by SCIP-SDP (OA) are much bigger and more variable,
especially as the dimension increases. 
In summary, Boscia is more robust and scales better than SCIP-SDP (B\&B) while at the same time achieving a similar gap compared to SCIP-SDP (OA).
\begin{figure}
    \centering
    \begin{subfigure}{0.49\textwidth}
        \centering
        \includegraphics[width=\textwidth]{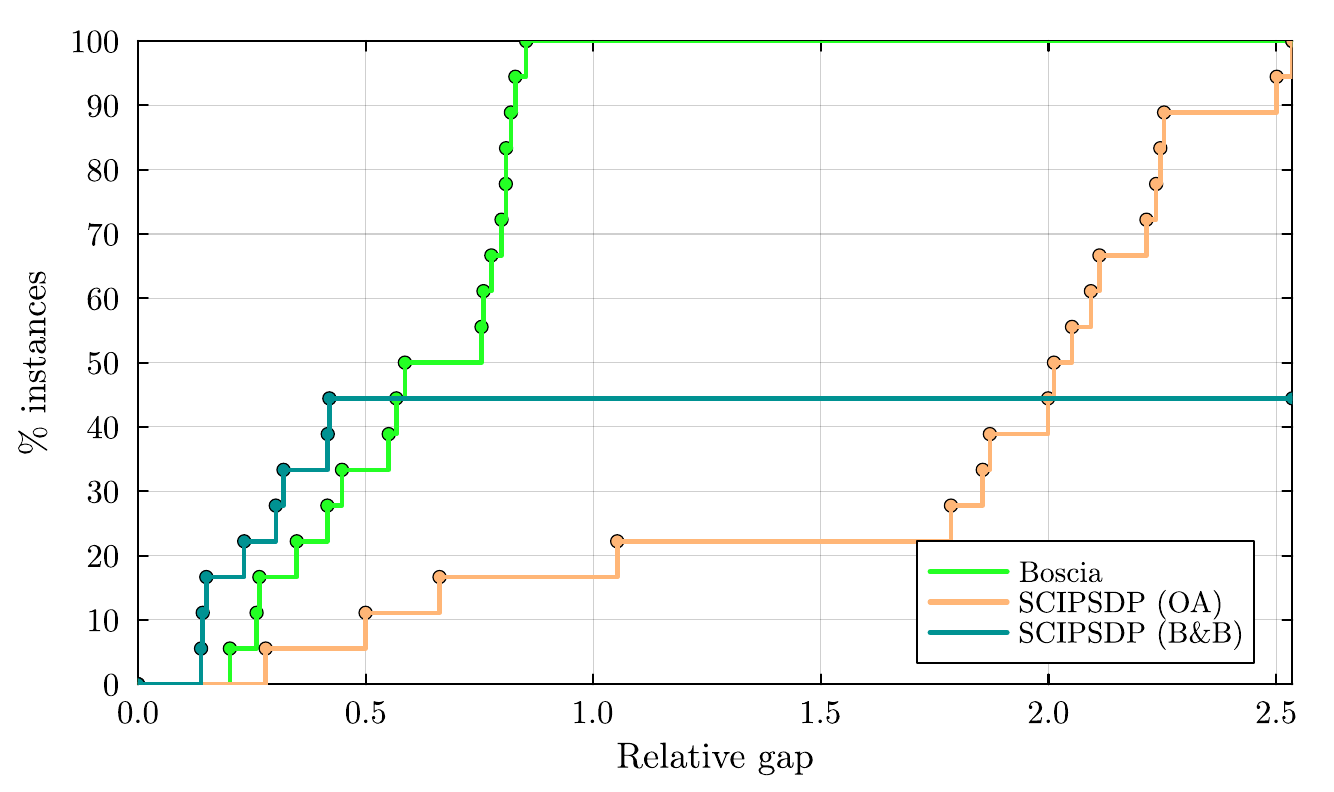}
        \caption{Correlated data and budget $N=1.5n\log(n)$}
    \end{subfigure}
    \hfill
    \begin{subfigure}{0.49\textwidth}
        \centering
        \includegraphics[width=\textwidth]{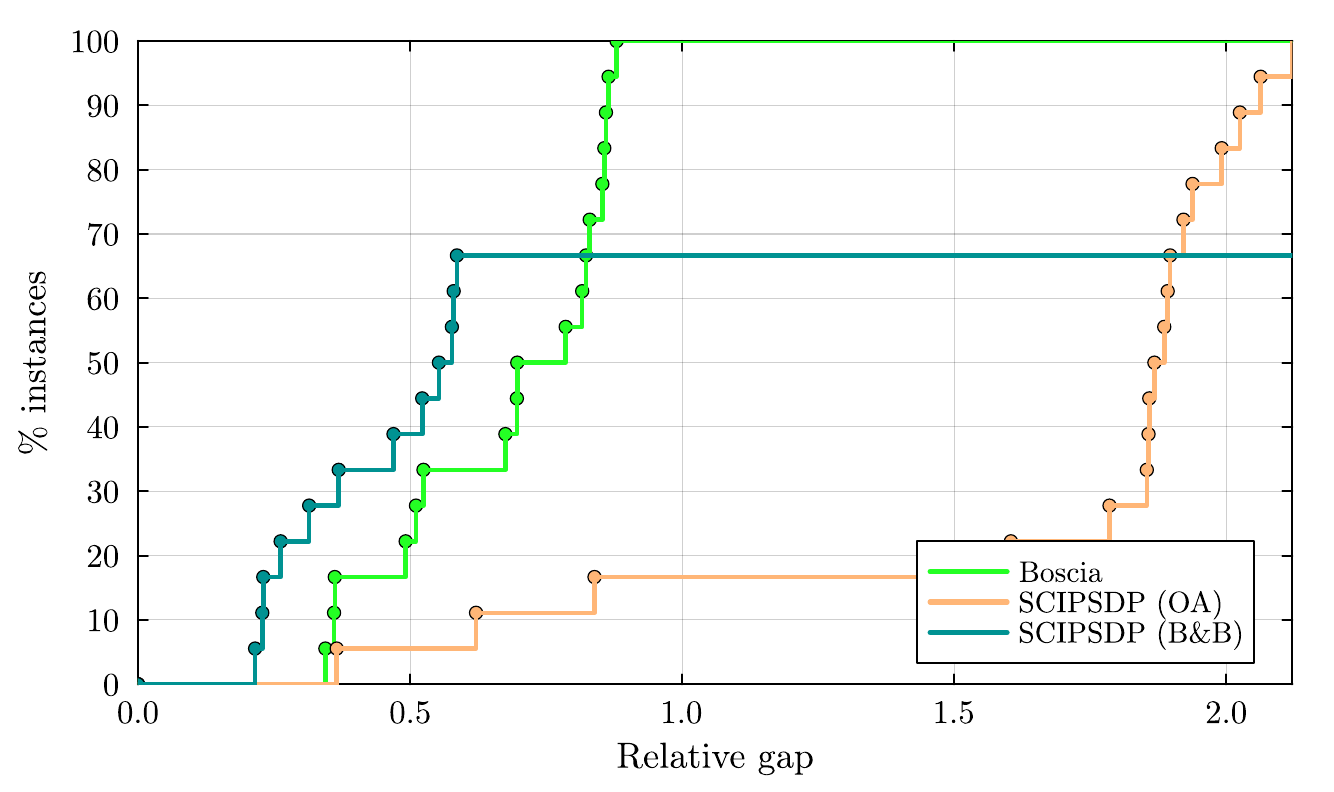}
        \caption{Independent data and budget $N=1.5n\log(n)$}
    \end{subfigure}

    \begin{subfigure}{0.49\textwidth}
        \centering
        \includegraphics[width=\textwidth]{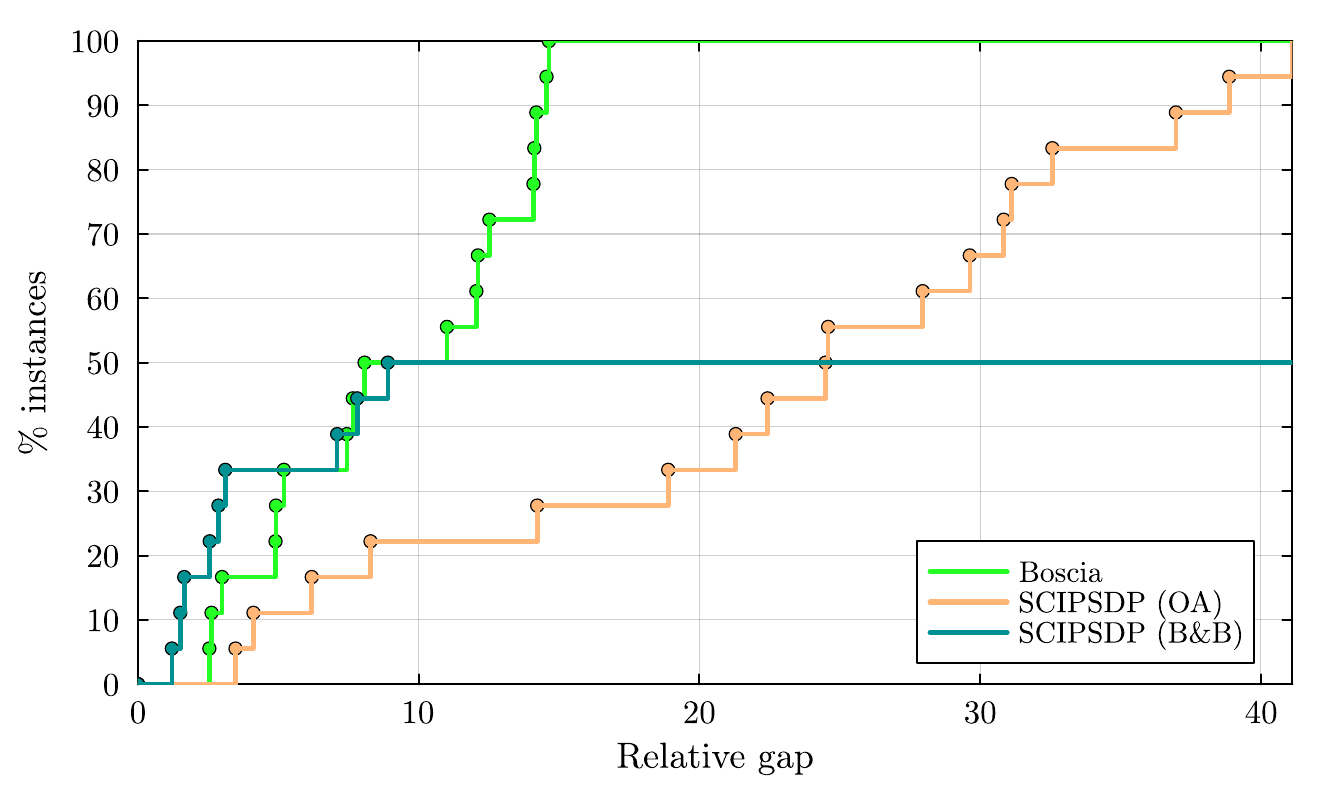}
        \caption{Correlated data and budget $N=1.5n$}
    \end{subfigure}
    \hfill
    \begin{subfigure}{0.49\textwidth}
        \centering
        \includegraphics[width=\textwidth]{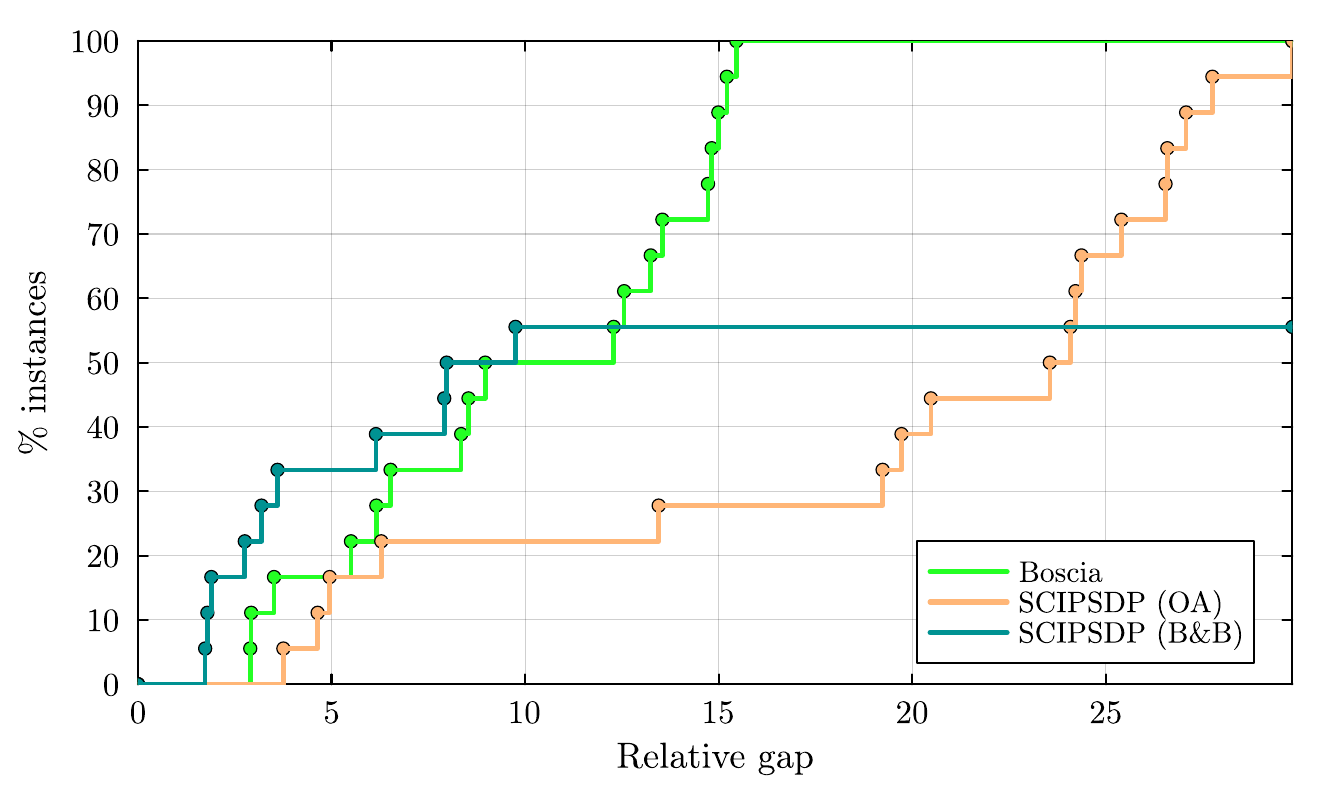}
        \caption{Independent data and budget $N=1.5n$}
    \end{subfigure}
    \caption{Cumulative distribution of the relative gaps achieved by the three solvers after 1 hour of runtime on EOD with 
    dimensions from 500 to 10000; read as "Solver S solves at least y\% of the instances with a relative gap equal to or less than x\%".
    The SCIP-SDP (B\&B) line stops below the 100 \% on the y-axis because even the root node computation cannot be completed within 1 hour and no 
    valid upper bound and incumbent can be computed.}
    \label{fig:rel_gap_large_dim}
\end{figure}

Another case in which our proposed approach is superior is the Algebraic Connectivity Spanning Tree problem.
It can explicitly encode the spanning tree constraint and LMO can be efficiently computed, for example with Kruskal's algorithm.
Note that we also tested the knapsack formulation of the problem, see \cref{fig:ACSTcomparison}, which performs worse (though it can solve at least one instance
in contrast to the two SCIP-SDP approaches).
We can conclude that the knapsack formulation is looser which in turn also explains why both \textsc(SCIPSDP) variants were 
not able to solve any instance.
This highlights that our approach is well-suited for problems with more complex structure constraints as we can decouple these from the SDP constraint.

\begin{figure}
    \centering
    \includegraphics[width=0.7\textwidth]{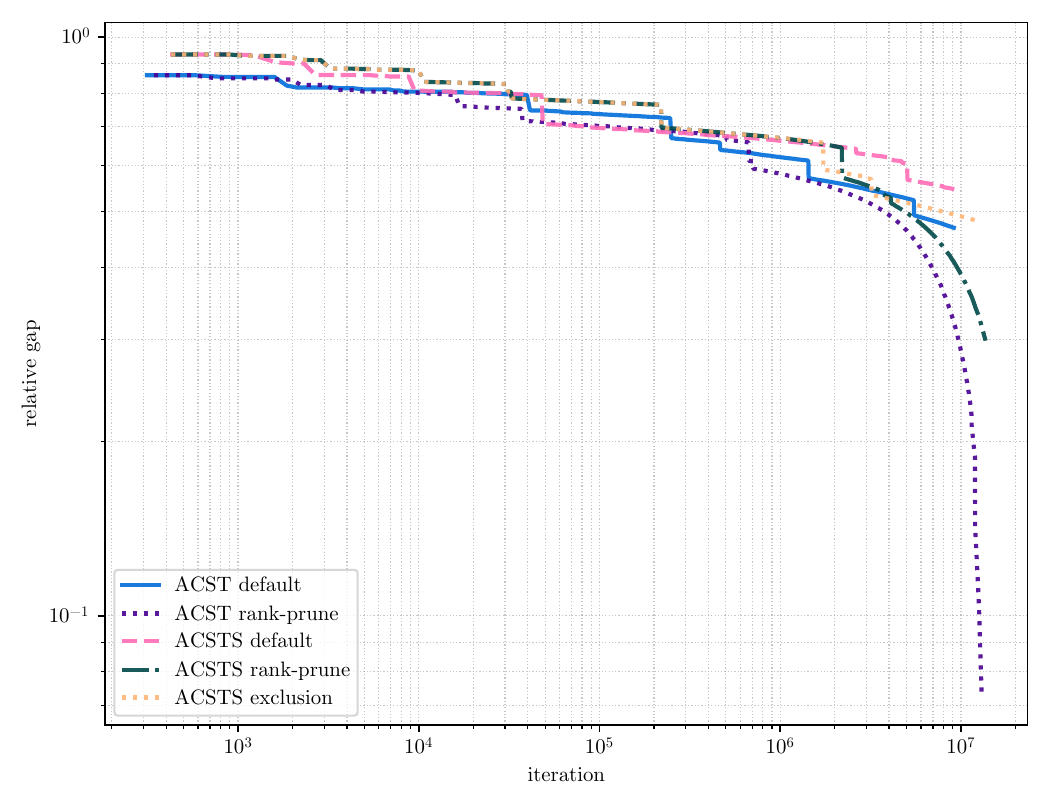}
    \caption{Trajectory of the relative gap for an instance with 12 nodes and 66 potential edges of the Algebraic Connectivity Spanning Tree problem.}
    \label{fig:ACSTcomparison}
\end{figure}

In \cref{fig:ACSTcomparison}, we also see the effect of the rank-based pruning strategy.
For this particular problem, the rank-based pruning is very reasonable since the budget is very tight and any branching
can quickly lead to a tree where a spanning tree of the required size cannot constructed anymore.

Further ablation studies can be found in the \cref{app:comp_appendix}.
The fixing routines do not exhibit any benefits, they are relatively expensive to compute and rarely lead to any actual fixings.
The pruning strategies are more effective, especially the eigenvalue-based pruning for AGC. 
The eigenvalue-based pruning does not strongly impact the number of solved instances but it improves the 
speed of the approach which points to it indeed pruning nodes but only deeper in the B\&B tree.
The rank-based pruning is not as effective for EOD and AGC, likely due to their large budget.
It is unlikely that a current node problem cannot achieve full rank, unless very deep in the B\&B tree.
If the experiments are independent of each other, any large enough subset of them will achieve full rank, further limiting the benefit of the rank-based pruning.
Truncated the gradient too much is detrimental for the performance as can be seen in \cref{tab:smoothing_ablation}. 
Only using a third of the eigenvalue spectrum yields consistently worse results. 
The correction term in the FW gap is growing too large which leads to a slower convergence.
Utilizing half the spectrum can yield better results for some problem instances.

\section{Conclusion}

We proposed a new approach for the discrete constrained eigenvalue maximization problem, combing entropic smoothing with first-order
methods within a branch-and-bound framework.
The convergence of the approach is analyzed and convergence guarantees are established even with inexact gradient information. 
Additional algorithmic improvements are derived, from fixing to pruning strategies.
The proposed approach is clearly outperforming state-of-the-art Mixed-Integer SDP solvers for problems with complex structural constraints and in large 
dimensional settings where it is more robust than SCIP-SDP relying on interior-point methods for the SDP relaxation.
Investigations into upper bound improvements and a more systematic study of the effect of the smoothing parameter and reduced spectrum in the gradient would further
improve the performance of the approach.

\section*{Acknowledgments}

Research reported in this paper was partially supported through the Research Campus Modal funded by the German Federal Ministry of Education and Research (fund numbers 05M14ZAM,05M20ZBM),
the Deutsche Forschungsgemeinschaft (DFG, German Research Foundation) under Germany's Excellence Strategy, the Berlin Mathematics
Research Center MATH+ (EXC-2046/1, EXC-2046/2, project ID: 390685689), and by the French ANR through the MIAI Cluster @ Grenoble (reference ANR-23-IACL-0006).

\section*{Statements and Declarations}

The authors declare no conflict of interest.

\bibliographystyle{icml2021}
\bibliography{references}

\newpage
\appendix
\onecolumn

\section{Proofs of the main results}\label{app:proofs}

This appendix contains the proofs of the results in the main text and supporting lemmas.

\cref{prop:eoptgeneric} identifies a strictly feasible, trace-bounded SDP with an instance of the constrained minimum-eigenvalue template.
\begin{proof}[Proof of \cref{prop:eoptgeneric}]
We introduce a single scalar variable $y$ and the matrix:
\begin{align*}
    Z = \frac{1}{\tau} \begin{bmatrix*}
        Y & 0 \\
        0 & y
    \end{bmatrix*},
\end{align*}
leading to the following re-write of the problem:
\begin{align*}
\min_{Z} \; & \innp{\hat C}{Z} & \\
\mathrm{s.t.} \; & \innp{\hat{A}_i}{Z} = b_i \;\;\forall i \in [m] \;\; & (x_i) \nonumber \\
& \innp{E_{j}}{Z} = 0 \;\; \forall j \in [n] & (\lambda_j) \\
& \mathrm{Tr}(Z) = 1 \;\; & (\alpha) \\
& Z \in \SymMat_+^{n+1},&
\end{align*}
where $E_j$ is a matrix with ones only at entries $(n+1,j)$ and $(j,n+1)$ and zero everywhere else,
so that $\innp{E_{j}}{Z}$ selects the $j$-th entry of the last column and of the last row,
$\hat{C}$ and $\hat{A}_i$ correspond to matrices $\tau C$, $\tau A_i$ respectively, augmented with a last row and a last column of zeros.
This derivation is quite standard and highlights that one can assume without loss of generality that the optimization problem is over the spectraplex.
The dual variables we associate with each constraint are indicated in parentheses.
The dual problem is:
\begin{align*}
\max_{x, \lambda, \alpha}\; & \sum_i b_i x_i + \alpha \\
\mathrm{s.t.} \; & C - \sum_{i\in[m]} A_i x_i - \sum_{j\in[n]} E_j \lambda_j \succeq \alpha I,
\end{align*}
which can be re-written as:
\begin{align*}
\max_{x, \lambda}\; & \sum_i b_i x_i + \lambda_{\min}\left(C - \sum_{i\in[m]} A_i x_i - \sum_{j\in[n]} E_j \lambda_j\right).
\end{align*}
By the strict feasibility assumption, strong duality holds, both the primal and dual problems attain the optimal value and yield a primal-dual pair.
The dual problem can be cast to the generic template of \eqref{eq:EOptGeneral} with the appropriate affine map $M(\cdot)$.
\end{proof}

\cref{lemma:gradient} expresses the gradient of the smoothed objective.
\begin{proof}[Proof of \cref{lemma:gradient}]
Let $f_{\mu}(X) = -\mu \log\left(\sum_{i=1}^n e^{-\lambda_i(X) / \mu} \right) + \mu \log n$.
The rank-one eigenvector outer product is a subgradient of the corresponding eigenvector w.r.t.~the matrix, i.e., $\partial \lambda_i(X) / \partial X = \vu_i \vu_i^\intercal$.
Through the chain rule and differentiability of spectral functions (see, e.g., \citet{lewis1996derivatives}),
we derive
\begin{align*}
\nabla f_{\mu}(X) = \sum_{i=1}^n \frac{e^{ -\lambda_i / \mu }}{\sum_{j=1}^n e^{ -\lambda_j / \mu }} \vu_i \vu_i^\intercal.
\end{align*}
Furthermore, $\partial M / \partial x_k = \vm_k \vm_k^\intercal$,
by using the chain rule on the composition $g_\mu(\vx) = f_{\mu}(M(\vx))$,
we obtain
\begin{align*}
\left[\nabla g_{\mu}(\vx)\right]_k =& \sum_{i=1}^n \frac{e^{ -\lambda_i / \mu }}{\sum_{j=1}^n e^{ -\lambda_j / \mu }} \innp{\vm_k \vm_k^\intercal}{\vu_i \vu_i^\intercal} \\
 =& \sum_{i=1}^n \frac{e^{ -\lambda_i / \mu }}{\sum_{j=1}^n e^{ -\lambda_j / \mu }} \innp{\vu_i}{\vm_k}^2.
\end{align*}
\end{proof}

\begin{lemma}[\citet{nesterov2007smoothing}; Section 4]\label{lemma:nesterovsmoothness}
    The function $f_{\mu}$ is $L$-smooth with Lipschitz constant $L=1/\mu$.
\end{lemma}

\cref{th:FrankWolfeConvergence} transfers the Frank-Wolfe rate for the smoothed objective to a convergence guarantee for the original minimum-eigenvalue objective.
\begin{proof}[Proof of \cref{th:FrankWolfeConvergence}]
By standard Frank-Wolfe \citep{braun2022conditional}, we have
\begin{align*}
    g_\mu(\vx_t) - g_\mu(\vx_\mu^*) \leq \frac{2D^2 \frac{\norm{M}_\mathrm{op}^2}{\mu}}{t+3},
\end{align*}
where $D$ is the diameter of $\Xset$.
For the E-optimal experiment design, we can specify the value of the diameter $D=\sqrt{2N}$ for the hypersimplex $\Xset$ achieved by two extreme points that do not share any support index.
We can also explicit the dependence on the problem data $\norm{M}_\mathrm{op} = \norm{A}_2^2$.
As for the primal gap with respect to the original objective,
\begin{align*}
    g(\vx^*) -  g(\vx_t)  &\leq g_\mu(\vx^*) - g(\vx_t)  \\
    &\leq g_\mu(\vx_\mu^*) - g(\vx_t  \\
    \intertext{where $\vx_\mu^*$ is an optimal solution of the smoothed problem. Continuing}
    g(\vx^*) - g(\vx_t)  &\leq g_\mu(\vx_\mu^*) - g_\mu(\vx_t) + g_\mu(\vx_t) - g(\vx_t) \\
    &\leq \mu \log n + g_\mu(\vx_\mu^*) - g_\mu(\vx_t)  \\
    &\leq \mu \log n + \frac{2D^2\norm{M}_\mathrm{op}^4 }{\mu(t+3)} . \\
    &\leq \frac{\epsilon}{2} + \frac{2D^2 \norm{M}_\mathrm{op}^4 }{\mu(t+3)} . \\
    \intertext{We may increase the right-hand fraction by dropping the 3. Ultimately, we want that $g(\vx) - g^* \leq \epsilon$, so the right hand fraction has to be less than $\epsilon/2$. Thus, we get}
    g(x^*) - g(x_t)  &\leq \norm{M}_\mathrm{op} D \sqrt{\frac{8 \log n}{t}} .
\end{align*}
This concludes the proof.
\end{proof}

The rank rule prunes a node whenever the fixed and remaining rank contributions cannot produce a full-rank matrix.
\begin{proof}[Proof of \cref{prop:rankbasedpruning}]
By subadditivity of the rank and the fact that each selected entry contributes a rank-one matrix, and at most $\rank(M^l)$ if all of them were selected,
we have
\begin{align*}
\rank \left(C_l+\sum_{i \in \mathcal{S}_l^f} x_i \vm_i \vm_i^\intercal \right) \leq \rank(C_l) + \min\{N_v, \rank(M^l)\} < n,
\end{align*}
making the resulting matrix singular for any completion, the node can be pruned.
\end{proof}

\cref{prop:interlacing_bound} allows pruning of a node by bounding the best achievable minimum eigenvalue for the node problem with an eigenvalue of the partial matrix at the node.
\begin{proof}[Proof of \cref{prop:interlacing_bound}]
Fix any $\vx \in \Xset_l$ and write
$\Delta_{\vx} := \sum_{i \in \mathcal{S}_l^f} x_i \vm_i \vm_i^\intercal$,
so that the matrix of interest is $C_l + \Delta_{\vx}$.
Since $\vx \in \{0,1\}^{|\mathcal{S}_l^f|}$ with
$\sum_{i \in \mathcal{S}_l^f} x_i = N_l$, the matrix $\Delta_{\vx}$ is a sum of
$N_l$ rank-one positive semi-definite matrices.
Let $W := \mathrm{ker}(\Delta_{\vx})$, so that $\mathrm{dim}(W) \geq n - N_l$.
Let $E$ be the subspace of $\mathbb{R}^n$ spanned by the eigenvectors of $C_l$
associated with its $N_l+1$ smallest eigenvalues
$\lambda_1(C_l), \ldots, \lambda_{N_l+1}(C_l)$ (counted with multiplicity),
so that $\mathrm{dim}(E) = N_l + 1$. Since
\[
    \mathrm{dim}(W) + \mathrm{dim}(E) \;\geq\; (n - N_l) + (N_l + 1) \;=\; n+1 \;>\; n,
\]
we have $\mathrm{dim}(W \cap E) \geq 1$.
Pick any unit vector $\vw \in W \cap E$.
By the variational characterisation of the minimum eigenvalue:
\[
    \lambda_{\min}(C_l + \Delta_{\vx})
    \;\leq\; \vw^\intercal (C_l + \Delta_{\vx})\, \vw
    \;=\; \vw^\intercal C_l \vw + \vw^\intercal \Delta_{\vx} \vw.
\]
Since $\vw \in W = \mathrm{ker}(\Delta_{\vx})$, the second term vanishes.

Let $q_1, \ldots, q_{N_l+1}$ be the orthonormal eigenvectors of $C_l$ spanning $E$,
with $C_l q_i = \lambda_i(C_l) q_i$.
Since $\vw \in E$ is a unit vector, it expands as $\vw = \sum_{i=1}^{N_l+1} \alpha_i q_i$
with $\sum_{i=1}^{N_l+1} \alpha_i^2 = 1$. Then by orthonormality:
\begin{align*}
\lambda_{\min}(C_l + \Delta_{\vx}) \leq \vw^\intercal C_l \vw = \sum_{i=1}^{N_l+1} \alpha_i^2 \lambda_i(C_l) \leq \lambda_{N_l+1}(C_l) \sum_{i=1}^{N_l+1} \alpha_i^2 = \lambda_{N_l+1}(C_l),
\end{align*}
\end{proof}

\begin{lemma}\label{lemma:lowerboundtrace}
We consider the subproblem solved at a node $l$ after fixing a subset of variables
with the notations defined in \cref{prop:rankbasedpruning}.
Let $\vm_{(j)}$ denote the $j$-th vector of $\mathcal{S}_l^f$,
sorted by non-increasing Euclidean norm, i.e., such that
\begin{align*}
\norm{\vm_{(j)}}_2^2 \geq \norm{\vm_{(j+1)}}_2^2 \;\; \forall j \in [N_v-1]
\end{align*}
Define the following quantities:
\begin{align*}
T_l^0 :=& \Tr(C_l) \\
T_l   :=& T_l^0 + \sum_{j=1}^{N_v} \norm{\vm_{(j)}}_2^2.
\end{align*}
For any $\vx \in \Xset_l$ such that $M_l(\vx) \succ 0$,
we have
\begin{align}
& \lambda_{\max}(M_l(\vx)) \leq T_l \label{eq:maxeigenboundtrace} \\
& \lambda_{\min}(M_l(\vx)) \geq \frac{1}{T_l^{n-1}}.\label{eq:mineigenboundtrace}
\end{align}
\end{lemma}
\begin{proof}
The quantity $T_l$ is a bound on $\Tr(M_l(\vx))$ for any feasible $\vx$, i.e., binary and respecting the budget $N_l$.
Let $\lambda_i$ be the $i$-th eigenvalue of $M_l(\vx)$ ordered in non-decreasing order, so that $\lambda_{\min} = \lambda_1$.
By $M_l(\vx)$ being positive definite, the trace bounds the maximum eigenvalue,
\begin{align*}
\lambda_{\max}(M_l(\vx)) \leq \Tr(M_l(\vx)) \leq T_l.
\end{align*}
Since $M_l(\vx)$ is an integer positive definite matrix, its determinant is greater or equal to one,
which yields
\begin{align*}
& 1 \leq \det(M_l(\vx)) = \lambda_{\min} \prod_{i=2}^{n} \lambda_i \leq \lambda_{\min} T_l^{n-1}.
\end{align*}
Rearranging gives the desired inequality \eqref{eq:mineigenboundtrace}.
\end{proof}

\cref{prop:simplicityeigengap} gives an eigengap condition under which the minimum eigenvalue remains simple throughout the node.
\begin{proof}[Proof of \cref{prop:simplicityeigengap}]
Fix any $\vx \in \Xset_l$ and write $X = M_l(\vx) + C$. Since $M_l(\vx) \succeq 0$,
Weyl's inequality gives:
\begin{align}
    \lambda_{2}(X) & \geq \lambda_{2}(C) + \lambda_{\min}(M_l(\vx)) \geq \lambda_{2}(C), \label{eq:weyl_upper} \\
    \lambda_{\min}(X) &\leq \lambda_1(C) + \lambda_{\max}(M_l(\vx)) \leq \lambda_{\min}(C) + T_l, \label{eq:weyl_lower}
\end{align}
where \eqref{eq:weyl_lower} uses $\lambda_{\max}(M_l(\vx)) \leq \mathrm{Tr}(M_l(\vx)) \leq T_l$ from \cref{lemma:lowerboundtrace}.
Subtracting \eqref{eq:weyl_lower} from \eqref{eq:weyl_upper}:
\begin{align*}
    \lambda_{2}(X) - \lambda_{\min}(X) \geq \gamma_C - T_l > 0,
\end{align*}
where the final strict inequality is exactly \eqref{eq:simplicity_condition}.
Hence, $\lambda_{\min}(X)$ is strictly separated from $\lambda_{2}(X)$ and is a simple eigenvalue.
\end{proof}

\begin{lemma}[Spectral tail weight and trace norm]\label{lemma:spectraltailweight}
For any $\vx \in \mathcal{X}$ and $r<n$, we have
\begin{align*}
\|P_\mu(\vx) - \widetilde{P}_\mu^{(r)}(\vx)\|_* = 2\tau_r(\vx).
\end{align*}
Moreover, $\tau_r(\vx) \leq (n-r)\, e^{-\Delta_r(\vx)/\mu}$.
\end{lemma}
\begin{proof}
Define the tail density matrix supported on the discarded eigenspace:
\begin{align*}
R_\mu^{(r)}(\vx) := \frac{1}{\tau_r(\vx)\,\rho_\mu(\vx)}
\sum_{i=r+1}^{n} e^{-\lambda_i(M(\vx))/\mu}\,
\vu_i(M(\vx))\vu_i(M(\vx))^\intercal.
\end{align*}
A direct check of coefficients on each rank-one term $\vu_i\vu_i^\intercal$ yields the convex decomposition
\begin{align*}
P_\mu(\vx) = (1-\tau_r(\vx))\, \widetilde{P}_\mu^{(r)}(\vx) + \tau_r(\vx)\, R_\mu^{(r)}(\vx),
\end{align*}
so that $P_\mu(\vx) - \widetilde{P}_\mu^{(r)}(\vx) = \tau_r(\vx)\left(R_\mu^{(r)}(\vx) - \widetilde{P}_\mu^{(r)}(\vx)\right)$.
Since $\widetilde{P}_\mu^{(r)}(\vx)$ and $R_\mu^{(r)}(\vx)$ are density matrices supported on the orthogonal subspaces $\mathrm{span}\{\vu_i\}_{i\leq r}$ and $\mathrm{span}\{\vu_i\}_{i>r}$, respectively, the eigenvalues of their difference are exactly the (nonnegative) eigenvalues of $R_\mu^{(r)}(\vx)$ together with the negatives of those of $\widetilde{P}_\mu^{(r)}(\vx)$. Hence
\begin{align*}
\left\|R_\mu^{(r)}(\vx) - \widetilde{P}_\mu^{(r)}(\vx)\right\|_* = \mathrm{Tr}(R_\mu^{(r)}(\vx)) + \mathrm{Tr}(\widetilde{P}_\mu^{(r)}(\vx)) = 2,
\end{align*}
and therefore $\|P_\mu(\vx) - \widetilde{P}_\mu^{(r)}(\vx)\|_* = 2\tau_r(\vx)$.
For the spectral gap bound, since $\lambda_i \geq \lambda_{r+1}$ for all $i \geq r+1$ and
$\rho_\mu(\vx)\geq e^{-\lambda_1/\mu}$:
\begin{align*}
\tau_r(\vx)
= \frac{\sum_{i=r+1}^{n} e^{-\lambda_i/\mu}}{\rho_\mu(\vx)}
\leq \frac{(n-r)e^{-\lambda_{r+1}/\mu}}{e^{-\lambda_1/\mu}}
= (n-r)e^{-(\lambda_{r+1}-\lambda_1)/\mu}
= (n-r)e^{-\Delta_r(\vx)/\mu}.
\end{align*}
\end{proof}

\begin{lemma}[Instance-specific oracle error]\label{lemma:oracleerror}
Consider a polytope $\mathcal{X} \subseteq [0,1]^m$.
For all $\vx \in \mathcal{X}$ and all $\vy, \vz \in \mathcal{X}$, we have
\begin{align*}
\left|\left\langle \nabla g_\mu(\vx) - \widetilde{\nabla} g_\mu^{(r)}(\vx),\; \vy - \vz \right\rangle\right| \leq \delta_r(\vx) := 2\|M\|_{\mathrm{op}}^2\, \tau_r(\vx).
\end{align*}
\end{lemma}
\begin{proof}
Using the gradient expression from \cref{lemma:spectraltailweight} and linearity of $M$:
\begin{align*}
\left\langle \nabla g_\mu(\vx) - \widetilde{\nabla} g_\mu^{(r)}(\vx),\, \vy - \vz\right\rangle &= \sum_{k=1}^m (y_k - z_k)\left\langle P_\mu(\vx) - \widetilde{P}_\mu^{(r)}(\vx),\, \vm_k \vm_k^\intercal\right\rangle \\
&=\mathrm{Tr}\!\left[\left(P_\mu(\vx) - \widetilde{P}_\mu^{(r)}(\vx)\right) \underbrace{\sum_{k=1}^m (\vy_k-\vz_k) \vm_k \vm_k^\intercal}_{= M(\vy) - M(\vz)}\right].
\end{align*}
Applying the Hölder inequality for the trace and operator norms yields
\begin{align*}
\left|\mathrm{Tr}\!\left[\left(P_\mu(\vx) - \widetilde{P}_\mu^{(r)}(\vx)\right)\left(M(\vy) - M(\vz)\right)\right]\right| \leq \|P_\mu(\vx) - \widetilde{P}_\mu^{(r)}(\vx)\|_*\cdot\|M(\vy) - M(\vz)\|_{\mathrm{op}}.
\end{align*}
For any unit vector $\vvv \in \R^n$:
\begin{align*}
\vvv^\intercal(M(\vy) - M(\vz))\vvv = \sum_{k=1}^m (y_k - z_k)(\vm_k^\intercal \vvv)^2 \leq \|\vy - \vz\|_\infty \cdot \|M\vvv\|^2 \leq \|M\|_{\mathrm{op}}^2,
\end{align*}
where the last inequality uses $\|\vy - \vz\|_\infty \leq 1$ for $\vy, \vz \in \mathcal{X}$,
resulting in $\|M(\vy) - M(\vz)\|_{\mathrm{op}} \leq \|M\|_{\mathrm{op}}^2$.

We can now combine this inequality with \cref{lemma:spectraltailweight}
\begin{align*}
\left|\left\langle \nabla g_\mu(\vx) - \widetilde{\nabla} g_\mu^{(r)}(\vx),\, \vy - \vz\right\rangle\right| \leq 2\tau_r(\vx)\cdot\|M\|_{\mathrm{op}}^2 = \delta_r(\vx),
\end{align*}
giving the desired inequality.
\end{proof}

\cref{thm:fwtruncated} controls the Frank-Wolfe convergence loss caused by replacing the full gradient with a uniformly accurate truncated gradient.
\begin{proof}[Proof of \cref{thm:fwtruncated}]
Let $g_\mu^* = \max_{\vx \in \mathcal{X}} g_\mu(\vx)$ and $\vx^*_\mu$ be its maximizer.
Define $h_t = g_\mu^* - g_\mu(\vx_t)$, the primal gap at iteration $t$.

By $L$-smoothness of $g_\mu$ and
\begin{align*}
\|\vx_{t+1} - \vx_t\|^2 = \gamma_t^2\|\tilde{\vvv}_t - \vx_t\|^2 \leq \gamma_t^2 D^2,
\end{align*}
we obtain
\begin{align}
h_{t+1} \leq h_t - \gamma_t \left\langle \nabla g_\mu(\vx_t),\, \tilde{\vvv}_t - \vx_t\right\rangle + \frac{\gamma_t^2 L}{2} D^2. \label{eq:termAintruncation}
\end{align}

We decompose using the truncation error:
\begin{align*}
\innp{\nabla g_\mu(\vx_t)}{\tilde{\vvv}_t - \vx_t} = \underbrace{\innp{\widetilde{\nabla} g_\mu^{(r_t)}(\vx_t)}{\tilde{\vvv}_t - \vx_t}  }_{(\mathrm{I})} + \underbrace{\innp{\nabla g_\mu(\vx_t) - \widetilde{\nabla} g_\mu^{(r_t)}(\vx_t)}{\tilde{\vvv}_t - \vx_t}  }_{(\mathrm{II})}.
\end{align*}

For (I): since $\tilde{\vvv}_t$ maximizes the approximate gradient over $\mathcal{X}$:
\begin{align*}
\left\langle \widetilde{\nabla} g_\mu^{(r_t)}(\vx_t),\, \tilde{\vvv}_t - \vx_t\right\rangle \geq \left\langle \widetilde{\nabla} g_\mu^{(r_t)}(\vx_t),\, \vx_\mu^* - \vx_t\right\rangle = \left\langle \nabla g_\mu(\vx_t),\, \vx_\mu^* - \vx_t\right\rangle + \left\langle \widetilde{\nabla} g_\mu^{(r_t)}(\vx_t) - \nabla g_\mu(\vx_t),\, \vx_\mu^* - \vx_t\right\rangle.
\end{align*}
Applying \cref{lemma:oracleerror} to the remaining error term:
\begin{align*}
|\langle \widetilde{\nabla} g_\mu^{(r_t)} - \nabla g_\mu, \vx_\mu^* - \vx_t \rangle| \leq \bar{\delta}.
\end{align*}
The term (II) satisfies $|(\mathrm{II})| \leq \bar{\delta}$ from \cref{lemma:oracleerror} applied directly, and in particular $(\mathrm{II}) \geq -\bar{\delta}$.
Combining both bounds and using concavity of $g_{\mu}$, which gives $\innp{\nabla g_\mu(\vx_t)}{\vx_\mu^* - \vx_t} \geq g_\mu^* - g_\mu(\vx_t) = h_t$, we obtain
\begin{align}
    \innp{\nabla g_\mu(\vx_t)}{\tilde{\vvv}_t - \vx_t} \geq h_t - 2\bar{\delta}. \label{eq:termBintruncation}
\end{align}

Substituting \eqref{eq:termBintruncation} into \eqref{eq:termAintruncation}:
\begin{align}
h_{t+1} \leq (1-\gamma_t)\,h_t + \frac{\gamma_t^2 L D^2}{2} + 2\gamma_t\bar{\delta}.\label{eq:tempeqFWerror}
\end{align}

Define $\Phi_t = \frac{t(t+1)}{2} h_t$.
Multiplying \eqref{eq:tempeqFWerror} by $\frac{(t+1)(t+2)}{2}$ and substituting $\gamma_t = 2/(t+2)$:
\begin{align*}
\Phi_{t+1} \leq \Phi_t + \frac{(t+1)}{t+2}LD^2 + 2(t+1)\bar{\delta} \leq \Phi_t + LD^2 + 2(t+1)\bar{\delta}.
\end{align*}

Summing from $t = 0$ to $T-1$ with $\Phi_0 = 0$:
\begin{align*}
\frac{T(T+1)}{2}\,h_T \leq T\cdot LD^2 + 2\bar{\delta}\cdot\frac{T(T+1)}{2}.
\end{align*}

Dividing by $T(T+1)/2$ gives the smoothed objective bound:
\begin{align*}
g_\mu^* - g_\mu(\vx_T) \leq \frac{2LD^2}{T+1} + 2\bar{\delta}.
\end{align*}
Following the same argument as \cref{th:FrankWolfeConvergence}, the smoothing sandwich
$g\leq g_\mu\leq g+\mu\log n$ implies $g^*\leq g_\mu^*$ and
$0\leq g_\mu(\vx_T)-g(\vx_T)\leq\mu\log n$:
\begin{align*}
g^* - g(\vx_T)
= \underbrace{(g^* - g_\mu^*)}_{\leq 0}
+ \underbrace{(g_\mu^* - g_\mu(\vx_T))}_{= h_T}
+ \underbrace{(g_\mu(\vx_T) - g(\vx_T))}_{\leq \mu\log n}
\leq \mu\log n + \frac{2LD^2}{T+1} + 2\bar{\delta}.
\end{align*}
\end{proof}

\cref{cor:complexitytruncation} specializes the truncated-gradient guarantee to the hypersimplex and converts the spectral-tail estimate into an iteration and eigengap condition.
\begin{proof}[Proof of \cref{cor:complexitytruncation}]
From \cref{thm:fwtruncated}, with $L = \|M\|_{\mathrm{op}}^2/\mu$, $\mu = \varepsilon/(2\log n)$, and $D^2 = 2N$, we have $2LD^2/(T+1) = 8N\|M\|_{\mathrm{op}}^2\log n/(\varepsilon(T+1))$, so
\begin{align*}
    g^* - g(\vx_T) \leq \mu\log n + \frac{8N\|M\|_{\mathrm{op}}^2\log n}{\varepsilon(T+1)} + 2\bar{\delta}.
\end{align*}
Setting $\bar{\delta} = \varepsilon/8$ so that $\mu\log n + 2\bar{\delta} = 3\varepsilon/4$, and requiring $\frac{8N\|M\|_{\mathrm{op}}^2\log n}{\varepsilon(T+1)} \leq \varepsilon/4$ yields the stated bound on $T$.
For the truncation condition, $\delta_{r_t}(\vx_t) = 2\|M\|_{\mathrm{op}}^2 \tau_{r_t}(\vx_t) \leq \varepsilon/8$ is equivalent to $\tau_{r_t}(\vx_t) \leq \varepsilon/(16\|M\|_{\mathrm{op}}^2)$.
Applying the spectral gap bound from \cref{lemma:spectraltailweight}:
\begin{align*}
\tau_{r_t}(\vx_t) \leq (n-r_t)\,e^{-\Delta_{r_t}(\vx_t)/\mu} \leq \frac{\varepsilon}{16\|M\|_{\mathrm{op}}^2}
\end{align*}
holds as soon as $\Delta_{r_t}(\vx_t) \geq \mu \log\!\left(\frac{16(n-r_t)\|M\|_{\mathrm{op}}^2}{\varepsilon}\right)$.
\end{proof}  

\cref{thm:safesmoothing} presents the lower-bound on the separation between optimal and non-optimal integral objective values using their characteristic polynomials.
\begin{proof}[Proof of \cref{thm:safesmoothing}]
Let $m_y \in \Z[\lambda]$ denote the minimal polynomial of $g_y$ over $\mathbb{Q}$;
it is monic and irreducible over $\mathbb{Q}[\lambda]$, with roots $\{g_{\vy} = \beta_1, \dots, \beta_s\}$, with $s=\deg(m_y)$, the algebraic conjugates of $g_y$.
We also note $\{\alpha_i\}_{i\in[n]}$ the roots of $q_{\vx^*}$, i.e., the eigenvalues of $M_l(\vx^*)$, ordered in non-increasing order.
Since $m_{\vy}$ is irreducible over $\mathbb{Q}[\lambda]$ and does not divide $q_x$, $\mathrm{gcd}(m_{y},q_x) = 1$ over $\mathbb{Q}[\lambda]$,
the two polynomials are coprime.

The polynomials $m_{\vy}$ and $q_x$ are monic elements of $\Z[\lambda]$, their resultant is an integer as the determinant of their Sylvester matrix.
Combined with coprimality, the resultant is an integer with absolute value greater or equal to one.
The resultant can be expressed as:
\begin{align}\label{eq:resultant}
|\mathrm{Res}(m_{\vy}, q_{\vx^*})| = \prod_{j=1}^{s} \prod_{i=1}^{n} |\beta_j - \alpha_i| \geq 1.
\end{align}
We can now bound the absolute differences $|\beta_j - \alpha_i|\, \forall j \in [s], i \in [n]$.
By \cref{lemma:lowerboundtrace}, all eigenvalues of $M_l(\vx^*)$ and $M_l(\vy)$ lie in $(0,T_l]$.
The polynomial $m_y$ is the monic polynomial of least degree in $\mathbb{Q}[\lambda]$ vanishing at $g_{\vy}$.
The polynomial $q_{\vy}$ also vanishes at $g_{\vy}$ since this is one of the eigenvalues of $M_l(\vy)$.
By the minimality and irreducibility of $m_y$, we have that $m_y$ divides $q_{\vy}$ in $\mathbb{Q}[\lambda]$,
all the roots $\{\beta_j\}_{j\in[s]}$ are roots of $q_{\vy}$, and by \cref{lemma:lowerboundtrace} thus all lie in $(0,T_l]$.
We can conclude for any pair $(i,j) \in [n] \times [s]$ that $|\beta_j - \alpha_i| \leq \max \{\beta_j,\alpha_i\} \leq T_l$.
Combining with \eqref{eq:resultant}:
\begin{align}
& 1 \leq \prod_{j=1}^{s} \prod_{i=1}^{n} |\beta_j - \alpha_i| \nonumber \\
& 1 \leq (\beta_1 - \alpha_n) \prod_{(i,j) \in [n]\times [s] (n,1)} |\beta_j - \alpha_i|\nonumber \\
& 1 \leq \delta_l T_l^{sn-1}.\label{eq:boundwiths}
\end{align}
Since $m_y$ is the minimal polynomial of $g_{\vy}$, it divides $q_{\vy}$.
The value $g_{\vy}$ is not a root of $\gcd(q_{\vx^*}, q_{\vy})$ by
optimality of $\vx^*$ and suboptimality of $\vy$.
Hence, $m_{\vy}$ divides the cofactor $q_{\vy} / \gcd(q_{\vx},q_{\vy})$ which has degree $n-d$, which is thus an upper bound on $s$.
Replacing $s$ in \eqref{eq:boundwiths} and re-arranging, we obtain the desired bound of the theorem.
\end{proof}

\cref{prop:tightening} uses a dual upper bound and the bound multipliers to certify variable fixings against the incumbent.
\begin{proof}[Proof of \cref{prop:tightening}]
Let $P$ and $D$ denote the current primal-dual problem pair and $\bar{P}_j$ and $\bar{D}_j$ the new primal-dual problem pair. 
We will only present the proof for the first tightening condition, the other follow analogously.
First note that $S_D$ is still valid solution for the new dual problem $\bar{D}_j$ and its
objective value is an upper bound both on the optimal value of $\bar{D}_j$ and on the optimal value of $\bar{P}_j$ by weak duality.
It is easy to see that the objective value of $S_D$ for $\bar{D}_j$ is precisely 
given by $d - \alpha_j$.
In a B\&B tree, we can discard nodes with their optimal solution is smaller than the current incumbent.
Therefore, if $d - \alpha_j \leq \hat{G}$, we may fix $x_j = 0$.
\end{proof}

\cref{lem:dual_solution_from_primal} presents the construction of dual variables from the minimum eigenspace and the ordered experiment scores of a primal point.
\begin{proof}[Proof of \cref{lem:dual_solution_from_primal}]
We will start with the derivation of $Z$.
In the following, we will denote the optimal value of a given primal or dual variable by a superscripted star $*$.
    From the complementary slackness of the primal KKT conditions, we have 
\begin{align*}
    \innp{Z^*}{M(\vx^*) - \lambda^* I} &= 0 \\
    \Rightarrow \quad \innp{Z^*}{M(\vx^*)} &= \lambda^* \Tr(Z^*).
\end{align*}
Since $Z$ cannot be the zero matrix, due to the constraint on its trace, we can conclude that $Z^*$
must be a convex combination of $\{\vw_i^*\vw_i^{*\intercal}\}_{i=1}^p$ where $\{\vw_i^*\}_{i\in [p]}$ are the eigenvectors
associated with the smallest eigenvalue of $M(\vx^*)$ and $p$ is the corresponding multiplicity.
Using this observation, we arrive at the expression in \eqref{eq:Z_from_primal}.
Note that since the eigenvectors of $M(\vx)$ and $M(\vx^*)$ can be arbitrarily different, even if the objective values
are close, $Z$ is potentially a poor approximation of $Z^*$ but it is nevertheless valid.

For the choice of $\zeta$, we first compute the scores $v_i = \innp{Z}{\vm_i\vm_i^\intercal}$ for all $i \in [m]$ and order $\vvv$ in 
descending order. 
Let $\kappa(\zeta) = \# \{v_i > \zeta \mid i \in [m]\}$ be the number of experiments with score greater than $\zeta$.
The constraints $\zeta = v_i + \alpha_i - \beta_i$ force $\alpha_i = \max(0, \zeta - v_i)$ and $\beta_i = \max(0, v_i -\zeta)$.
Therefore, we can rewrite the objective function as
\begin{align}
    N\zeta - \sum_{i=\kappa(\zeta)+1 }^m l_i \max(0, \zeta - v_i) + \sum_{j=1}^{\kappa(\zeta)} u_j \max(0, v_j -\zeta) + \innp{Z}{C}
\end{align}
and our choice of $\zeta$ should minimize the above expression.
In the experiments, we are only considering binary variables, and in this case, we find
$\zeta = \vvv_{(N)}$, the $N$-th largest score.
If some experiments have already been fixed, we only consider the scores of the unfixed experiments
and the budget $N$ is reduced accordingly.
\end{proof}

\begin{lemma}[Dual Price Approximation]
\label{lem:dual_price}
At node $l$, consider the concave maximization node subproblem:
\begin{align}
    \max_{\vx \in \mathrm{conv}(\mathcal{X}_l)} g_{\mu_l}(\vx). \label{eq:CPl}
\end{align}
At iteration $t$, the Frank-Wolfe linear subproblem is:
\begin{align}
    \vvv_t = \argmax_{\vvv \in \mathrm{conv}(\mathcal{X}_l)}
    \nabla g_{\mu_l}(\vx_t)^\intercal \vvv, \label{eq:LPl}
\end{align}
with FW gap $\Gamma_t = \nabla g_{\mu_l}(\vx_t)^\intercal(\vvv_t-\vx_t) \geq 0$;
it is an instance of a continuous knapsack with uniform weights.
The dual of \eqref{eq:LPl} with respect to the budget constraint
$\sum_{i \in S_f^l} z_i = N_l$ has optimal dual variable:
\begin{align}
    \bar{\zeta}_t = [\nabla g_{\mu_l}(\vx_t)]_{(N_l)}, \label{eq:tbar}
\end{align}
the $N_l$-th largest component of $\nabla g_{\mu_l}(\vx_t)$ over $i \in S_f^l$,
computed as a by-product of solving \eqref{eq:LPl}. Let $\zeta^*$ be the optimal
dual variable associated with the budget constraint $\sum_{i \in S_f^l} x_i = N_l$
in the dual problem \eqref{eq:SDPDual}. Then:
\begin{align}
    \bar{\zeta}_t - \Gamma_t \leq \zeta^* \leq \bar{\zeta}_t + \Gamma_t. \label{eq:DP}
\end{align}
\end{lemma}
\begin{proof}
The problem \eqref{eq:CPl} is a concave maximization over the polytope
$\mathrm{conv}(\mathcal{X}_l)$, defined by the budget constraint
$\sum_{i \in S_f^l} x_i = N_l$ together with the bound constraints
$\vl \leq \vx \leq \vu$. Applying the result of \citet{braun2021dual} to the
equivalent convex minimization of $-g_{\mu_l}$ and translating the signs back to
the maximization convention, for any
$\vx_t \in \mathrm{conv}(\mathcal{X}_l)$ with FW gap $\Gamma_t$, the dual
price $\bar{\zeta}_t$ of the linearized LP \eqref{eq:LPl} approximates the true
dual price $\zeta^*$ of the budget constraint at the optimum of \eqref{eq:CPl}
with error at most $\Gamma_t / \delta_b$, where $\delta_b > 0$ is the budget
perturbation magnitude. Since $N_l \in \mathbb{Z}$ and perturbations smaller
than $1$ do not change the combinatorial structure of $\mathrm{conv}(\mathcal{X}_l)$,
we have $\delta_b = 1$, and \eqref{eq:DP} follows.
\end{proof}

\cref{thm:exclusion} combines eigenspace perturbation bounds with an approximate budget dual price to exclude a variable safely.
\begin{proof}[Proof of \cref{thm:exclusion}]
We first bound the difference in Frobenius norm between two matrices.
Since $M_l(\cdot)$ is linear in $\vx$ and both
$\vx_t, \vx^*_{\mathrm{rel}} \in \mathrm{conv}(\mathcal{X}_l)$ with
$\|\vx_t - \vx^*_{\mathrm{rel}}\|_2 \leq D_l = \sqrt{2N_l}$:
\begin{align}
    \|M_l(\vx_t) - M_l(\vx^*_{\mathrm{rel}})\|_F
    \leq \|M_l \|_{\mathrm{op}}\sqrt{2N_l}
    = \varepsilon_l.
    \label{eq:frob}
\end{align}

By Weyl's inequality applied between $M_l(\vx^*_{\mathrm{rel}})$
and $M_l(\vx_t)$,
\begin{align}
    |\lambda_i(M_l(\vx^*_{\mathrm{rel}})) - \lambda_i(M_l(\vx_t))|
    \leq \|M_l(\vx_t) - M_l(\vx^*_{\mathrm{rel}})\|_F
    \leq \varepsilon_l
    \quad \forall i \in [n]. \label{eq:weyl}
\end{align}
Using \eqref{eq:safetyweyl}, the minimum eigenvalue of
$M_l(\vx^*_{\mathrm{rel}})$ is also simple with eigengap at least
$\gamma_t - 2\varepsilon_l > 0$.

Applying the Davis-Kahan $\sin\theta$ theorem to the pair
$(M_l(\vx_t), M_l(\vx^*_{\mathrm{rel}}))$, using the eigengap $\gamma_t$
of the known matrix $M_l(\vx_t)$ and the perturbation bound
\eqref{eq:frob}:
\begin{align}
    \|\hat{\vw}_t\hat{\vw}_t^\intercal -
    \vw^*_{\mathrm{rel}}\vw^{*\top}_{\mathrm{rel}}\|_F
    = \sqrt{2}|\sin\theta|
    \leq \frac{\sqrt{2}\,\varepsilon_l}{\gamma_t},
    \label{eq:dk}
\end{align}
where $\vw^*_{\mathrm{rel}}$ is the exact minimum eigenvector of
$M_l(\vx^*_{\mathrm{rel}})$.

By the triangle inequality, \eqref{eq:dk} and the eigenvector computation precision $\eta$:
\begin{align}
    \|\vw_t\vw_t^\intercal - \vw^*_{\mathrm{rel}}\vw^{*\top}_{\mathrm{rel}}\|_F
    &\leq
    \|\vw_t\vw_t^\intercal - \hat{\vw}_t\hat{\vw}_t^\intercal\|_F
    + \|\hat{\vw}_t\hat{\vw}_t^\intercal -
    \vw^*_{\mathrm{rel}}\vw^{*\top}_{\mathrm{rel}}\|_F
    \leq \eta + \frac{\sqrt{2}\,\varepsilon_l}{\gamma_t}
    = \sigma_l.
    \label{eq:total_err}
\end{align}

For any $\vm_i$, using $\|P\|_2 \leq \|P\|_F$ and \eqref{eq:total_err}:
\begin{align*}
    & |(\vw_t^\intercal \vm_i)^2 - (\vw^{*\top}_{\mathrm{rel}} \vm_i)^2| \\
    =&  |\vm_i^\intercal (\vw_t\vw_t^\intercal - \vw^*_{\mathrm{rel}}\vw^{*\top}_{\mathrm{rel}}) \vm_i| \\
    \leq & \|\vm_i\|_2^2 \|\vw_t\vw_t^\intercal - \vw^*_{\mathrm{rel}}\vw^{*\top}_{\mathrm{rel}}\|_2\\
    \leq & \|\vm_i\|_2^2 \|\vw_t\vw_t^\intercal - \vw^*_{\mathrm{rel}}\vw^{*\top}_{\mathrm{rel}}\|_F\\
    \leq & \|\vm_i\|_2^2 \cdot \sigma_l,
\end{align*}
and therefore:
\begin{align}
    (\vw^{*\top}_{\mathrm{rel}} \vm_i)^2
    \leq (\vw_t^\intercal \vm_i)^2 + \|\vm_i\|_2^2 \cdot \sigma_l.
    \label{eq:IP}
\end{align}

Assuming Slater's condition holds for \eqref{eq:SDPPrimal} at
node $l$ and strong duality between \eqref{eq:SDPPrimal} and
\eqref{eq:SDPDual} is satisfied, the KKT stationarity condition of
\eqref{eq:Lagrangian} with respect to $x_i$, using
$M_i = \vm_i\vm_i^\intercal$, gives:
\begin{align}
    \innp{Z^*}{\vm_i\vm_i^\intercal} = \zeta^* - \alpha^*_i + \beta^*_i
    \quad \forall i \in S_f^l.
    \label{eq:stationarity}
\end{align}
The complementary slackness conditions for the bound constraints give
$\alpha^*_i((x^*_{\mathrm{rel}})_i - l_i) = 0$ and
$\beta^*_i(u_i - (x^*_{\mathrm{rel}})_i) = 0$.
For $i \in S_f^l$ with $(x^*_{\mathrm{rel}})_i > l_i$, we have
$\alpha^*_i = 0$, so \eqref{eq:stationarity} gives
$\innp{Z^*}{\vm_i\vm_i^\intercal} = \zeta^* + \beta^*_i \geq \zeta^*$.
Taking the contrapositive:
\begin{align}
    \innp{Z^*}{\vm_i\vm_i^\intercal} < \zeta^*
    \implies (x^*_{\mathrm{rel}})_i = l_i = 0.
    \label{eq:exclusion_dual}
\end{align}
The SDP complementary slackness condition
$\innp{Z^*}{M_l(\vx^*_{\mathrm{rel}}) - \lambda^* I} = 0$ with
$Z^* \succeq 0$ and $M_l(\vx^*_{\mathrm{rel}}) - \lambda^* I \succeq 0$
forces $Z^*$ to lie in the null space of
$M_l(\vx^*_{\mathrm{rel}}) - \lambda^* I$.
By the simplicity of the minimum eigenvalue of $M_l(\vx^*_{\mathrm{rel}})$ (from the eigengap argument),
this null space is one-dimensional and spanned by
$\vw^*_{\mathrm{rel}}$, and since $\Tr(Z^*) = 1$ and $Z^* \succeq 0$:
\begin{align}
    Z^* = \vw^*_{\mathrm{rel}}\vw^{*\top}_{\mathrm{rel}}.
    \label{eq:Zstar}
\end{align}
Substituting \eqref{eq:Zstar} into \eqref{eq:exclusion_dual}:
\begin{align}
    (\vw^{*\top}_{\mathrm{rel}}\vm_i)^2 < \zeta^*
    \implies (x^*_{\mathrm{rel}})_i = 0.
    \label{eq:EX}
\end{align}
Instead of solving the dual problem to compute the value of $t$ in an optimal solution,
we rely on the dual price estimation from the linear oracle.
By \cref{lem:dual_price}, $\zeta^* \geq \bar{\zeta}_t - \Gamma_t$.
Combining \eqref{eq:IP} with condition \eqref{eq:EC}:
\begin{align}
    (\vw^{*\top}_{\mathrm{rel}} \vm_i)^2
    \leq (\vw_t^\intercal \vm_i)^2 + \|\vm_i\|_2^2 \cdot \sigma_l
    < \bar{\zeta}_t - \Gamma_t
    \leq \zeta^*,
    \label{eq:final}
\end{align}
which satisfies the premise of \eqref{eq:EX}, giving
$(x^*_{\mathrm{rel}})_i = 0$. Fixing $x_i = 0$ at node $l$ therefore does
not remove the relaxation optimizer from the feasible set and is safe in
the branch-and-bound.
\end{proof}

\section{Additional computational results}\label{app:comp_appendix}

In this appendix, we showcase more detailed and complementary results to the main text.
In tables \cref{tab:avg_by_dimension_connected,tab:by_dimension_(correlated),tab:avg_by_dimension_disconnected,tab:spanning_tree,tab:by_dimension_(independent)},
we compare the performance of Boscia, SCIPSDP B\&B and SCIPSDP OA split up by dimension.
Tables \cref{tab:by_N_(correlated),tab:by_N_(independent)} display the effect of the magnitude of the budget $N$ for the 
E-Optimal Experiment Design instances. 
"One" denotes $N=1.5 n$ and "log" $N= 1.5n\log(n)$. 
Observe that a larger $N$ is beneficial for the performance of almost all solvers. 
This is consistent with the observation in \citet{ahipasaoglu2025column} in which the same pattern was observed for the D-Optimal Experiment Design.
\cref{tab:smoothing_ablation} shows the effect of the truncated gradient and the smoothing parameter on the performance of Boscia.
Adaptive $\mu$ uses the values specified in \cref{eq:mu_schedule_adaptive}; otherwise, the values in \cref{tab:mu_schedules} are used.

\cref{fig:ACST_independent_pruning,fig:AGC_correlated_exclusion,fig:AGC_correlated_pruning,fig:AGC_independent_pruning,fig:E_correlated_pruning,fig:E_independent_pruning,tab:smoothing_ablation} 
showcase the effect of the algorithmic improvements on the considered problems.
In \cref{fig:ACST_independent_pruning}, we see that the simplied formulation of the Algebraic Connectivity Spanning Tree problem, which 
replaces the spanning tree constraint with a knapsack constraint, does not improve performance.
Even though the probability simplex LMO is computationally cheaper, enforcing the spanning tree constraint leads to a tighter relaxation
which ultimately is more beneficial for the performance.
As expected, the exclusion/fixings criteria are quite expensive to perform, especially if the actual dual is solved (dubbed dual exclusion), see
Figure \cref{fig:AGC_correlated_exclusion}. 
The dual fixing variant solves the actual dual problem with Mosek, the others use the approximation from \cref{lem:dual_solution_from_primal}. The random variant computes a couple of possible $Z$ as convex combination of the 
eigenvectors of the minimum eigenvalue and chooses the one that achieves the smallest $\zeta$ value.
The tighter tol variant uses tighter tolerance for Frank-Wolfe to test if that positively effects the dual variable approximation.
Even computing the approximated dual variables does not positively affect runs.
Generally, the few fixings can be generated, even when the exact dual information is used. 
The effect of the pruning strategies is present in
\cref{fig:AGC_correlated_pruning,fig:AGC_independent_pruning,fig:E_correlated_pruning,fig:E_independent_pruning}.
Clearly, the eigenvalue-based pruning has a positive effect.
Note that the rank-based pruning provides weaker information than the eigenvalue-based pruning, so it is no surprise that it does not perform well.
Lastly, we compare retaining one half or one third of the spectrum in the truncated gradient, as well as scaling the smoothing parameter to the instance spectrum.
The results for EOD, AGC, and ACST/ACSTS are summarized in \cref{tab:smoothing_ablation}.
The effect of the truncated gradient is particularly strong for independent data.
This is likely due to the spread of the eigenvalues as can be seen in \cref{fig:Trajectory_Eigenvalues}.
In the independent data case (right-hand plot), the distance of the largest eigenvalue to all others is far more pronounced than for the correlated data.
\cref{fig:UpperboundsRoot} shows the upper bound comparison at the root node for EOD and AGC. 
It also shows that the SDP relaxation yields the best upper bound. 

\begin{table}[htbp]
    %{\tiny
    \centering
    \caption{Performance comparison for EOD with correlated data grouped by dimension. Geometric mean of the time is computed with a 1s shift. There are 50 instances in total.}
    \label{tab:by_dimension_(correlated)}
    \begin{tabular}{llrrrrrr}
\toprule
 Solver & Metric & 50 & 80 & 100 & 120 & 150 & all \\
\midrule
\multirow{3}{*}{Boscia} & \% sol. & \textbf{100.0} & 60.0 & \textbf{50.0} & 30.0 & \textbf{10.0} & 50.0 \\
  & time (s) & 75.7 & 368.1 & 1487.4 & 3110.7 & \textbf{3322.4} & 845.9 \\
  & rel. gap & --- & 9.77e-02 & 2.28e-01 & 1.80e-01 & 2.16e-01 & 1.83e-01 \\
\midrule
\multirow{3}{*}{SCIPSDP (OA)} & \% sol. & \textbf{100.0} & \textbf{100.0} & \textbf{50.0} & \textbf{40.0} & 0.0 & \textbf{58.0} \\
  & time (s) & \textbf{4.8} & \textbf{21.4} & \textbf{1179.4} & \textbf{901.5} & 3600.0 & \textbf{217.6} \\
  & rel. gap & --- & --- & --- & 6.45e-02 & 2.15e-01 & 1.93e-01 \\
\midrule
\multirow{3}{*}{SCIPSDP (B\&B)} & \% sol. & \textbf{100.0} & \textbf{100.0} & 40.0 & 10.0 & 0.0 & 50.0 \\
  & time (s) & 54.8 & 112.4 & 2672.3 & 3582.3 & 3600.0 & 736.5 \\
  & rel. gap & --- & --- & \textbf{5.68e-02} & \textbf{5.55e-02} & \textbf{8.57e-02} & \textbf{6.64e-02} \\
\bottomrule
\end{tabular}

    %}
\end{table}

\begin{table}[htbp]
    %{\tiny
    \centering
    \caption{Performance comparison for EOD with independent data grouped by dimension. Geometric mean of the time is computed with a 1s shift. There are 50 instances in total.}
    \label{tab:by_dimension_(independent)}
\begin{tabular}{llrrrrrr}
\toprule
 Solver & Metric & 50 & 80 & 100 & 120 & 150 & all \\
\midrule
\multirow{3}{*}{Boscia} & \% sol. & 80.0 & 20.0 & 20.0 & 0.0 & 0.0 & 24.0 \\
  & time (s) & 220.2 & 2594.1 & 3392.4 & 3600.0 & 3600.0 & 1906.6 \\
  & rel. gap & 7.29e-02 & 9.42e-02 & 1.99e-01 & 2.22e-01 & 2.71e-01 & 1.80e-01 \\
\midrule
\multirow{3}{*}{SCIPSDP (OA)} & \% sol. & \textbf{100.0} & 40.0 & 10.0 & 0.0 & 0.0 & 30.0 \\
  & time (s) & \textbf{14.0} & 2328.6 & 3267.8 & 3600.0 & 3600.0 & 1080.1 \\
  & rel. gap & --- & 7.62e-02 & 3.28e-01 & 2.99e-01 & 4.00e-01 & 2.63e-01 \\
\midrule
\multirow{3}{*}{SCIPSDP (B\&B)} & \% sol. & \textbf{100.0} & \textbf{100.0} & \textbf{40.0} & \textbf{30.0} & 0.0 & \textbf{54.0} \\
  & time (s) & 58.2 & \textbf{153.9} & \textbf{1403.3} & \textbf{2697.3} & 3600.0 & \textbf{658.8} \\
  & rel. gap & --- & --- & \textbf{1.17e-01} & \textbf{1.12e-01} & \textbf{1.46e-01} & \textbf{1.27e-01} \\
\bottomrule
\end{tabular}

    %}
\end{table}

\begin{table}[htbp]
    %{\tiny
    \centering
    \caption{Performance comparison for AGC with a connected base graph grouped by dimension. Geometric mean of the time is computed with a 1s shift. There are 20 instances in total.}
    \label{tab:avg_by_dimension_connected}
\begin{tabular}{llrrrrr}
\toprule
 Solver & Metric & 80 & 100 & 150 & 200 & all \\
\midrule
\multirow{3}{*}{Boscia} & \% sol. & \textbf{100.0} & \textbf{100.0} & \textbf{100.0} & \textbf{100.0} & \textbf{100.0} \\
  & time (s) & 37.3 & 12.0 & 1.7 & 7.5 & 9.3 \\
  & rel. gap & --- & --- & --- & --- & --- \\
\midrule
\multirow{3}{*}{SCIPSDP (OA)} & \% sol. & \textbf{100.0} & \textbf{100.0} & \textbf{100.0} & \textbf{100.0} & \textbf{100.0} \\
  & time (s) & \textbf{0.9} & \textbf{0.4} & \textbf{0.3} & \textbf{1.2} & \textbf{0.7} \\
  & rel. gap & --- & --- & --- & --- & --- \\
\midrule
\multirow{3}{*}{SCIPSDP (B\&B)} & \% sol. & \textbf{100.0} & \textbf{100.0} & \textbf{100.0} & \textbf{100.0} & \textbf{100.0} \\
  & time (s) & 50.4 & 50.2 & 50.2 & 50.3 & 50.3 \\
  & rel. gap & --- & --- & --- & --- & --- \\
\bottomrule
\end{tabular}

    %}
\end{table}

\begin{table}[htbp]
    %{\tiny
    \centering
    \caption{Performance comparison for AGC with a disconnected base graph grouped by dimension. Geometric mean of the time is computed with a 1s shift. There are 20 instances in total.}
    \label{tab:avg_by_dimension_disconnected}
\begin{tabular}{llrrrrr}
\toprule
 Solver & Metric & 80 & 100 & 150 & 200 & all \\
\midrule
\multirow{3}{*}{Boscia} & \% sol. & 20.0 & 40.0 & \textbf{20.0} & 0.0 & 20.0 \\
  & time (s) & 998.4 & 934.5 & \textbf{995.9} & 3600.0 & 1352.5 \\
  & rel. gap & 8.40e-02 & 1.15e-01 & 1.83e-01 & 1.20e-01 & 1.21e-01 \\
\midrule
\multirow{3}{*}{SCIPSDP (OA)} & \% sol. & 20.0 & 20.0 & 0.0 & 0.0 & 10.0 \\
  & time (s) & 2941.1 & 827.4 & 3600.7 & 3600.2 & 2370.2 \\
  & rel. gap & 9.34e-02 & 1.00e-01 & 1.56e-01 & 1.65e-01 & 1.28e-01 \\
\midrule
\multirow{3}{*}{SCIPSDP (B\&B)} & \% sol. & \textbf{80.0} & \textbf{80.0} & \textbf{20.0} & 0.0 & \textbf{45.0} \\
  & time (s) & \textbf{378.2} & \textbf{423.4} & 3052.6 & 3600.0 & \textbf{1152.4} \\
  & rel. gap & \textbf{2.08e-02} & \textbf{6.06e-02} & \textbf{8.24e-02} & \textbf{7.77e-02} & \textbf{6.89e-02} \\
\bottomrule
\end{tabular}

    %}
\end{table}

\begin{table}[htbp]
    %{\tiny
    \centering
    \caption{Performance comparison for ACST problem grouped by dimension. Geometric mean of the time is computed with a 1s shift. There are 12 instances in total.}
    \label{tab:spanning_tree}
\begin{tabular}{llrrrrr}
\toprule
 Solver & Metric & 45 & 66 & 105 & 300 & all \\
\midrule
\multirow{3}{*}{Boscia} & \% sol. & \textbf{100.0} & 0.0 & 0.0 & 0.0 & \textbf{14.3} \\
  & time (s) & \textbf{124.0} & 3600.0 & 3600.4 & 3602.7 & \textbf{2227.5} \\
  & rel. gap & --- & 8.14e-01 & \textbf{2.58e+00} & \textbf{5.42e+00} & \textbf{8.46e+00} \\
\midrule
\multirow{3}{*}{SCIPSDP (OA)} & \% sol. & 0.0 & 0.0 & 0.0 & 0.0 & 0.0 \\
  & time (s) & 3600.0 & 3600.0 & 3600.0 & 3600.0 & 3600.2 \\
  & rel. gap & --- & --- & 5.53e+00 & 1.00e+20 & 7.47e+15 \\
\midrule
\multirow{3}{*}{SCIPSDP (B\&B)} & \% sol. & 0.0 & 0.0 & 0.0 & 0.0 & 0.0 \\
  & time (s) & 3600.0 & 3600.0 & 3600.0 & 3600.0 & 3600.0 \\
  & rel. gap & 1.08e-01 & \textbf{7.24e-01} & 3.13e+00 & 2.09e+01 & 2.23e+04 \\
\bottomrule
\end{tabular}

    %}
\end{table}

\begin{table}[htbp]
    %{\tiny
    \centering
    \caption{Performance comparison for EOD with correlated data grouped by the budget construct. Geometric mean of the time is computed with a 1s shift. There are 50 instances in total. "One" denotes $N=1.5n$, the label "log" denotes $N=1.5n\log(n)$.}
    \label{tab:by_N_(correlated)}
\begin{tabular}{llrr}
\toprule
 Solver & Metric & one & log \\
\midrule
\multirow{3}{*}{Boscia} & \% sol. & 12.0 & 36.0 \\
  & time (s) & 2816.7 & 1290.5 \\
  & rel. gap & 2.74e-01 & 1.01e-01 \\
\midrule
\multirow{3}{*}{SCIPSDP (OA)} & \% sol. & 32.0 & 28.0 \\
  & time (s) & 1374.8 & 848.5 \\
  & rel. gap & 5.43e-01 & 1.33e-01 \\
\midrule
\multirow{3}{*}{SCIPSDP (B\&B)} & \% sol. & \textbf{40.0} & \textbf{68.0} \\
  & time (s) & \textbf{982.8} & \textbf{441.5} \\
  & rel. gap & \textbf{1.78e-01} & \textbf{6.73e-02} \\
\bottomrule
\end{tabular}

    %}
\end{table}

\begin{table}[htbp]
    %{\tiny
    \centering
    \caption{Performance comparison for EOD with independent data grouped by the budget construct. Geometric mean of the time is computed with a 1s shift. There are 50 instances in total. "One" denotes $N=1.5n$, the label "log" denotes $N=1.5n\log(n)$.}
    \label{tab:by_N_(independent)}
    \begin{tabular}{llrr}
\toprule
 Solver & Metric & one & log \\
\midrule
\multirow{3}{*}{Boscia} & \% sol. & 24.0 & \textbf{76.0} \\
  & time (s) & 2122.5 & 336.8 \\
  & rel. gap & 2.49e-01 & 6.84e-02 \\
\midrule
\multirow{3}{*}{SCIPSDP (OA)} & \% sol. & \textbf{40.0} & \textbf{76.0} \\
  & time (s) & \textbf{692.6} & \textbf{67.9} \\
  & rel. gap & 6.45e-01 & 7.05e-02 \\
\midrule
\multirow{3}{*}{SCIPSDP (B\&B)} & \% sol. & \textbf{40.0} & 60.0 \\
  & time (s) & 859.8 & 630.8 \\
  & rel. gap & \textbf{1.27e-01} & \textbf{2.50e-02} \\
\bottomrule
\end{tabular}

    %}
\end{table} 

\begin{table}[htbp]
    \centering
    \resizebox{\textwidth}{!}{\begin{tabular}{llrrrrrr}
\toprule
 Setting & Metric & EOD corr. & EOD ind. & AGC corr. & AGC ind. & ACST & ACSTS \\
\midrule
\multirow{3}{*}{Baseline} & \% sol. & 48.0 & 0.0 & \textbf{100.0} & 10.0 & \textbf{14.3} & \textbf{4.8} \\
 & time (s) & 2266.8 & 3600.0 & \textbf{40.5} & 3241.2 & 3223.6 & \textbf{3487.1} \\
 & rel. gap & 2.14e-01 & 6.85e-01 & --- & 2.09e-01 & 1.54e+01 & 5.56e+14 \\
\midrule
\multirow{3}{*}{Adaptive $\mu$} & \% sol. & \textbf{54.0} & \textbf{28.0} & 70.0 & \textbf{15.0} & 0.0 & 0.0 \\
 & time (s) & \textbf{2161.1} & \textbf{2895.6} & 1105.3 & 3212.1 & 3600.0 & 3600.0 \\
 & rel. gap & 2.51e-01 & \textbf{2.48e-01} & 5.96e-02 & 1.26e-01 & 9.91e+00 & 3.96e+14 \\
\midrule
\multirow{3}{*}{Half spectrum} & \% sol. & 50.0 & 10.0 & 90.0 & 10.0 & \textbf{14.3} & \textbf{4.8} \\
 & time (s) & 2178.9 & 3491.9 & 769.8 & 3241.1 & \textbf{3186.8} & 3580.0 \\
 & rel. gap & \textbf{2.13e-01} & 4.57e-01 & 6.13e-02 & 2.27e-01 & 2.27e+01 & 2.63e+15 \\
\midrule
\multirow{3}{*}{Third spectrum} & \% sol. & 8.0 & 0.0 & 70.0 & 10.0 & 4.8 & 0.0 \\
 & time (s) & 3387.2 & 3600.0 & 1340.9 & 3331.0 & 3596.5 & 3600.0 \\
 & rel. gap & 5.09e-01 & 1.20e+00 & \textbf{5.91e-02} & 3.11e-01 & 2.31e+01 & \textbf{1.42e+00} \\
\midrule
\multirow{3}{*}{Half spectrum + adaptive $\mu$} & \% sol. & 42.0 & 10.0 & 70.0 & \textbf{15.0} & 0.0 & 0.0 \\
 & time (s) & 2317.2 & 3486.6 & 1097.8 & \textbf{3080.0} & 3600.0 & 3600.0 \\
 & rel. gap & 2.16e-01 & 3.19e-01 & 6.07e-02 & \textbf{1.25e-01} & 8.26e+00 & 3.76e+14 \\
\midrule
\multirow{3}{*}{Third spectrum + adaptive $\mu$} & \% sol. & 4.0 & 0.0 & 70.0 & 10.0 & 0.0 & 0.0 \\
 & time (s) & 3467.7 & 3600.0 & 1260.0 & 3240.5 & 3600.0 & 3600.0 \\
 & rel. gap & 5.05e-01 & 1.19e+00 & 8.60e-02 & 1.26e-01 & \textbf{7.21e+00} & 2.99e+00 \\
\bottomrule
\end{tabular}
}
    \caption{Ablation of adaptive smoothing and one-half/one-third truncated gradients.
    Time is the arithmetic mean over the full instance set, with missing and timed-out runs set to the time limit.
    The relative dual gap is averaged over unsolved runs with a finite reported gap.}
    \label{tab:smoothing_ablation}
\end{table}

\begin{figure}
    \centering
    \begin{subfigure}{0.49\textwidth}
        \centering
        \includegraphics[width=\textwidth]{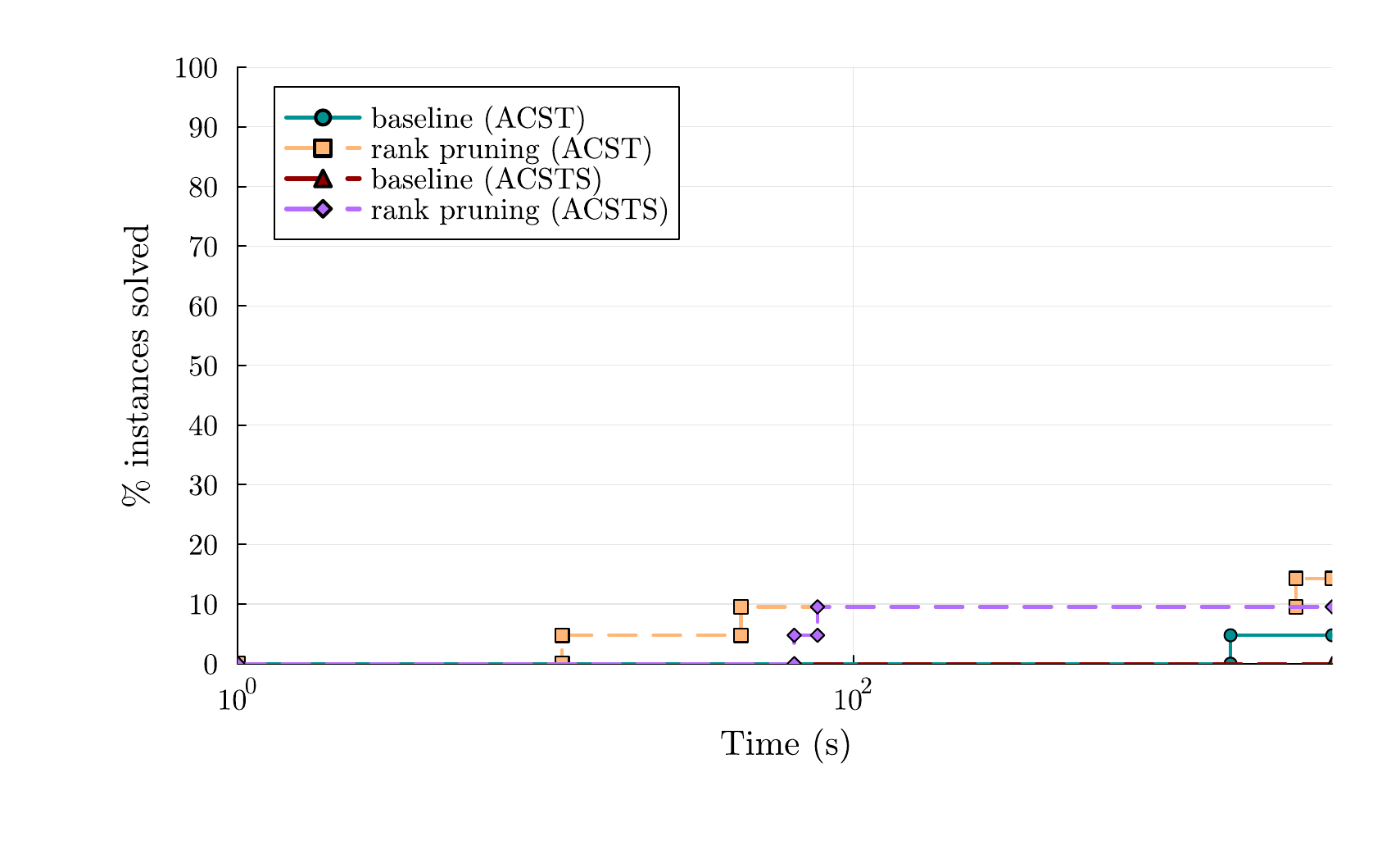}
        \caption{Percentage of instances solved over time for ACST. ACSTS denotes the knapsack formulation.}
        \label{fig:ACST_independent_pruning}
    \end{subfigure}
    \hfill
    \begin{subfigure}{0.49\textwidth}
        \centering
        \includegraphics[width=\textwidth]{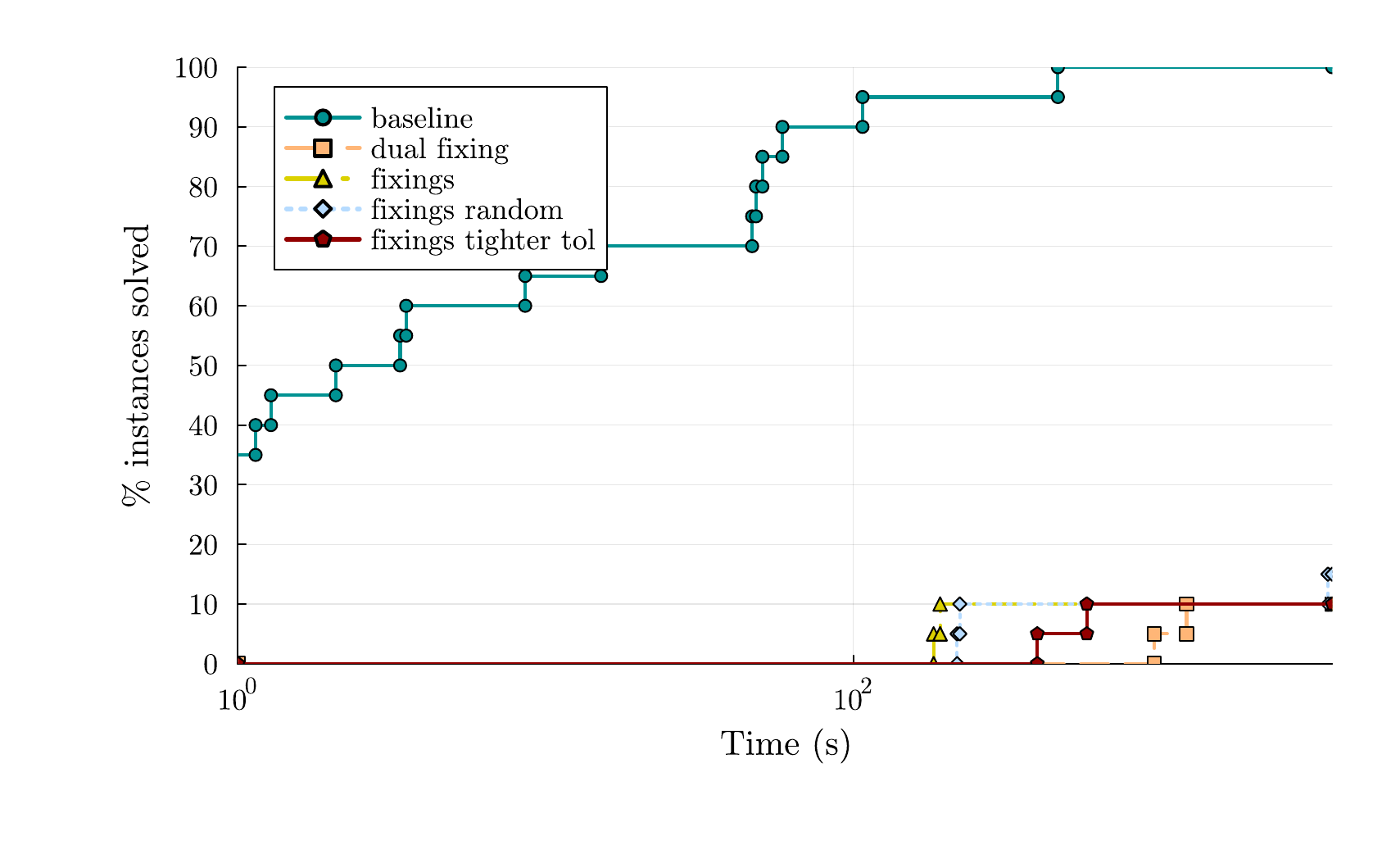}
        \caption{Percentage of instances solved over time for the different fixing techniques compared to the baseline for AGC with correlated data.}
        \label{fig:AGC_correlated_exclusion}
    \end{subfigure}
    \caption{Comparing the pruning strategies for ACST and the fixing techniques for AGC.}
\end{figure}

\begin{figure}
    \centering
    \begin{subfigure}{0.49\textwidth}
        \centering
        \includegraphics[width=\textwidth]{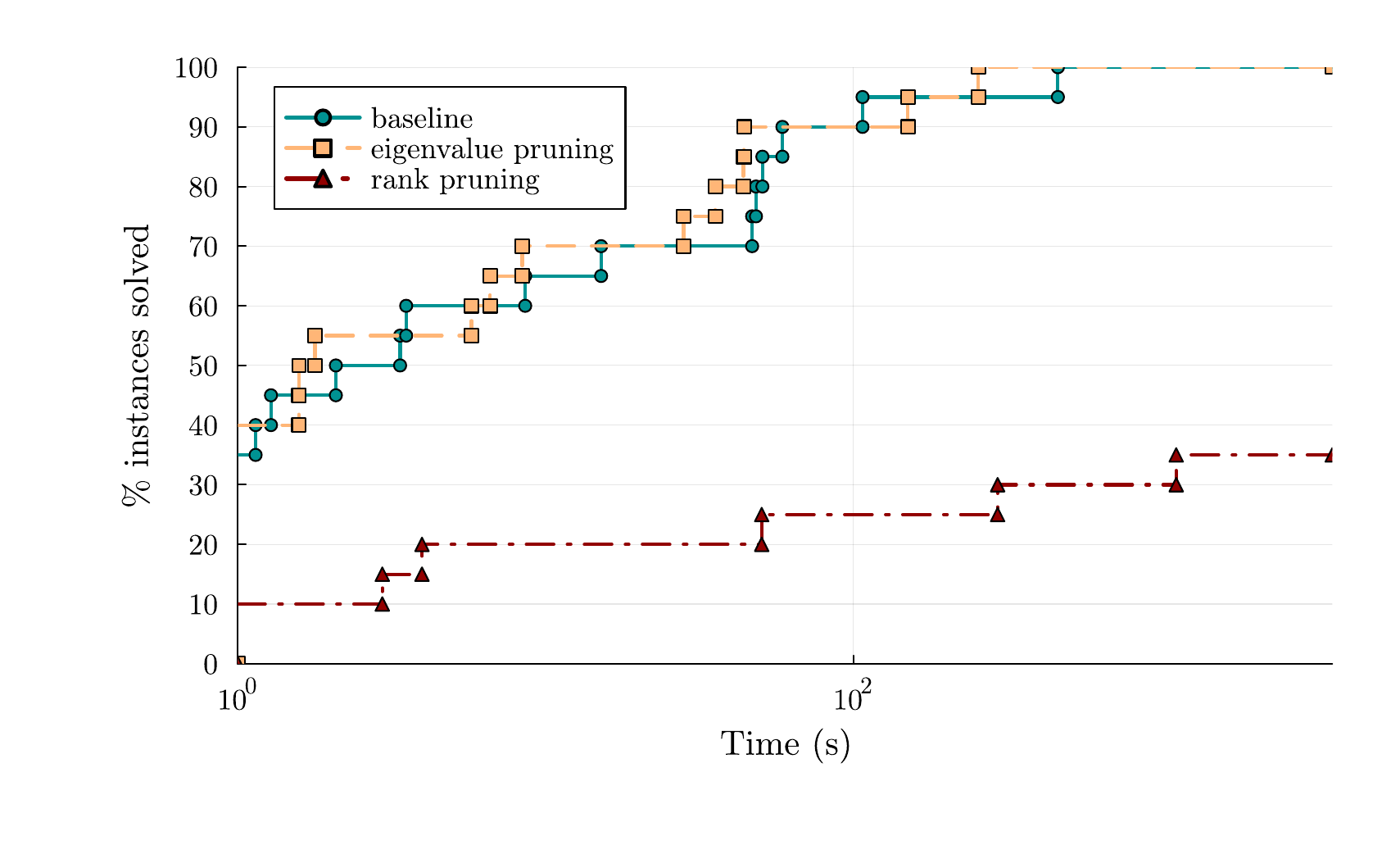}
        \caption{Percentage of instances solved over time for the different pruning techniques compared to the baseline for AGC with correlated data.}
        \label{fig:AGC_correlated_pruning}
    \end{subfigure}
    \hfill
    \begin{subfigure}{0.49\textwidth}
        \centering
        \includegraphics[width=\textwidth]{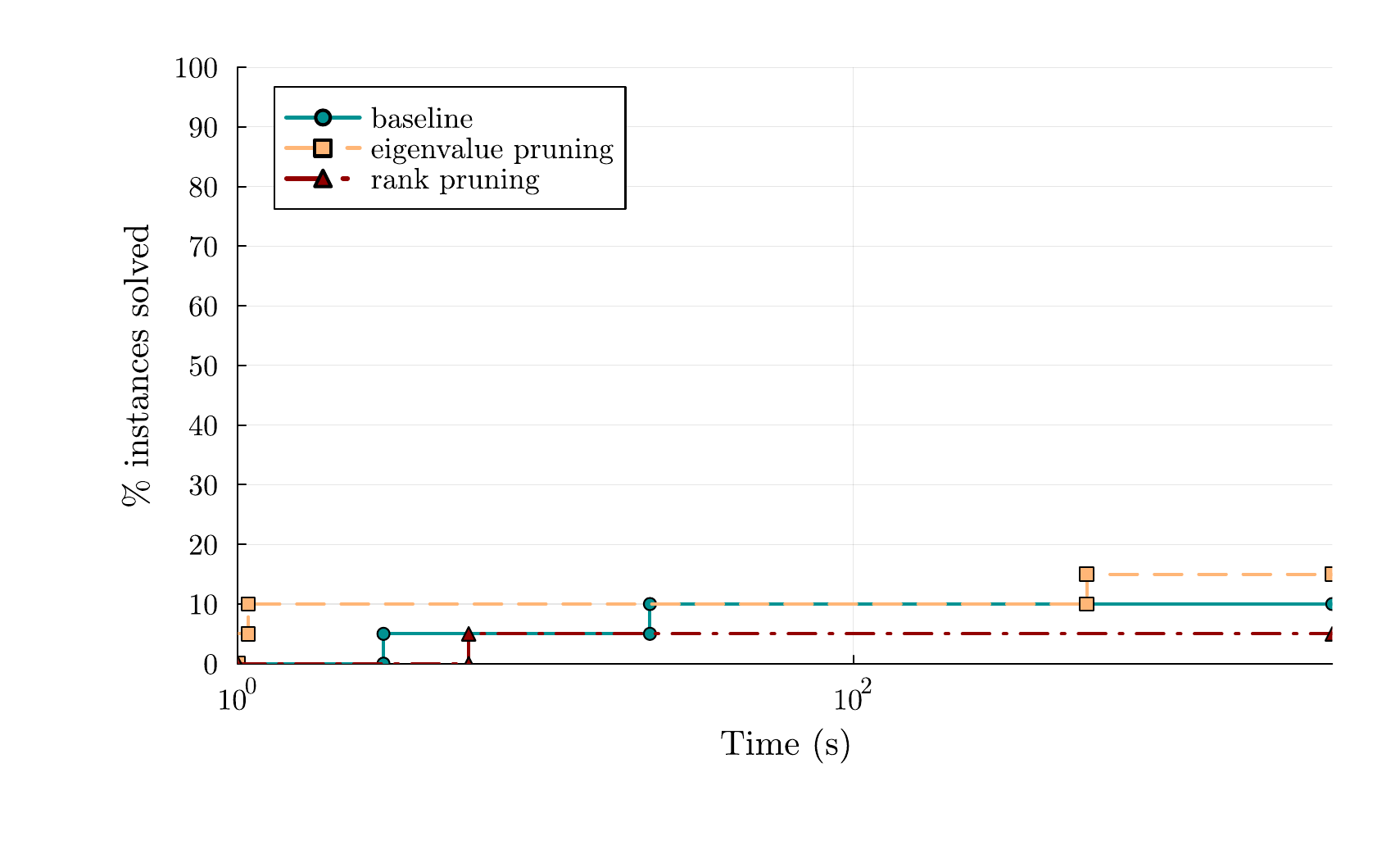}
        \caption{Percentage of instances solved over time for the different pruning techniques compared to the baseline for AGC with independent data.}
        \label{fig:AGC_independent_pruning}
    \end{subfigure}
    \caption{Comparing the pruning strategies for AGC}
\end{figure}

\begin{figure}
    \centering
    \begin{subfigure}{0.49\textwidth}
        \centering
        \includegraphics[width=\textwidth]{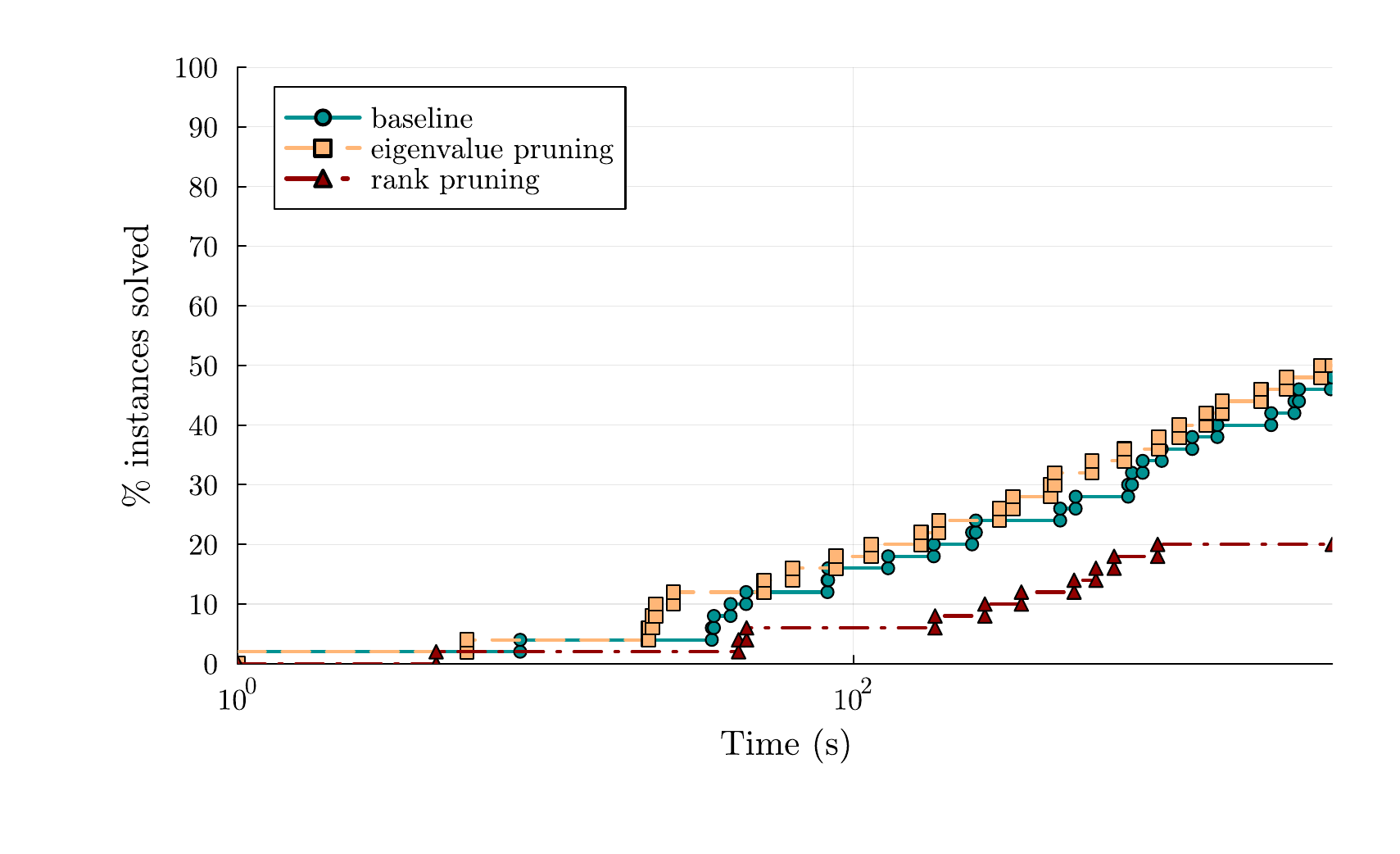}
        \caption{Percentage of instances solved over time for the different pruning techniques compared to the baseline for EOD with correlated data.}
        \label{fig:E_correlated_pruning}
    \end{subfigure}
    \hfill
    \begin{subfigure}{0.49\textwidth}
        \centering
        \includegraphics[width=\textwidth]{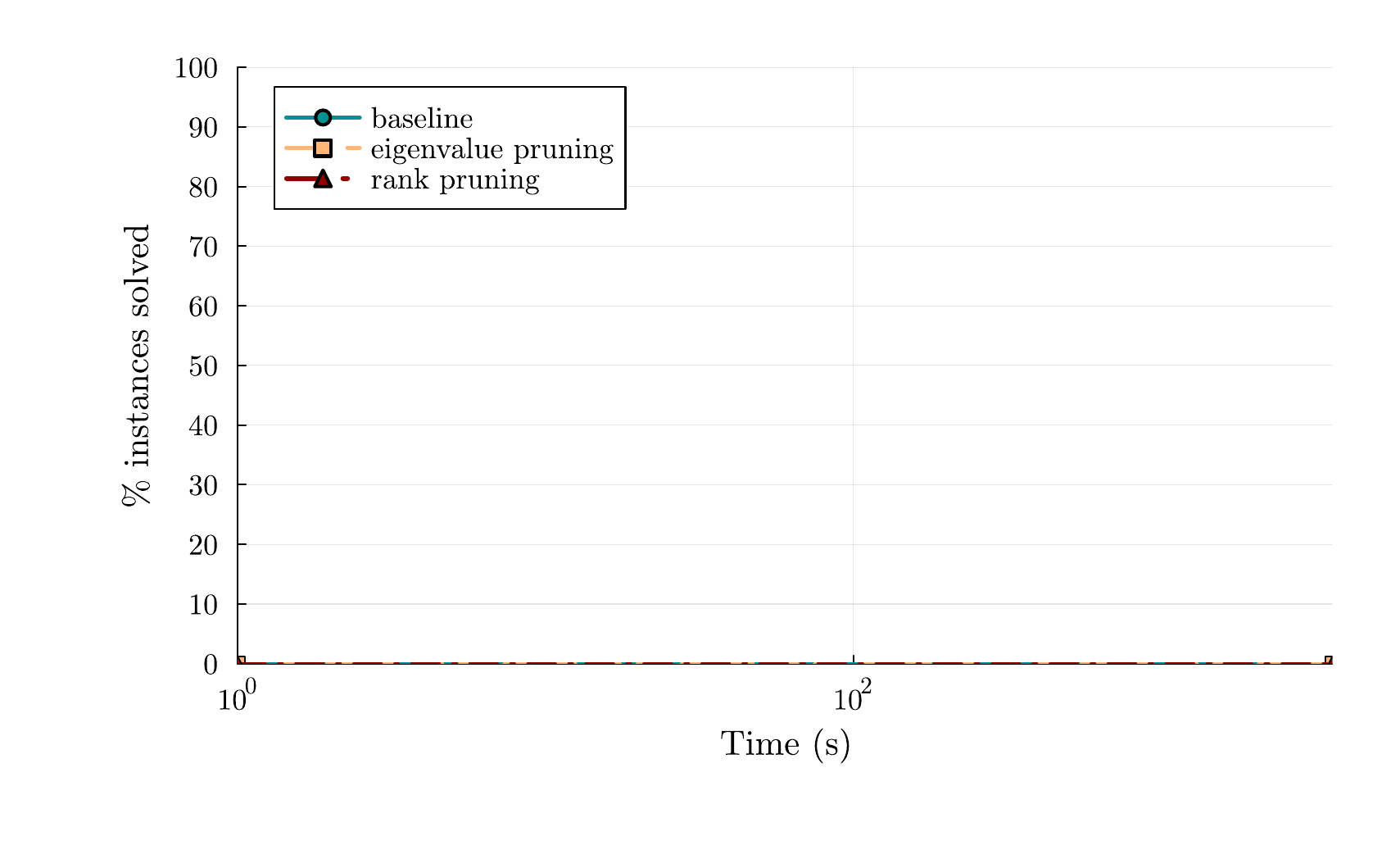}
        \caption{Percentage of instances solved over time for the different pruning techniques compared to the baseline for EOD with independent data.}
        \label{fig:E_independent_pruning}
    \end{subfigure}
    \caption{Comparing the pruning strategies for EOD}
\end{figure}

\begin{figure}
    \centering
    \includegraphics[width=1.0\textwidth]{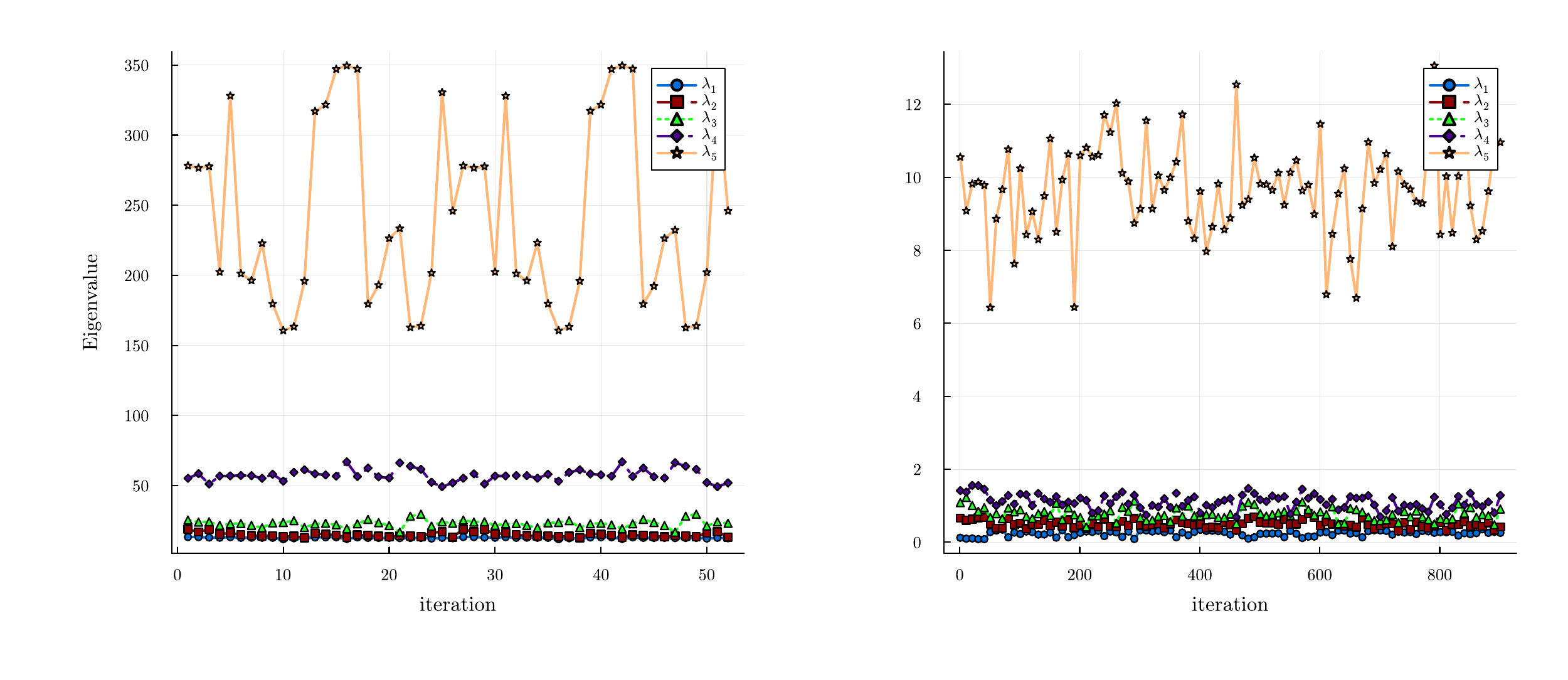}
    \caption{The trajectories of the eigenvalues of the information matrix for an instance of the E-Optimal Design problem with 30 variables.}
    \label{fig:Trajectory_Eigenvalues}
\end{figure}

\begin{figure}
    \centering
    \includegraphics[width=1.0\textwidth]{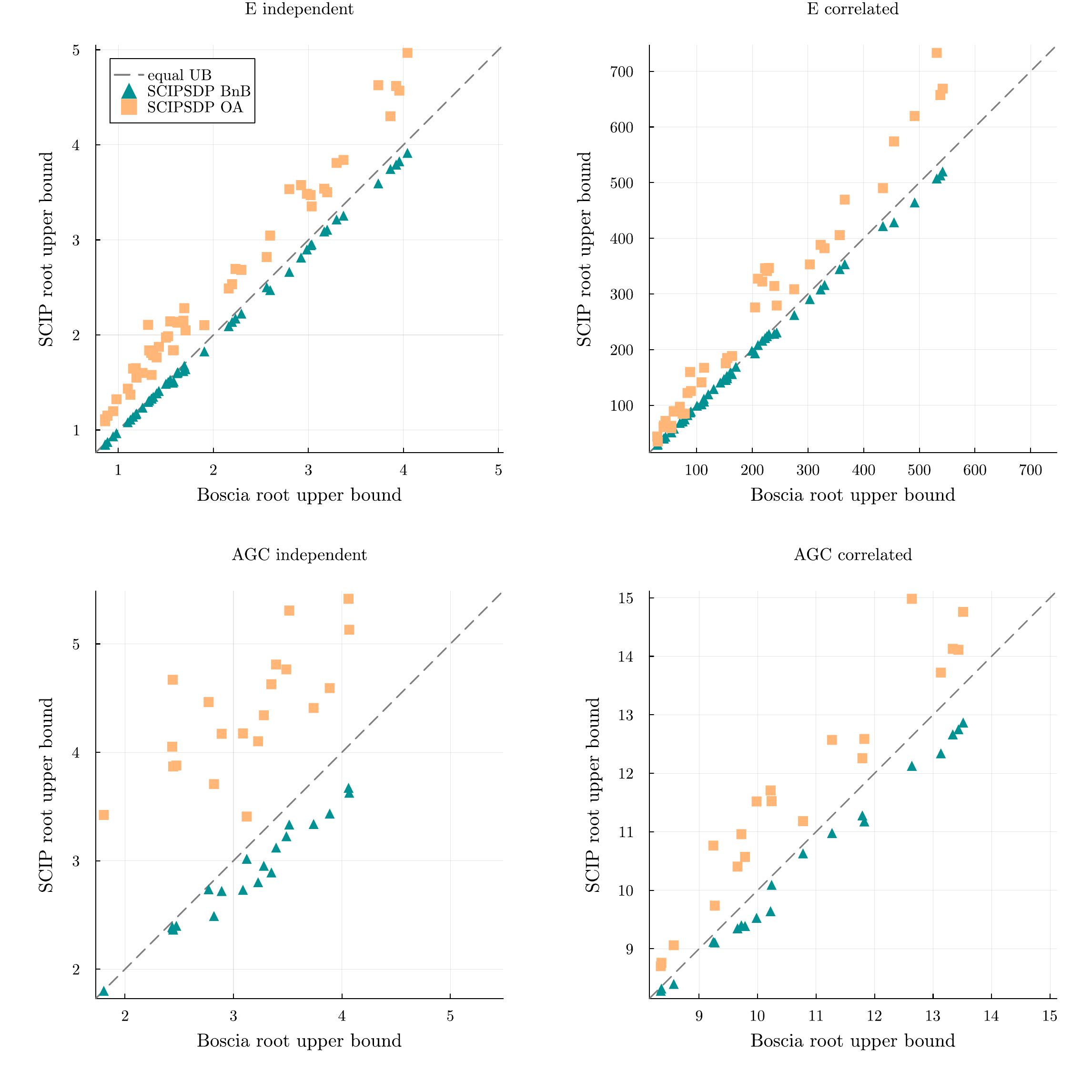}
    \caption{Upper bound comparison at the root node for EOD and AGC. A point below the line indicates an upper bound better than Boscia's.}
    \label{fig:UpperboundsRoot}
\end{figure}

\end{document}